\documentclass[10pt,a4paper,reqno]{amsart} 

\usepackage{amsmath}
\usepackage{mathtools}
\usepackage{amssymb}
\usepackage{graphicx}
\usepackage{color}
\usepackage{latexsym}
\usepackage[T1]{fontenc}

\newcommand{\ns}{{\mathbb N}} 
\newcommand{\zs}{{\mathbb Z}} 
\newcommand{\qs}{{\mathbb Q}}  
\newcommand{\cs}{{\mathbb C}} 

\newcommand{\al}{\alpha}
\newcommand{\be}{\beta}

\newcommand{\eps}{\epsilon}

\newcommand{\bx}{\bar x}
\newcommand{\by}{\bar y}

\newcommand{\GK}{\mathbb{K}}

\DeclareMathOperator{\Pol}{Pol}
\DeclareMathOperator{\CT}{CT}
\DeclareMathOperator{\LHS}{LHS}
\DeclareMathOperator{\RHS}{RHS}
\DeclareMathOperator{\dig}{dig}
\DeclareMathOperator{\ddig}{ddig}
\DeclareMathOperator{\df}{df}
\DeclareMathOperator{\dv}{dv}
\DeclareMathOperator{\cc}{c}
\DeclareMathOperator{\vv}{v}
\DeclareMathOperator{\ff}{f}
\DeclareMathOperator{\ee}{e}

\DeclareMathOperator{\Tpol}{T}
\DeclareMathOperator{\Ppol}{P}

\newcommand{\mM}{\mathcal{M}}
\newcommand{\gM}{M}
\newcommand{\gtM}{\tilde{M}}
\newcommand{\mN}{\mathcal{N}}
\newcommand{\gN}{N}

\newcommand{\mR}{\mathcal{R}}
\newcommand{\gR}{R}
\newcommand{\mS}{\mathcal{S}}
\newcommand{\gS}{S}
\newcommand{\mT}{\mathcal{T}}
\newcommand{\gT}{T}

\newcommand{\R}{R}

\newcommand{\mQ}{\mathcal{Q}}
\newcommand{\vQ}{\vec{Q}}
\newcommand{\vR}{\vec{R}}
\newcommand{\vmQ}{\vec{\mQ}}

\newcommand{\gQ}{Q}

\newcommand{\gtQ}{\tilde{Q}}

\newcommand{\titre}[1]{\noindent\textbf{#1} }

\newcommand{\cD}{\mathcal D}
\newcommand{\cI}{\mathcal I}

\newcommand{\parfrac}[2]{\left( \frac{#1}{#2} \right)}

\newtheorem{Theorem}{Theorem}
\newtheorem{Proposition}[Theorem]{Proposition}
\newtheorem{Claim}[Theorem]{Claim}
\newtheorem{Corollary}[Theorem]{Corollary}
\newtheorem{Conjecture}[Theorem]{Conjecture}

\newtheorem{Lemma}[Theorem]{Lemma}

\newcommand{\beq}{\begin{equation}}
\newcommand{\eeq}{\end{equation}}

\newcommand{\gf}{generating function}
\newcommand{\gfs}{generating functions}
\newcommand{\fps}{formal power series}

\def\emm#1,{{\em #1}}

\newcommand{\J}{J}
\newcommand{\JI}{H}
\newcommand{\ji}{h}

\catcode`\@=11
\def\section{\@startsection{section}{1}%
 \z@{.7\linespacing\@plus\linespacing}{.5\linespacing}%
 {\normalfont\bfseries\scshape\centering}}

\def\subsection{\@startsection{subsection}{2}%
  \z@{.5\linespacing\@plus\linespacing}{.5\linespacing}%
  {\normalfont\bfseries\scshape}}

\def\subsubsection{\@startsection{subsubsection}{3}%
 \z@{.5\linespacing\@plus\linespacing}{-.5em}
  {\normalfont\bfseries\itshape}}
\catcode`\@=12

\def\qed{$\hfill{\vrule height 3pt width 5pt depth 2pt}$}

\begin{document}
\title
[Counting colored planar maps: algebraicity results]
{Counting colored planar maps: algebraicity results}

\author[O. Bernardi]{Olivier Bernardi}
\address{O. Bernardi: CNRS, Laboratoire de Math\'ematiques,
B\^at. 425, Universit\'e Paris-Sud,
91405 Orsay Cedex, France}
\email{olivier.bernardi@math.u-psud.fr}

\author[M. Bousquet-M\'elou]{Mireille Bousquet-M\'elou}
\address{M. Bousquet-M\'elou: CNRS, LaBRI, Universit\'e Bordeaux 1, 
351 cours de la Lib\'eration, 33405 Talence, France}
\email{mireille.bousquet@labri.fr}

\thanks{Both authors were supported by  the French ``Agence Nationale
de la Recherche'', projects SADA ANR-05-BLAN-0372 and A3 ANR-08-BLAN-0190.}

\keywords{Enumeration -- Colored planar maps -- Tutte polynomial --
  Algebraic generating functions}
\subjclass[2000]{05A15, 05C30, 05C31}

\begin{abstract}
We address the enumeration of properly $q$-colored planar maps, or
more precisely, the enumeration of rooted planar maps $M$ weighted by their
chromatic polynomial $\chi_M(q)$ and counted by the number of
vertices and faces. We prove that the associated \gf\ is algebraic 
when $q\not=0,4$ is of the form $2+2\cos (j\pi/m)$, for integers $j$ and $m$.
This includes the two integer values $q=2$ and $q=3$. We extend this
to planar maps weighted by their Potts polynomial $\Ppol_M(q,\nu)$, which
counts all $q$-colorings  (proper or not) by the number of
monochromatic edges. We then 
prove similar results for planar triangulations, thus generalizing
some results of Tutte which dealt with their proper $q$-colorings. In statistical
physics terms, the problem we study consists in solving the Potts model on random planar
lattices. From  a technical viewpoint, this means solving non-linear
equations with two ``catalytic'' variables. To our knowledge, this is the
first time such equations are being solved since Tutte's remarkable
solution of  properly $q$-colored triangulations. 
\end{abstract}

\date{18 November 2010}
\maketitle


\section{Introduction}
In 1973, Tutte began his enumerative study of colored triangulations
by publishing the following functional equation~\cite[Eq.~(13)]{lambda12}: 
\begin{multline}\label{eq-Tutte}
\gT(x,y)=xy^2q(q-1)+\frac{xz}{yq}\gT(1,y)\gT(x,y)
+xz\frac{\gT(x,y)-y^2\gT_2(x)}{y}-x^2yz\frac{\gT(x,y)-\gT(1,y)}{x-1},
\end{multline}
where $T_2(x)$ stands for $\frac 1 2 \frac{\partial ^2T}{\partial
  y^2}(x,0)$
(in other words, $T_2(x)$ is the coefficient of $y^2$ in $T(x,y)$).
This equation defines a  unique \fps\ 
in $z$, denoted $\gT(x,y)$, which has 
polynomial coefficients in $q$, $x$ and $y$.
 Tutte's interest in this series relied on the fact
that it ``contains'' the \gf\ of properly $q$-colored 
triangulations of the sphere 
(Figure~\ref{fig:triang-coloree}).
More precisely,  the coefficient of
$y^2$ in $\gT(1,y)$ is
$$
\gT_2(1)= [y^2]\gT(1,y)
= q(q-1)+\sum_{T}z^{\ff(T)}{\chi_T(q)},
$$
where the sum runs over all rooted triangulations of the sphere,
$\chi_T$ is the chromatic polynomial of $T$, and 
$\ff(T)$  the number of faces of $T$.

\begin{figure}[ht!]\begin{center} \input{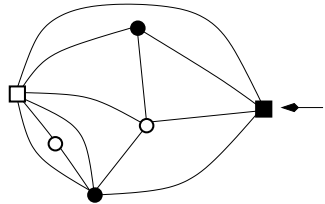}\caption{A (rooted) triangulation of the sphere, properly colored with 4 colors.}\label{fig:triang-coloree} \end{center}\end{figure}

In the ten years that followed, Tutte devoted at least nine papers
to the study of this
equation~\cite{lambda12,lambda3,lambda-tau,tutteIV,tutteV,tutte-pair,tutte-chromatic-sols,tutte-chromatic-solsII,tutte-differential}. His
work culminated in 1982, when he proved that the series $T_2(1)$
counting $q$-colored triangulations
satisfies a  non-linear differential
equation~\cite{tutte-chromatic-solsII,tutte-differential}. More
precisely, with $t=z^2$ and $H\equiv H(t)= t^2T_2(1)$,
\beq\label{Tutte-ED}
2q^2(1-q)t +(qt+10H-6tH')H''+q(4-q)(20H-18tH'+9t^2H'')=0.
\eeq

This \emm tour de force, has remained isolated since then, and it is
our objective to reach a better understanding of Tutte's rather
formidable approach, 
and to apply it to other problems in the enumeration of colored planar
maps. We recall here that a planar map is a connected planar graph
properly embedded in the sphere. More  definitions will be
given later.

\medskip

We focus in this paper on two main families of  maps: general planar
maps, and planar triangulations. We generalize Tutte's problem by counting
\emm all, colorings of these maps (proper and non-proper), assigning a
weight $\nu$ to every monochromatic edge. Thus a  typical series we 
consider is
\beq\label{M-def-simple}
\sum_{M} t^{\ee(M)}w^{\vv(M)}z^{\ff(M)}{\Ppol_M(q,\nu)},
\eeq
where $M$ runs over a given set of planar maps (general maps or
triangulations, for instance), $\ee(M)$, $\vv(M)$, $\ff(M)$
  respectively denote the number of edges, vertices and faces of $M$, and
$$
\Ppol_M(q,\nu)=\sum_{c: V(M) \rightarrow \{1, \ldots, q\}} \nu^{m(c)}
$$
counts all colorings of the vertices of $M$ in $q$ colors, weighted by
the number $m(c)$ of monochromatic edges.
As explained in Section~\ref{sec:def}, $\Ppol_M(q,\nu)$ is actually a
\emm polynomial, in $q$ and $\nu$ called, in statistical physics,
the \emm partition  function of the Potts model, on $M$. Up to a change of
variables, it coincides with the so-called \emm Tutte polynomial, of $M$.
Note that $\Ppol_M(q,0)$ is the chromatic polynomial of $M$.

In this paper, we  climb Tutte's scaffolding halfway. Indeed, one
key step in his solution of~\eqref{eq-Tutte} is to prove that, when
$q\not = 0,4$ is of the form  
\beq\label{q-form}
q=2+2\cos \frac{j\pi}m,
\eeq
for  integers $j$ and $m$, 
then $\gT_2(1)$ is an \emm
algebraic series,, that is, satisfies a polynomial
equation\footnote{Strictly speaking,
Tutte only proved this for certain values of $j$ and $m$. Odlyzko and Richmond~\cite{odlyzko-richmond}
 proved later that his work implies   algebraicity
for $j=2$. 
We prove here that it holds for all $j$ and $m$, except
those that yield the extreme values $q=0,4$.  For $q=0$ the
polynomials $\Ppol_M(q,\nu)$ vanish, but we actually weight our maps
by $\Ppol_M(q,\nu)/q$, which gives sense to the restriction $q\not = 0$.}
$$
P_q(\gT_2(1),z)=0
$$
for some polynomial $P_q$ that depends on $q$.  Numbers of the
form~\eqref{q-form} 
generalize  Beraha's numbers (obtained for $j=2$), 
which occur frequently in connection with  chromatic properties
 of planar graphs~\cite{beraha-kahane-weiss,fendley-chromatic,jacobsen-richard-salas,jacobsen-salas,martin,saleur}. Our main result is that \emph{the series defined by~\eqref{M-def-simple}  
is also algebraic for these values of $q$}, whether the sum runs over all planar
maps (Theorem~\ref{thm:alg-planaires}), non-separable planar maps
(Corollary~\ref{coro:non-sep}), or planar triangulations
(Theorem~\ref{thm:alg-triang}). 
These series are \emm not, algebraic for a generic value of $q$. In a forthcoming paper, we will 
establish the counterpart of~\eqref{Tutte-ED}, in the form of (a system of)
differential equations for these series, valid for all $q$. 

\medskip
Hence this paper generalizes in two directions the series of papers
devoted by Tutte to~\eqref{eq-Tutte}, which he then revisited in his
1995 survey~\cite{tutte-chromatic-revisited}: firstly, because we include
  non-proper colorings, and secondly, because we study
  two classes of  planar maps (general/triangulations), the
 second being more  complicated than the first. 
We provide in Sections~\ref{sec:two}
and~\ref{sec:three} explicit 
results (and a conjecture) for families of 2- and 3-colored maps. Some of them
have an attractive form, and should stimulate the research of
alternative proofs based on trees, in the spirit of what has been done
in the past 15 years for uncolored maps~(see for instance \cite{Sch97,BDG-planaires,mbm-schaeffer-ising,bouttier-mobiles,chassaing-schaeffer,fusy-dissections,chapuy,bernardi-fusy}). Finally, our results
constitute a springboard for the general solution (for a generic value
of $q$), in preparation.

The functional equations we start from are established in
Section~\ref{sec:eq-func} (Propositions~\ref{prop:eq-M}
and~\ref{prop:eq-Q}). As~\eqref{eq-Tutte}, they involve two \emm
catalytic, variables $x$ and $y$. Much progress has been made in the
past few years on the solution of \emm linear, equations of this
type~\cite{bousquet-motifs,Bous05,Mishna-Rechni,mbm-mishna}, but those that govern the enumeration of colored maps
are non-linear. In fact, Equation~\eqref{eq-Tutte} is so far, to our knowledge,  the only
instance of such an equation that has ever been solved. Our main two algebraicity
results are stated in Theorems~\ref{thm:alg-planaires}
and~\ref{thm:alg-triang}. 
In  Section~\ref{sec:example} below,
we describe on  a simple example (2-colored planar maps)
the steps that yield from an equation to an algebraicity theorem. 
It is then easier to give a more detailed outline of the
paper (Section~\ref{sec:outline}). 
Roughly speaking, the general idea  is to construct, for values of $q$
of the form~\eqref{q-form}, an equation with only \emm one, catalytic
variable satisfied by a relevant specialization of the main series
(like $T(1,y)$ in the problem studied by Tutte). For instance, we
 derive in Section~\ref{sec:example} the simple
 equation~\eqref{M-eq-y-20} from the more complicated
 one~\eqref{M-eq-xy-20}. 
One then applies a general algebraicity theorem (Section~\ref{sec:alg}), according to which
solutions of such equations are always algebraic.

\medskip
Most calculations were done using Maple: several Maple sessions
accompanying this paper are available on the second author's web page
(next to this paper in the publication list).

\medskip
To conclude this introduction, let us mention that the problems we
study here have also attracted attention in theoretical physics, and
more precisely in the study of models for 2-dimensional \emm quantum
gravity,. 
In particular, our results on triangulations share at least a common
flavour with a paper  by Bonnet and Eynard~\cite{eynard-bonnet-potts}.
Let us briefly describe their approach.
The solution of the Potts model on triangulations can be expressed
fairly easily in terms of a matrix integral. Starting from this
formulation, Daul and then
Zinn-Justin~\cite{daul,zinn-justin-dilute-potts} used a saddle point
approach to obtain certain \emm critical exponents,. 
Bonnet and Eynard went further using the \emph{equation of
  motion} method~\cite{eynard-bonnet-potts}. First, they derived from
the integral formulation 
a (pair of) polynomial equations with two catalytic variables
(the so-called \emph{loop-equations})\footnote{These equations
  differ from the functional equation~\eqref{eq:Q} we establish
  for the same problem. But they 
are  of a similar nature, and we actually believe that our method
   applies to them as well.}.
 From there, they postulated the
existence of a change of variables which 
transforms the loop-equations into an equation 
occurring in another classical model, the $O(n)$ model.
The results
of~\cite{Eynard-Kristjansen:On-model1,Eynard-Kristjansen:On-model2,Eynard-ZinnJustin:On-model}
on  the $O(n)$ model
then translate into results on the Potts model. In particular, when
the parameter $q$ of the Potts model is of the form $q=2+2\cos
\frac{j\pi}m$, Bonnet and Eynard obtain an equation  with one catalytic variable
\cite[Eq.~5.4]{eynard-bonnet-potts} which may correspond to our
equation~\eqref{eq-inv-triang}.

\section{A glimpse at our approach: properly 2-colored planar maps}
\label{sec:example}
The aim of this paper is to prove that, for certain values of $q$ (the
number of  colors), the \gf\ of $q$-colored planar maps, and of $q$-colored
triangulations, is algebraic. 
Our starting point will be the functional equations of
Propositions~\ref{prop:eq-M} and~\ref{prop:eq-Q}. In order to
illustrate our approach, we treat here  the case
of  properly $2$-colored planar maps counted by edges. 
It will follow from Proposition~\ref{prop:eq-M} that this means
solving the following equation: 
\begin{multline}
  \label{M-eq-xy-20}
M ( x,y ) =
1+xyt ( y+1 ) M ( x,y ) M ( 1,y ) -xytM ( x,y ) M ( x,1 ) 
\\-{\frac {txy
 ( xM ( x,y ) -M ( 1,y )  ) }{x-1}}+{
\frac {txy ( yM ( x,y ) -M ( x,1 ) 
 ) }{y-1}}.
\end{multline}
Here,
$$
M(x,y):= \frac 1 2 \sum_M t^{\ee(M)} x^{\dv(M)} y^{\df(M)} \chi_M(2)
$$
counts planar maps $M$, weighted by their chromatic polynomial
$\chi_M(q)$ at $q=2$,  by the number $\ee(M)$ of edges and
by the degrees
$\dv(M)$ and $\df(M)$ of the root-vertex and root-face (the precise
definitions of these statistics are not important for the moment). We are especially
interested in the specialization
$$
M(1,1)= \frac 1 2 \sum_M t^{\ee(M)} \chi_M(2).
$$
However, there is no obvious way to derive from~\eqref{M-eq-xy-20} an
equation for $M(1,1)$, or even for $M(x,1)$ or $M(1,y)$.
Still,~\eqref{M-eq-xy-20} allows us to determine, by
induction on $n$, the coefficient of $t^n$ in $M(x,y)$.
The variables
$x$ and $y$ are said to be \emm catalytic,. 

We can  see some readers frowning: there is a much simpler way
to approach this enumeration problem! Indeed, a planar map has a proper
2-coloring if and only if it is bipartite, and
every bipartite map admits exactly two proper
2-colorings. Thus $M(1,1)$ is simply  the \gf\ 
of  bipartite planar
maps, counted by edges.
But one has known for  decades how to find this series: a
recursive description of bipartite maps based on the deletion of the
root-edge easily gives:
\beq\label{M-eq-y-20}
M(y)=1+ ty^2 M(y)^2 +ty^2\, \frac{M(y)-M(1)}{y^2-1}
\eeq
where $M(y)\equiv M(1,y)$. This equation has only \emm one, catalytic variable, namely $y$, 
and can be solved using the quadratic
method~\cite[Section~2.9]{goulden-jackson}. In particular, $M(1)\equiv
M(1,1)$ is found to be algebraic:  
$$
M(1,1)=\frac {(1-8t)^{3/2}-1+12t+8\,{t}^{2}}
{32 {t}^{2}}.
$$

What our method precisely does is to \emm reduce  the number of catalytic
variables from two to one,: once this is done, a general algebraicity theorem
(Section~\ref{sec:alg}),
which states that all series satisfying a (proper) equation with one catalytic
variable are algebraic, allows us to conclude.
In  the above example, our approach derives the simple
equation~\eqref{M-eq-y-20} from the 
more difficult equation~\eqref{M-eq-xy-20}. We now detail the 
steps of this derivation.

\subsection{The kernel of the equation, and its roots}
\label{sec:kernel-20}
The functional equation~\eqref{M-eq-xy-20} is linear in $M(x,y)$
(though not globally in $M$, because of quadratic terms like
$M(x,y)M(1,y)$). 
It reads
\beq\label{eq-func-ex}
K(x,y) M(x,y)= R(x,y),
\eeq
where the \emm kernel, $K(x,y)$ is
$$
K(x,y)=1+{\frac {{x}^{2}yt}{x-1}}-{\frac {x{y}^{2}t}{y-1}}
 +xyt M ( x,1 ) -xyt \left( y+1\right) M ( 1,y ) , 
$$
and the right-hand side $R(x,y)$ is:
$$
R(x,y)=1+{\frac {xyt M ( 1,y ) }{x-1}}-{\frac {xytM ( x,1) }{y-1}}.
$$
 Following the principles of the \emm kernel
method,~\cite{hexacephale,banderier-flajolet,bousquet-petkovsek-1,prodinger},
we are  interested in the existence of series $Y\equiv Y(t;x)$ that 
 cancel the kernel. We seek solutions $Y$ in the space of 
 formal power series in $t$ with coefficients in $\qs(x)$
(the field of fractions in $x$). 
The equation $K(x,Y)=0$ can be rewritten
$$
Y-1=
tY\left( 
{ {x{Y}}}-{\frac {{x}^{2}(Y-1)}{x-1}}
 -x (Y-1) M (x,1 ) +x(Y^2-1)  M ( 1,Y )
\right) .
$$
This shows that there exists a unique power series solution $Y(t;x)$ (the
coefficient of $t^n$ in $Y$ can be determined by induction on $n$, 
once the expansion of $M(x,y)$ is known at order ${n-1}$).
However, 
the term having denominator $x-1$ suggests that we will find more
solutions if we  set $x=1+st$, 
with $s$ an indeterminate, and  look for
$Y(t;s)$ in the space of formal power series in $t$ with coefficients in
$\qs(s)$. Indeed, the equation $K(x,Y)=0$ now reads (with $x=1+st$):
$$
 ( Y-1 )(1+Y/s)
=tY\Big(
x {Y}- ( 1+x ) ( Y-1 ) 
-x   ( Y-1 ) M (x,1 ) 
+ x ( Y^2-1 )  M ( 1,Y )  
\Big),
$$
which shows that there exist \emm two series, $Y_1(t;s)$ and
$Y_2(t;s)$ that cancel the kernel for this choice of $x$. One of them
has constant term 1, the other has constant term $-s$. Again, the
coefficient of $t^n$ can be determined inductively. Here are the
first few terms of $Y_1$ and $Y_2$:
\begin{eqnarray*}
 Y_1&=&
1+{\frac {s}{1+s}}t+{\frac {{s}^{2} ( 1+3\,s+{s}^{2} ) }{
 ( 1+s ) ^{3}}}{t}^{2}+O ( {t}^{3} ) \\
Y_2&=&-s+{
\frac {{s}^{2} ( 2+2\,s+{s}^{2} ) }{1+s}}t-{\frac {{s}^{2}
 ( -1+7\,{s}^{2}+17\,{s}^{3}+15\,{s}^{4}+6\,{s}^{5}+{s}^{6}
 ) }{ ( 1+s ) ^{3}}}{t}^{2}+O ( {t}^{3} ). 
\end{eqnarray*}
Replacing $y$ by $Y_i$ in the
 functional equation~\eqref{eq-func-ex} gives  $R(x,Y_i)=0$. We thus
 have four   equations,
 \beq\label{syst1}
K(x,Y_1)=R(x,Y_1)=K(x,Y_2)=R(x,Y_2)=0,
\eeq
that relate $Y_1$, $Y_2$,
 $M(1,Y_1)$, $M(1,Y_2)$, $x$ and $M(x,1)$.

\subsection{Invariants}
\label{sec:inv-20}
We  now eliminate from the system~\eqref{syst1} the series $M(x,1)$
and the indeterminate $x$
to obtain two equations relating $Y_1$, $Y_2$,  $M(1,Y_1)$ and
$M(1,Y_2)$. 
 This elimination is performed in Section~\ref{sec:inv-maps} for a
 general value of $q$. So let us
just give the pair of equations we obtain. The first one is:
$$
2t\,Y_1\,M ( 1,Y_1 ) -2t\,Y_2\,M ( 1,Y_2) =
-{\frac { (1 -Y_1-Y_2+(1-t)Y_1\,Y_2 )  ( Y_1-Y_2 ) }{
Y_1\,Y_2 ( Y_1-1 ) ( Y_2-1 )  }}
$$
or equivalently,
$$
I(Y_1)=I(Y_2),
$$
with
\beq
  I(y)= 2t y M ( 1,y )+ {\frac {y-1}{y}}+{\frac {ty}{y-1}}.
\label{I-def-20}
\eeq
 Following Tutte~\cite{tutteV}, we say that $I(y)$, which
 takes the same value at  $Y_1$ and $Y_2$,  is an \emm invariant,.

Let us denote $\cI=I(Y_1)=I(Y_2)$. The second equation obtained by eliminating $x$ and  $M(x,1)$ from the  system~\eqref{syst1} then reads: 
\beq\label{Y-2nd}
  Y_1^2+Y_2^2
- \left(  \cI^{2}- 2\cI+2 t+2 \right)Y_1^2Y_2^2=0.
\eeq
Define
\beq\label{J-def-20}
J(y)=\left(I(y)^{2}-2\,I(y) +2\,t+2\right) ^{2}-8\by^2
  (I(y)^{2}-2\,I(y)+2\,t+2)+8\,\by^4,
\eeq
where $\by=1/y$.
Then  an elementary calculation shows that the
identity~\eqref{Y-2nd}, combined with 
$I(Y_1)=I(Y_2)=\cI$, implies
$$
J(Y_1)=J(Y_2).
$$
 We have thus obtained a second \emm invariant,\,\footnote{There is no real need to include the term $\left(I(y)^{2}-2\,I(y)
   +2\,t+2\right) ^{2}$ (which is itself an invariant) in $J(y)$. However, we will see later than
   this makes $J(y)$ a Chebyshev polynomial, a convenient property.}.

\subsection{The theorem of invariants}
\label{sec:thm-inv-20}
Consider the  invariants~\eqref{I-def-20} and~\eqref{J-def-20} that we
have constructed. Both are  series in $t$ with
coefficients in $\qs(y)$, the field of rational functions in $y$. 
In  $I(y)$, these coefficients are not singular at
$y=1$, except for the coefficient of $t$, which has a simple pole at
$y=1$. We say that $I(y)$ has \emm valuation, $-1$ in $(y-1)$.
Similarly, $J(y)$ has  valuation $-4$ in $(y-1)$ (because of the term $I(y)^4$). 

Observe that all polynomials in $I(y)$ and $J(y)$ with coefficients
in $\qs((t))$ (the ring of Laurent series in $t$) are  invariants. We
prove in Section~\ref{sec:inv-thm} a theorem --- the Theorem of 
invariants --- that says that there are ``few'' invariants, and that,
in particular, $J(y)$ must be a polynomial in $I(y)$ with  coefficients
in $\qs((t))$. Considering the valuations of $I(y)$ and
$J(y)$ in $(y-1)$ shows that this polynomial has degree 4. That is,
there exist Laurent series $C_0, \ldots, 
C_4$ in $t$, with coefficients in $\qs$, such that
\beq\label{eq-inv-20}
\left(I(y)^{2}-2\,I(y) +2\,t+2\right) ^{2}-8\by^2
  (I(y)^{2}-2\,I(y)+2\,t+2)+8\,\by^4=
\sum _{r=0}^{4} C_r\, I(y)^r.
\eeq

\subsection{An equation with one catalytic variable}
\label{sec:one-20}
%
In~\eqref{eq-inv-20}, replace $I(y)$ by its
expression~\eqref{I-def-20} in terms of $M(y)\equiv 
M(1,y)$. The resulting equation  involves $M(y)$, $t$, $y$, and  five
unknown series $C_r \in \qs((t))$.   The variable $x$ has
disappeared. 
Let us now write $M(y)=M(1)+ (y-1)M'(1)+ \cdots$, and
expand the equation in the neighborhood of $y=1$. 
This gives the values of  the  series $C_r$:
$$
  C_4=1, \quad 
C_3=-4,\quad
C_2=4t,\quad
C_1=8(1+t)
$$
and
$$
C_0=
-4-40\,t-4\,{t}^{2}+32\,{t}^{2}M (1).
$$
Let us  replace in~\eqref{eq-inv-20} each series $C_r$ by its
expression: we obtain 
$$
{y}^{2}t ( y^2-1 )  M ( y )^{2}
+ (1 -{y}^{2}+{y}^{2}t ) M ( y ) 
-t {y}^{2}M ( 1 )+  y^2-1 =0,
$$
which is exactly the equation with one catalytic variable~\eqref{M-eq-y-20}
obtained by deleting recursively the root-edge in bipartite planar
maps.
It can now be solved using the quadratic
method~\cite[Section~2.9]{goulden-jackson} or its extension (which works for equations of a higher degree in
$M(y)$) described in~\cite{mbm-jehanne} and generalized further in
Section~\ref{sec:alg}.

\subsection{Detailed outline of the paper}
\label{sec:outline}
%
With this example at hand, it is  easier  to describe the structure
of the paper. We begin with recalling in Section~\ref{sec:def}
standard definitions on maps, power series, and the Tutte (or Potts)
polynomial.  In Section~\ref{sec:eq-func} we establish functional equations
for $q$-colored planar maps and for $q$-colored triangulations. We
then construct a pair $(I(y), J(y)$) of invariants 
 in Sections~\ref{sec:inv-maps} (for planar maps) and~\ref{sec:inv-triang} (for
triangulations). 
The  construction of the invariant $J(y)$ is non-trivial, and relies on
  an independent result which is the topic of 
  Section~\ref{sec:source}. It is at this stage that the condition
  $q=2+2 \cos j\pi/m$ naturally occurs.  We then prove two ``theorems of invariants'',
  one for planar maps and one for triangulations
  (Section~\ref{sec:inv-thm}).  
Applying them  provides  counterparts of~\eqref{eq-inv-20}, where only
\emm one, catalytic variable $y$ is now involved. 
  Unfortunately, the general algebraicity theorem  of~\cite{mbm-jehanne}
  does not apply directly to these equations: we thus 
  extend it slightly (Section~\ref{sec:alg}). In
  Sections~\ref{sec:alg-maps} and~\ref{sec:alg-triang}, 
we prove that this extended theorem indeed applies to the  equations with one
catalytic variable derived from the theorems of invariants; we thus
obtain the main algebraicity results of the paper. 
Explicit results are given for two and three
 colors in Sections~\ref{sec:two} and~\ref{sec:three}. Finally, we explain in
Section~\ref{sec:non-sep} that the algebraicity results obtained for
general planar maps imply similar results for \emm non-separable, planar maps.

\section{Definitions and notation}
\label{sec:def}

\subsection{Planar maps}
%
A \emph{planar map} is a proper
 embedding of a connected planar graph in the
oriented sphere, considered up to orientation preserving
homeomorphism. Loops and multiple edges are allowed. The \emph{faces}
of a map are the connected components of 
its complement. The numbers of
vertices, edges and faces of a planar map $M$, denoted by $\vv(M)$,
$\ee(M)$ and $\ff(M)$,  are related by Euler's relation
$\vv(M)+\ff(M)=\ee(M)+2$.
 The \emph{degree} of a vertex or face is the number
of edges incident to it, counted with multiplicity. A \emph{corner} is
a sector delimited by two consecutive edges around a vertex;
hence  a vertex or face of degree $k$ defines $k$ corners. 
Alternatively, a corner can be described as an incidence between a
vertex and a face. The \emph{dual} of a
map $M$, denoted $M^*$, is the map obtained by placing a 
vertex of $M^*$ in each face of $M$ and an edge of $M^*$ across each
edge of $M$; see Figure~\ref{fig:example-map}. A \emph{triangulation}
is a map in which every face has degree 3. Duality transforms
triangulations into \emph{cubic} maps, that is, maps in which every vertex has
degree 3.

For counting purposes it is convenient to consider \emm rooted, maps. 
A map is rooted by choosing a corner, called  the \emm root-corner,.
The vertex and face that are incident at this corner are respectively
the \emm root-vertex, and the \emm root-face,.
 In figures, we  indicate the rooting by 
 an arrow pointing to the root-corner, and take the root-face
as the infinite face (Figure~\ref{fig:example-map}). 
This explains why we often call the root-face the \emm outer face, and
its degree the \emm outer degree,. 
This way of rooting maps is equivalent to the more standard way, where
an edge, called the \emm root-edge,, is distinguished and
oriented. For instance, one can choose the edge that follows the
root-corner in counterclockwise 
order around the root-vertex, and  orient it away from this vertex.
The reason why we prefer our convention is that it gives a natural
way to root  the dual of a rooted map $M$ in such a way the
root-vertex of $M$ becomes the root-face of $M^*$, and vice-versa: it
suffices to draw the vertex of $M^*$ corresponding to the root-face of
$M$ at the starting point of the arrow that points to the root-corner
of $M$,
and to reverse this arrow, to obtain a canonical rooting of $M^*$ (Figure~\ref{fig:example-map}). In this way, taking the dual of a map
exchanges the degree of the root-vertex and the degree of the
root-face, which will be  convenient for our study.

From now on, every {map}
is \emph{planar} and \emph{rooted}. By convention,
we include among rooted planar maps the \emph{atomic map}
$m_0$ having one vertex and no edge.

\begin{figure}[ht!]\begin{center} \input{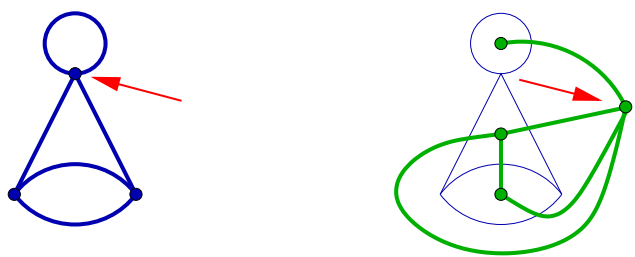}\caption{A rooted planar map and its dual
(rooted at the dual corner).}\label{fig:example-map} \end{center}\end{figure}
%

\subsection{Power series}
Let $A$ be a commutative ring and $x$ an indeterminate. We denote by
$A[x]$ (resp. $A[[x]]$) the ring of polynomials (resp. \fps) in $x$
with coefficients in $A$. If $A$ is a field, then $A(x)$ denotes the field
of rational functions in $x$, and $A((x))$ the field of Laurent series
in $x$. These notations are generalized to polynomials, fractions
and series in several indeterminates. We 
denote by bars the reciprocals of variables: that is, $\bx=1/x$, so that $A[x,\bx]$ is the ring of Laurent
polynomials in $x$ with coefficients in $A$.
The coefficient of $x^n$ in a Laurent  series $F(x)$ is denoted
by $[x^n]F(x)$, and the constant term by $\CT F(x):=[x^0]F(x)$.
 The \emm valuation, of a Laurent series
$F(x)$ is the smallest $d$ such that $x^d$ occurs in $F(x)$ with a
non-zero coefficient. If $F(x)=0$, then the valuation is $+\infty$.
More generally, for a series $F(t;x) =\sum_n F_n(x) t^n 
\in A(x)[[t]]$, and $a \in A$, we
say that $F(t;x)$ has valuation at least $-d$ in $(x-a)$  if no
coefficient $F_n(x)$ has a pole of order larger than $d$ at $x=a$.

Recall that a power series $F(x_1, \ldots, x_k) \in \GK[[x_1, \ldots, x_k]]$, where $\GK$ is a
field, is \emm algebraic , (over $\GK(x_1, \ldots, x_k)$) if it satisfies a
non-trivial polynomial equation $P(x_1, \ldots, x_k, F(x_1, \ldots,
x_k))=0$.

\subsection{The Potts model and the Tutte polynomial}
Let $G$ be a graph with vertex set $V(G)$ and edge set $E(G)$.
Let $\nu$ be an indeterminate, and take $q\in \ns$. 
A \emm coloring, of the vertices of $G$ in $q$ colors is a map $c : V(G)
\rightarrow \{1, \ldots, q\}$. An edge of $G$ is \emm monochromatic,
if its endpoints share the same color. Every loop is thus
monochromatic. The number of monochromatic edges is denoted by $m(c)$.
The \emm partition function of the  Potts model,
on $G$ counts colorings by the number of monochromatic edges:
$$
\Ppol_G(q, \nu)= \sum_{c  : V(G)\rightarrow \{1, \ldots, q\}}
\nu^{m(c)}.
$$
The Potts model is a classical magnetism model in statistical physics, which
includes (when $q=2$) the famous Ising model (with no magnetic
field)~\cite{welsh-merino}. Of course, $\Ppol_G(q,0)$ is the chromatic
polynomial of $G$. 

If $G_1$ and $G_2$ are disjoint graphs and $G=G_1 \cup G_2$, then clearly
\beq\label{Potts-disjoint}
\Ppol_{G}(q,\nu)=\Ppol_{G_1}(q,\nu)\Ppol_{G_2}(q,\nu).
\eeq
If $G$ is obtained by attaching $G_1$ and $G_2$ at one vertex, then
\beq\label{eq:Potts-1components}
\Ppol_{G}(q,\nu)=\frac 1 q \, \Ppol_{G_1}(q,\nu)\Ppol_{G_2}(q,\nu).
\eeq
The Potts partition function can be computed by induction on the
number of edges. If $G$ has no edge, then $\Ppol_G(q,\nu)=
q^{|V(G)|}$. Otherwise, let $e$ be an edge of $G$. Denote by $G\backslash
  e$ the graph obtained by deleting $e$, and by ${G\slash e}$ the graph obtained by contracting $e$ (if $e$ is a loop, then it is simply deleted). Then
\beq\label{Potts-induction}
\Ppol_{G}(q,\nu)=  \Ppol_{G\backslash e}(q,\nu)+(\nu-1) \Ppol_{G\slash e}(q,\nu).
\eeq
Indeed, it is not hard to see that $\nu\Ppol_{G\slash e}(q,\nu)$
counts colorings for which $e$ is monochromatic, while $\Ppol_{G\backslash
  e}(q,\nu)- \Ppol_{G\slash e}(q,\nu)$ counts those for which $e$ is
bichromatic.
One important consequence of this induction is that 
  $\Ppol_{G}(q,\nu)$ is
always a \emm polynomial, in $q$ and $\nu$. From now on, we call
it the \emm Potts polynomial, of $G$. We will often consider $q$
as an indeterminate, or evaluate  $\Ppol_{G}(q,\nu)$ at 
real values $q$. We also observe that $\Ppol_{G}(q,\nu)$ is a multiple
of $q$: this explains why we will  weight maps by $\Ppol_{G}(q,\nu)/q$.

Up to a change of variables, the Potts polynomial is equivalent to
another, maybe better known, invariant of graphs: the \emm Tutte polynomial,
$\Tpol_G(\mu, \nu)$ (see e.g. \cite{Bollobas:Tutte-poly}):
$$
\Tpol_G(\mu,\nu):=\sum_{S\subseteq
  E(G)}(\mu-1)^{\cc(S)-\cc(G)}(\nu-1)^{\ee(S)+\cc(S)-\vv(G)},
$$
where the sum is over all spanning subgraphs of $G$ (equivalently,
over all subsets of edges) and  $\vv(.)$, $\ee(.)$ and $\cc(.)$ denote
respectively the number of vertices, edges and connected
components. For instance, the Tutte polynomial of a graph with no edge
is 1. The equivalence with the Potts polynomial was established  by
Fortuin and Kasteleyn~\cite{Fortuin:Tutte=Potts}: 
\begin{eqnarray}\label{eq:Tutte=Potts}
\Ppol_G(q,\nu)~=~\sum_{S\subseteq E(G)}q^{\cc(S)}(\nu-1)^{\ee(S)}~=~(\mu-1)^{\cc(G)}(\nu-1)^{\vv(G)}\,\Tpol_G(\mu,\nu),
\end{eqnarray}
for $q=(\mu-1)(\nu-1)$. 
In this paper, we work with $\Ppol_G$ rather than $\Tpol_G$ because
we wish to assign real values to $q$ (this is more natural than
assigning real values to $(\mu-1)(\nu-1)$). However, we will use one
property that looks more natural in terms of $\Tpol_G$: 
if $G$ and $G^*$ are dual 
connected planar graphs  (that is, if $G$ and $G^*$ can be embedded
as dual planar maps) then  
\beq\label{eq:duality-Tutte-poly}
\Tpol_{G^*}(\mu,\nu)=\Tpol_G(\nu,\mu).
\eeq
Translating this identity in terms of   Potts polynomials thanks
to~\eqref{eq:Tutte=Potts} gives:  
\begin{eqnarray}
\Ppol_{G^*}(q,\nu)&=&q(\nu-1)^{\vv(G^*)-1}\Tpol_{G^*}(\mu,\nu)\nonumber\\
&=&q(\nu-1)^{\vv(G^*)-1}\Tpol_{G}(\nu,\mu)\nonumber\\
&=&\frac{(\nu-1)^{\ee(G)}}{q^{\vv(G)-1}}\Ppol_{G}(q,\mu),\label{eq:duality-Potts-poly}
\end{eqnarray}
where $\mu=1+q/(\nu-1)$ and the last equality uses Euler's relation:
$\vv(G)+\vv(G^*)-2=\ee(G)$.

\section{Functional equations}
\label{sec:eq-func}
We  now establish functional equations for the generating
functions of two families of colored planar maps: general planar maps, and
triangulations. We begin  with general planar maps, for which Tutte
already did most of the work. However, he did not attempt, or did not
succeed, to solve the equation he had established. 

\subsection{A functional equation for colored planar  maps}
\label{sec:eq-M}

Let $\mM$ be the set of rooted maps. For a rooted map $M$,
denote by $\dv(M)$ and $\df(M)$
the degrees of the root-vertex and root-face.
We define the \emm Potts \gf,\ of planar maps by:
\beq\label{potts-planar-def}
\gM(x,y)\equiv \gM(q,\nu,t,w,z;x,y)
=\frac 1 q \sum_{M\in\mM}t^{\ee(M)}w^{\vv(M)-1}z^{\ff(M)-1} x^{\dv(M)}y^{\df(M)}
\Ppol_M(q,\nu).
\eeq
Since there is a finite number of maps with a given number of edges, 
and $\Ppol_M(q,\nu)$ is a multiple of $q$,
the generating function $\gM(x,y)$ is a power series in $t$ with
coefficients in $\qs[q,\nu,w,z,x,y]$. 
 \begin{Proposition}\label{prop:eq-M}
  The Potts \gf\ of planar maps satisfies:   
\begin{eqnarray}\label{eq:M}
\gM(x,y)&\!\!=\!\!& 1
+xywt\left((\nu-1)(y-1)+qy\right)\gM(x,y)\gM(1,y)
\nonumber\\
&&+xyzt(x\nu-1)\gM(x,y)\gM(x,1)\\
&&
+xywt(\nu-1)\frac{x\gM(x,y)-\gM(1,y)}{x-1}+xyzt\frac{y\gM(x,y)-\gM(x,1)}{y-1}.
\nonumber
\end{eqnarray}
\end{Proposition}
Observe that~\eqref{eq:M} characterizes  $\gM(x,y)$ entirely
as a series in $\qs[q,\nu,w,z,x,y][[t]]$ (think of extracting
recursively the coefficient of $t^n$ in this equation).  
Note also
that if $\nu=1$, then $\Ppol_M(q,\nu) = q^{\vv(M)}$, so that we are
essentially counting planar maps by edges, vertices and faces, and by
the root-degrees $\dv$ and $\df$.  The variable $x$ is no longer
catalytic: it can be set to 1 in the functional equation, which
becomes an equation for $M(1,y)$ with only one catalytic variable~$y$.
 \begin{proof}
    In~\cite{Tutte:dichromatic-sums}, Tutte considered the closely
    related generating function 
$$
\gtM(x,y)\equiv\gtM(\mu,\nu,w,z; x,y)=
\sum_{M\in\mM}w^{\vv(M)-1}z^{\ff(M)-1}x^{\dv(M)}y^{\df(M)} \Tpol_M(\mu,\nu),
$$
which counts  maps weighted by their Tutte polynomial. He established the
following functional equation:
\begin{multline}\label{eq:tM}
  \gtM(x,y)= 1+xyw(y\mu-1)\gtM(x,y)\gtM(1,y)~+~xyz(x\nu-1)\gtM(x,y)\gtM(x,1) \\
+xyw\parfrac{x\gtM(x,y)-\gtM(1,y)}{x-1}+xyz\parfrac{y\gtM(x,y)-\gtM(x,1)}{y-1}.
\end{multline}
Now, the relation~\eqref{eq:Tutte=Potts} between the Tutte and Potts
polynomials and Euler's relation ($\vv(M)+\ff(M)-2=\ee(M)$) give
\beq\label{Potts-Tutte-gfs}
\gM(q,\nu,t,w,z;x,y)=\gtM\left(1+\frac{q}{\nu-1},\nu,(\nu-1)tw,tz;x,y\right),
\eeq
from which~\eqref{eq:M} easily follows.
 \end{proof}

\subsection{A functional equation for colored triangulations}
\label{sec:eq-quasi-triang}
%
Tutte obtained~\eqref{eq:tM} via a recursive description of
planar maps involving deletion and contraction of the root-edge. We
would like  to proceed similarly for triangulations, but the
deletion/contraction of the root-edge may change the degrees of the
 faces that are adjacent to the root-edge, so that the resulting
maps may not be triangulations. This  has led us to
consider a larger class of maps.

We call \emm quasi-triangulations,  rooted planar maps such that every internal
face is either a 
\emph{digon} (degree~2) \emm incident to the root-vertex,, or a \emph{triangle} (degree~3). 
The set of quasi-triangulations is denoted by $\mQ$.
It includes the set of \emm near-triangulations,, which we define as
the maps in which all internal faces have degree 3.
For  $Q$ in $\mQ$, we denote by
$\dig(Q)$ and $\ddig(Q)$ respectively the number of internal digons and
the number of internal digons 
that are doubly-incident to the
root-vertex. For instance, the map $Q$ of Figure~\ref{fig:exampleQ}(a) satisfies
$\dig(Q)=3$ and  $\ddig(Q)=1$. 
A map in $\mQ$ is \emph{incidence-marked} by choosing  for each
internal digon one of its incidences
with the root-vertex. An incidence-marked map
is shown in Figure~\ref{fig:exampleQ}(b).  

\begin{figure}[ht!]\begin{center} \input{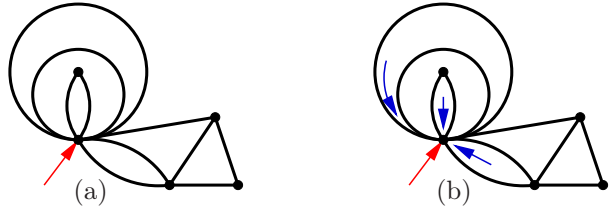}\caption{A map in $\mQ$ and one of the associated incidence-marked maps.}\label{fig:exampleQ} \end{center}\end{figure}

We define the \emm Potts generating function, of 
quasi-triangulations
by
\beq\label{Q-ser-def}
\gQ(x,y)\equiv
\gQ(q,\nu,t,w,z;x,y)=
\frac 1 q \sum_{Q\in\mQ}
t^{\ee(Q)}w^{\vv(Q)-1}z^{\ff(Q)-1}x^{\dig(Q)}y^{\df(Q)}
2^{\ddig(Q)}\Ppol_Q(q,\nu).
\eeq
As before, $\df(Q)$ denotes the degree of the root-face of $Q$.
Observe that a map $Q$ in $\mQ$ gives rise
 to $2^{\ddig(Q)}$ distinct incidence-marked maps. Hence the above
 series can be rewritten as
$$
\gQ(x,y)=\frac 1 q 
\sum_{\vQ\in\vmQ}
t^{\ee(Q)}w^{\vv(Q)-1}z^{\ff(Q)-1}x^{\dig(Q)}y^{\df(Q)} \Ppol_{Q}(q,\nu),
$$
where $\vmQ$ is the set of incidence-marked maps obtained from $\mQ$,
and for $\vQ\in\vmQ$, the underlying (unmarked) map is denoted $Q$.

\begin{Proposition}\label{prop:eq-Q}
The Potts \gf\ of quasi-triangulations, 
defined by~\eqref{Q-ser-def}, satisfies
\begin{multline}
  \label{eq:Q}
Q(x,y) = 1 
+ zt\, \frac{Q(x,y)-1-yQ_1(x)}{y}+  xzt (Q(x,y)-1) + xyzt Q_1(x) Q(x,y)
\\
+yzt(\nu-1)\gQ(x,y)(2x\gQ_1(x)+\gQ_2(x))
+y^2wt\left(q+ \frac{\nu-1}{1-xzt\nu}\right) \gQ(0,y)\gQ(x,y)
\\+\frac{ywt(\nu-1)}{1-xzt\nu} \frac{\gQ(x,y)-\gQ(0,y)}{x}
  \end{multline}
where $\gQ_1(x)=[y]\gQ(x,y)$ and $\displaystyle
\gQ_2(x)=[y^2]\gQ(x,y)=\frac{(1-2xzt\nu)}{zt\nu}\gQ_1(x)$.
\end{Proposition}
As in the case of general maps, Eq.~\eqref{eq:Q} characterizes the
series $Q(x,y)$ entirely 
as a series in $\qs[q,\nu,w,z,x,y][[t]]$ (think of extracting
recursively the coefficient of $t^n$ in this equation). 
Moreover, the variable $x$ is no longer catalytic when $\nu=1$, and
the equation becomes much easier to solve. Finally, 
 Tutte's original equation~\eqref{eq-Tutte} can be derived
from~\eqref{eq:Q}, as we explain in Section~\ref{sec:non-sep-triang}.
\begin{proof}
We first observe that it suffices to establish the equation when
$t=1$, that is, when we do not keep track of the
number of edges.
 Indeed, this number is $\ee(Q)=(\vv(Q)-1)+ (\ff(Q)-1)$, by
Euler's relation,  so that
$\gQ(q,\nu,t,w,z;x,y)=\gQ(q,\nu,1,wt,zt;x,y)$. Let us thus set $t=1$.

Equation~\eqref{Potts-induction} gives  
$$
\gQ(x,y)=1+\gQ_{\backslash }(x,y)+(\nu-1)\gQ_{\slash }(x,y),
$$
where the term 1 is the contribution of the atomic map $m_0$ having one
vertex and no edge, 
$$
\gQ_{\backslash }(x,y)=\frac 1 q  \sum_{Q\in
  \mQ\setminus\{m_0\}}w^{\vv(Q)-1}z^{\ff(Q)-1}x^{\dig(Q)}y^{\df(Q)}
2^{\ddig(Q)}\Ppol_{Q\backslash e}(q,\nu),
$$ 
and 
$$
\gQ_{\slash }(x,y)=\frac 1 q  \sum_{Q\in
  \mQ\setminus\{m_0\}}w^{\vv(Q)-1}z^{\ff(Q)-1}x^{\dig(Q)}y^{\df(Q)} 
2^{\ddig(Q)}\Ppol_{Q\slash e}(q,\nu),
$$
where $Q\backslash e$ and $Q\slash e$ denote respectively the maps
obtained from $Q$ by deleting and contracting the root-edge $e$.

\titre{A. The series  $\gQ_{\backslash }$.}
We consider the partition
$\mQ\setminus\{m_0\}=\mQ'_{\backslash }\uplus\mQ''_{\backslash }$, where
$\mQ'_{\backslash }$  (resp. $\mQ''_{\backslash }$) is the subset of
maps in $\mQ\setminus\{m_0\}$ such that the root-edge is (resp. is not)
an isthmus. 
We denote respectively by  $\gQ_{\backslash}'(x,y)$ and
$\gQ_{\backslash}''(x,y)$ the contributions of 
$\mQ_{\backslash}'$ and $\mQ_{\backslash}''$ to the generating function $\gQ_{\backslash}(x,y)$, so that
$$
\gQ_{\backslash}(x,y)=\gQ'_{\backslash}(x,y)+\gQ''_{\backslash}(x,y).
$$

\titre{A.1. Contribution of $\mQ'_{\backslash }$.}
Deleting the root-edge of a map in $\mQ'_{\backslash }$ leaves two maps
in $\mQ$, as illustrated 
in Figure~\ref{fig:decompositionQa}. Hence there is 
a simple bijection between  $\mQ'_{\backslash }$ and the set of ordered pairs
$(R,R')$ of rooted maps in $\mQ$, such that $R'$ has no
internal digon. The Potts polynomial of this pair can be determined
using~\eqref{Potts-disjoint}. One thus obtains
\beq\label{Q-del-1}
\gQ'_{\backslash}(x,y)=qy^2w\, \gQ(0,y)\gQ(x,y),
\eeq
as  $\gQ(0,y)$ is the generating function of  maps with no internal digon.

\begin{figure}[ht!]\begin{center} \input{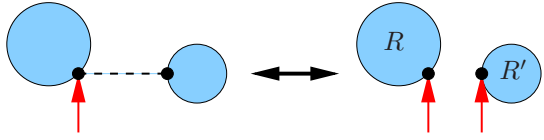}\caption{Decomposition of maps in $\mQ'_{\backslash }$.  The root-edge is dashed.}\label{fig:decompositionQa} \end{center}\end{figure}

\smallskip
\titre{A.2. Contribution of $\mQ''_{\backslash }$.}
Deleting the root-edge of a map $Q$ in $\mQ''_{\backslash }$ gives a
map $R$ in $\mQ$. Conversely, given  $R \in \mQ$, there are at most two
ways to  reconstruct a map of
$\mQ''_{\backslash }$ by adding a new edge and creating a new internal  face:
\begin{enumerate}
\item [--] If $\df(R)\ge 2$, one can create an internal triangle,
\item [--] If $\df(R)\ge 1$, one can create an internal digon; depending on
  whether the root-edge of $R$ is a loop, 
or not, this new digon will be doubly incident to the root, or not.
\end{enumerate}
In terms of incidence-marked maps, one can create an internal
triangle  (provided  $\df(R)\ge 2$), or an internal digon
marked at its first incidence with the root (provided  $\df(R)\ge 1$),
or an internal digon marked at its second incidence with the root
(provided the root-edge of $R$ is a loop). These three possibilities are
illustrated in Figure~\ref{fig:decompositionQa-bis}. In the third case,
the map  $R$ is obtained by gluing at the root two maps $R'$ and
$R''$ such that $R''$ has outer degree 1, and $\Ppol_R(q,\nu)$ is
easily determined using~\eqref{eq:Potts-1components}.
This gives:
\beq\label{Q-del-2}
\gQ''_{\backslash}(x,y)= z\, \frac{Q(x,y)-1-yQ_1(x)}{y}
+ xz (Q(x,y)-1) + xyz Q_1(x) Q(x,y)
\eeq
as  $\gQ_i(x):=[y^i] Q(x,y)$ 
is the generating function of  maps with outer degree $i$. 

\begin{figure}[ht!]\begin{center} \input{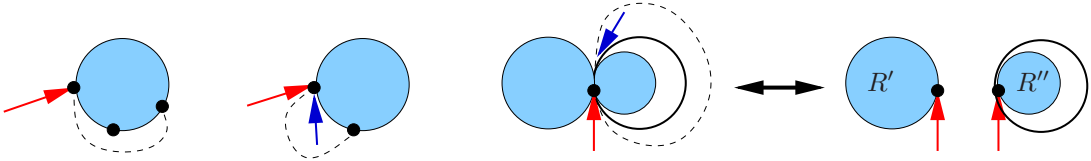}\caption{Decomposition of an incidence-marked map of $\mQ''_{\backslash }$.}\label{fig:decompositionQa-bis} \end{center}\end{figure}

\smallskip
\titre{B. The series  $\gQ_{\slash }$.}
 We now consider the partition $\mQ\setminus\{m_0\}=\mQ'_{\slash }\uplus\mQ''_{\slash }$, where
 $\mQ'_{\slash }$  (resp. $\mQ''_{\slash }$) is the subset of maps in
$\mQ$ such that the root-edge is (resp. is not) a loop.
We denote respectively by   $\gQ_{\slash}'(x,y)$ and 
$\gQ_{\slash}''(x,y)$ the contributions of 
$\mQ'_{\slash }$ and $\mQ''_{\slash }$ to the generating
function $\gQ_{\slash}(x,y)$, so that
$$
\gQ_{\slash}(x,y)=\gQ'_{\slash}(x,y)+\gQ''_{\slash}(x,y).
$$

\smallskip
\titre{B.1.  Contribution of $\mQ'_{\slash }$.}
Contracting the root-edge of a map in $\mQ'_{\slash }$ is
equivalent to deleting this edge. It gives a map $R$ in $\mQ$, formed
of two maps of $\mQ$  attached at a vertex. Hence there is 
a simple bijection, illustrated 
in Figure~\ref{fig:decompositionQb}, between  $\mQ'_{\slash }$ and the
set of ordered pairs 
$(R',R'')$ of rooted maps in $\mQ$, such that the map $R''$ has outer
degree 1 or 2. The Potts polynomial of the map $R$ obtained by gluing
$R'$ and $R''$  can be determined
using~\eqref{eq:Potts-1components}. One thus obtains 
\beq\label{Q-cont-1}
\gQ'_{\slash}(x,y)=yz\gQ(x,y)(2x\gQ_1(x)+\gQ_2(x)).
\eeq
The factor 2 accounts for the two ways of marking incidences in
the new digon that is created when $R''$ has outer degree 1.

\begin{figure}[ht!]\begin{center} \input{decompositionQb.pstex_t}\caption{Decomposition of maps in $\mQ'_{\slash }$.}\label{fig:decompositionQb} \end{center}\end{figure}

\smallskip
\titre{B.2. Contribution of $\mQ''_{\slash }.$}
Contracting the root-edge  $e_0$ of a map in $\mQ''_{\slash }$ gives a
map that may not belong to $\mQ$, as contraction may create faces of
degree 1. This happens when the face to
the left of $e_0$ is an internal digon (Figure~\ref{fig:contraction}).

\begin{figure}[ht!]\begin{center} \input{contraction.pstex_t}\caption{Top: the partition
  $\mQ''_{1\slash }\uplus\mQ_{2\slash }''$ of the set $\mQ''_{\slash }$. 
Bottom: the  result of contracting the root-edge $e_0$ and subsequently removing
  the loops $e_1,\ldots,e_k$.}\label{fig:contraction} \end{center}\end{figure}

For a map in $\mQ''_{\slash }$, we consider the maximal sequence of edges
$e_0,e_1,\ldots,e_k$, 
such that $e_0$ is the root-edge and for $i=1,\ldots, k$, the edges
$e_{i-1}$ and $e_i$ bound an internal digon. We partition $\mQ''_{\slash }$
further, writing $\mQ''_{\slash }=\mQ''_{1\slash } \uplus \mQ''_{2\slash }$,
depending on whether 
the face to the left of $e_k$ is the outer face, or not. We
consistently denote by $\gQ''_{1{\slash}}(x,y)$ and
$\gQ''_{2{\slash}}(x,y)$ the respective contributions of these sets to
$\gQ''_{\slash}(x,y)$.

\smallskip
\titre{B.2.1. Contribution of $\mQ_{1\slash }''$.}
As shown on the left of Figure~\ref{fig:contraction}, there is a bijection
between  the set $\mQ_{1\slash }''$ and the set of triples 
 $(k,R,R')$, where $k\ge 0$ and $R$, $R'$ are
maps in $\mQ$ such that $R'$ has no internal digon.
Let $Q$ be a map in $\mQ_{1\slash }''$ and let  $(k,R,R')$ be its image
by this bijection. By contracting the root-edge $e_0$ of $Q$, the edges
$e_1,\ldots,e_k$ become loops attached to the map  obtained
by gluing $R$ and $R'$ at their root-vertex.
Equation~\eqref{eq:Potts-1components} shows that the Potts polynomial of 
$Q\slash e_0$ is 
$\nu^{k}\Ppol_{R}(q,\nu)\Ppol_{R'}(q,\nu)/q$. 
Considering all triples $(k,R,R')$, one obtains 
\beq\label{Q-cont-2-1}
\gQ_{1\slash}''(x,y)=\frac{y^2w}{1-xz\nu}\, \gQ(0,y)\gQ(x,y).
\eeq

\titre{B.2.2. Contribution of $\mQ_{2\slash }''$.}
Here,  it is convenient to consider  incidence-marked maps. 
Recall that  $\vmQ$ is the set of incidence-marked
maps corresponding to $\mQ$. Similarly,  denote by
$\vmQ_{2\slash }''$ the set of incidence-marked maps corresponding to
$\mQ_{2\slash }''$. As observed above, a map $Q$ in $\mQ$ gives 
$2^{\ddig(Q)}$ incidence-marked maps in $\vmQ$. Hence  
$$
\gQ_{2\slash}''(x,y)=\frac 1 q
\sum_{\vQ\in\vmQ_{2\slash}''}w^{\vv(Q)-1}z^{\ff(Q)-1}x^{\dig(Q)}y^{\df(Q)}
\Ppol_{Q\slash  e_0}(q,\nu),
$$
where $e_0$ is the root-edge of $Q$.
Let $\vQ$ be an incidence-marked map in $\vmQ_{2\slash }''$. With the
notation   $e_1,e_2,\ldots,e_k$ introduced above, the edge $e_k$ is
incident to an internal triangle $f$. By contracting the root-edge
$e_0$, the edges $e_1,\ldots,e_k$ become loops. By deleting these $k$
loops, one obtains a map of $\vmQ$.  The face $f$ becomes a 
 digon $\hat{f}$ incident to
the root-vertex (and doubly incident to the root-vertex if $f$ 
is incident to only 2 vertices;
see Figure~\ref{fig:contraction}). One of
the incidences between the digon $\hat{f}$ and the root-vertex
indicates the position of the contracted edge $e_0$. By marking the digon
$\hat{f}$ at this incidence, one obtains a map $\vR$ in $\vmQ$ such that
 $\Ppol _{Q\slash   e_0}(q,\nu)=
\nu^{k}\,\Ppol_{R}(q,\nu)$ 
(we have used~\eqref{eq:Potts-1components} again).
 Moreover, the mark  created in  $\hat{f}$ is the first 
mark of $\vR$ encountered when turning around the root-vertex in
counter-clockwise direction, starting from the root-edge. This implies  that
the mapping which associates the pair $(k,\vR)$ to the map 
$\vQ$ is a bijection between maps of $\mQ_{2\slash }''$ and pairs $(k,\vR)$
made of a non-negative integer $k$ and an incidence-marked map $\vR$
in $\vmQ$ having at least one internal digon. Considering all pairs
$(k,\vR)$, one obtains
\beq\label{Q-cont-2-2}
\gQ_{2\slash}''(x,y)=\frac{yw}{1-xz\nu} \frac{\gQ(x,y)-\gQ(0,y)}{x}.
\eeq

\medskip
It remains to add up the contributions~(\ref{Q-del-1}--\ref{Q-del-2}),
then the contributions (\ref{Q-cont-1}--\ref{Q-cont-2-2}) 
multiplied by $(\nu-1)$,
and finally the contribution 1 of the  map $m_0$,
to obtain the functional equation of the proposition, at $t=1$. Then,  it
suffices to replace $w$ by $wt$ and $z$ by $zt$ to keep track of the
number of edges.
 The connection between $Q_1(x)$ and $Q_2(x)$ finally follows
 from~\eqref{eq:Q} by extracting  the coefficient of $y^1$.
\end{proof}

\section{A source of invariants}
\label{sec:source}
In this section, we establish an algebraic result that will be 
useful to construct  \emm invariants, (in the sense of Section~\ref{sec:inv-20})
associated with the functional equations of the previous section.

Let $T_m(x)$ be  the $m^{\hbox{\small th}}$ Chebyshev polynomial of the
first kind, defined by 
\beq\label{Cheb-def}
T_{m}(\cos \phi)=\cos(m\phi).
\eeq
This sequence of
polynomials satisfies the recurrence relation $T_0(x)=1$, $T_1(x)=x$,
and for $m\ge 2$, $T_m(x)=2xT_{m-1}(x)-T_{m-2}(x)$. In particular,
$T_m$ has degree $m$. 
Moreover, it is even or odd depending on whether $m$ is even or odd.
%
%
\begin{Proposition}\label{prop:source}
 Let $q \in \cs$, and let $K(u,v)= u^2+v^2-(q-2)uv-1$. 
There exists a polynomial $\J(u) \in \cs[u]$  such that  
$K(u,v)$ divides $\J(u)-\J(v)$ if and only if
$$
q=2+2\cos \frac {2k\pi} m
$$
for some integers $k$ and $m$ such that $0<2k<m$.

Assume this holds, with $k$ and $m$ relatively prime. Let
$J(u)=T_m\left(u\sin \frac{2k\pi}m\right)$. Then $J(u)$ is a solution of minimal degree. 
\end{Proposition}
\noindent Observe that we do not require $m$ to be odd.


\medskip
\noindent {\bf Examples}  \\
Given the conditions $0<2k<m$, the smallest possible value of
$m$ is $3$. 
We focus on values of $k$ and $m$ such that $q$ is an integer, and
 give $\J(u)=T_m\left(u\sin {2k\pi}/m\right)$  up to a constant factor. 

\smallskip\noindent $\bullet$ For  $m=3$ and   $k=1$, we have
$q=1$ and $K(u,v)=u^2+v^2+uv-1$.  The polynomial
$\J(u)$ is  $u^3-u$. The
difference $\J(u)-\J(v)$ satisfies the required divisibility
property: 
$$
\J(u)-\J(v)=(u-v)(u^2+v^2+uv-1).
$$

\smallskip\noindent 
$\bullet$ For  $ m=4$ and $k=1$, we have $q=2$ and $K(u,v)=u^2+v^2-1$.
We find $\J(u)=8u^4-8u^2+1$, and observe that  $\J(u)-\J(v)$ is divisible by $K(u,v)$:
$$
\J(u)-\J(v)= 8(u-v)(u+v) (u^2+v^2-1).
$$

\smallskip\noindent 
$\bullet$ Finally for  $ m=6$ and $k=1$, we have $q=3$ and
$K(u,v)=u^2+v^2-uv-1$.
We find $\J(u)={u}^{6}-2{u}^{4}+{u}^{2}-2/27$ and observe that 
$$
\J(u)-\J(v)= (u-v)(u+v)(u^2+v^2+uv-1) (u^2+v^2-uv-1).
$$

\medskip
\noindent \emph{Proof of Proposition}~\ref{prop:source}.
  Let  $z\in \cs\setminus\{0\}$ be such that
  $q-2=z+z^{-1}$. Such a value $z$   always exists (there are two such
  values in general, but only one if   $q=0$ or $q=4$).

Let $\Phi$ and $\Psi$ be the following linear transformations:
$$
\Phi(u,v)= (v,u) 
\quad \hbox{and }
 \quad \Psi(u,v)=\left(u,(z+z^{-1})u-v \right).
$$
Observe that $\Phi$ and $\Psi$ are involutions, and that they leave
$K(u,v)$ unchanged: 
$$
K(u,v)= K(\Phi(u,v))= K(\Psi(u,v)).
$$
\begin{Lemma}\label{lem:source}
For $i \in  \zs$, denote
$$
u_i= \al_{i}u- \al_{i-1} v \quad \hbox{with} \quad 
\al_i  
= \frac{z^i-z^{-i}}{z-z^{-1}}.
$$  
When $z=1$ (resp. $z=-1$), we take for $\al_i$ the limit value
$\al_i=i$ (resp. $\al_i=(-1)^{i-1}i$).

The orbit of $(u,v)$ under the action of the group generated by
  $\Phi$ and $\Psi$ consists of all pairs $(u_i,u_{i\pm 1})$ for $i$
in $\zs$ (in particular, $(u,v)= (u_1,u_0)$).
Consequently, for all $i$ in $\zs$,
$$
K(u_i,u_{i\pm 1})=K(u,v).
$$
\end{Lemma}
\begin{proof}
One has $(u,v)=(u_1,u_0)$, and for all  $i\in \zs$,
  \begin{align*}
    \Phi(u_i,u_{i\pm 1})= (u_{i\pm 1},u_{i}), \quad
\Psi(u_i,u_{i+1})= (u_{i},u_{i-1}) \quad \hbox{and} 
 \quad \Psi(u_i,u_{i-1})=(u_{i},u_{i+1}).
  \end{align*}
The description of the orbit  follows. The second result follows from
the fact that $\Phi$ and $\Psi$ leave $K(u,v)$ unchanged.
\end{proof}

We now return to the proof of Proposition~\ref{prop:source}. 
Assume  there exists  a polynomial $\J(u)$ such that $K(u,v)$ divides
$\Delta(u,v):=\J(u)-\J(v)$. 
We will prove that $\Delta(u,v)$ has many factors other than $K(u,v)$.  

For a start, an obvious factor of  $\Delta(u,v)$ is $u-v$.

Now for all $i\geq 0$, the polynomial $K(u_{i+1},u_{i})$
divides $J(u_{i+1})-J(u_{i})$.  By Lemma~\ref{lem:source}, this means that $K(u,v)$
divides  $J(u_{i+1})-J(u_{i})$ for all $i\geq 0$,
and it also divides the sum  
$$
\sum_{i=0}^{j-1}\left(J(u_{i+1})-J(u_i)\right)=J(u_{j})-J(u_0)
$$
for all $j\ge 0$.
 Similarly, for all $i\geq 0$, the polynomial
$K(u_{-i-1},u_{-i})=K(u,v)$ divides $J(u_{-i-1})-J(u_{-i})$. 
Summing over $i=0, \ldots, j-1$
shows that $K(u,v)$ divides  $J(u_{-j})-J(u_0)$ for
all $j\ge 0$. Thus finally:
\beq \label{eq:division}
K(u,v) \ \hbox{ divides }\  J(u_j)-J(v)\ \hbox{ for all }\ 
j \in \zs.
\eeq

 If $\al_j\not =0$, we can  express $u$ in terms of $u_j:=w$ and $u_0=v$,
 and the divisibility property 
(between polynomials in $v$ and $w$) reads
 \beq \label{more-div}
 K\left( \frac{ w+\al_{j-1}v} {\al_{j}}, v\right) \ \big |
 \ \J(w)-\J(v) .
 \eeq
 Given that 
 $$
 K\left( \frac{ u+\al_{j-1}v} {\al_{j}}, v\right) = \frac 1 {\al_j^2}\ 
 K_j(u,v)
 $$
 {with}
 $$ 
 K_j(u,v)=u^2+v^2-uv (z^j+z^{-j}) - \left(\frac{z^j-z^{-j}}{z-z^{-1}}\right)^2,
 $$
 we can rewrite~\eqref{more-div} as
 $$
 K_j(u,v) \ \big | \ \J(u)-\J(v)=\Delta(u,v)
 $$
 for all $j$ such that  $\al_{j}\not =0$. Observe that  the
 polynomials $K_j(u,v)$ and $K_k(u,v)$ are relatively prime, unless
 they coincide. Consequently, the collection of polynomials $K_j(u,v)$
 such that $\al_{j}\not =0$ must be finite. 

If there exists  $j\in\zs$ such that $\al_j=0$, then $z^{j}=z^{-j}$
and $z$ is a root of unity. 
If $\al_{j}\not =0$ for
all $j\in\zs$, then there exist $j\not = k$ such that $K_j(u,v)$ and $K_k(u,v)$ 
coincide. This  implies that either $z^{j}=z^{k}$ or  $z^{j}=z^{-k}$. Again, $z$
is a root of unity.

Let us first prove that $z$ cannot be $\pm1$, or equivalently, that $q$
cannot be $4$ or $0$.
For   $z=1$ and $q=4$,
 $$
 K_j(u,v):=u^2+v^2-2uv  - j^2,
 $$
 while for $z=-1$ and $q=0$,
 $$
 K_j(u,v):=u^2+v^2-2uv (-1)^j - j^2.
 $$
 In both cases, $\al_j=0$ if and only if $j=0$, 
 so that the polynomials $K_j$  for which $\al_j\not =0$ form an
 infinite family. Hence $z$ cannot be $\pm 1$.

 Let us denote 
 $$
 z= e^{i \theta } \quad \hbox{with} \quad \theta=\frac{2k\pi}m,
 $$
 with $k$ and $m$ coprime (again, we allow $m$ to be even).  This means
 that we started from
 $$
 q=2+2\cos \frac{2k\pi} m,
 $$
 and thus we may assume $0< 2k < m$, that is, $\theta \in (0, \pi)$. 
 We can now write
 $$
 K_j(u,v)= u^2+v^2-2uv\cos j\theta - \frac{\sin^2j\theta}{\sin^2\theta}.
 $$
 This polynomial divides $\Delta(u,v)$ as soon as $\al_j\not = 0$, that
 is, as soon  as $\sin j\theta \not = 0$. This includes of course $K_1(u,v)=K(u,v)$.

 As $j$ varies in $\zs$, there are as many distinct polynomials
 $K_j(u,v)$ such that  $\sin j\theta \not = 0$  as values of $\cos j\theta= \cos
 2jk\pi/m$ distinct from $\pm 1$.
 Using the fact that $k$ and $m$ are relatively prime, it is easy to
 see that there are  $\lfloor(m-1)/2\rfloor$ such
 values, namely  all values  $\cos 2j\pi/m$ for $j=1, \ldots,
 \lfloor(m-1)/2\rfloor$. Hence for $1\le j \le \lfloor(m-1)/2\rfloor$,
 $$
  u^2+v^2-2uv\cos 2j\pi/m -\frac{\sin^22j\pi/m}{\sin^2\theta}
 $$
 is a divisor of $\Delta(u,v)=\J(u)-\J(v)$. Another divisor is
 $(u-v)$.
 Finally,  if $m$ is even, then $k$ is odd and
 it is easy to see that $u_{m/2}=-v$. By~\eqref{eq:division}, $K(u,v)$ divides $J(-v)-J(v)$,
 which means that 
$J(v)$ is an even polynomial. As $(u-v)$ divides $\J(u)-\J(v)$, it
 follows that $(u+v)$ is another divisor of $\J(u)-\J(v)$. 
 Putting together all divisors we have found, we conclude that
 $\Delta(u,v)=\J(u)-\J(v)$ is a multiple of  
 $$
 \Delta_0(u,v):= (u-v)(u+v)^{\chi_{m,0}} \prod_{j=1}^{\lfloor(m-1)/2\rfloor}
  \left( u^2+v^2-2uv\cos 2j\pi/m -\frac{\sin^22j\pi/m}{\sin^2\theta}\right)
 $$
 where $\chi_{m,0}$ equals $1$ if $m$
 is even, and $0$ otherwise. 
 In particular,  $\J(u)$ has degree at  least $m$. 

 We now claim that 
 \beq\label{D0}
 \left( 2^{m-1}\sin^m \theta\right) \   \Delta_0(u,v)=\J(u)-\J(v),
 \eeq
 for $\J(u)=T_m(u\sin \theta)$.
 This will prove that $T_m(u\sin \theta)$ is a solution to our problem (since $K(u,v)$ is a factor of  $\Delta_0(u,v)$), of minimal degree $m$.

 It suffices to prove~\eqref{D0} for $v=\cos \psi/\sin
 \theta$, for a generic value of 
 $\psi$. We observe that both  sides of~\eqref{D0}
 are polynomials in $u$ of degree $m$ and leading coefficient
 $2^{m-1}\sin^m \theta$.  We now want to prove that they have the same roots.
  We can easily factor $\Delta_0(u,v)$ in linear factors of $u$:
 $$
  \Delta_0(u,v)= \prod_{j=-\lfloor (m-1)/2\rfloor}^{\lfloor m/2\rfloor}
  \left( u -\frac{\cos  (\psi+2j\pi/m)}{\sin \theta}\right).
 $$
 For a generic value of $\psi$, this polynomial has  $m$ distinct roots.
 So it remains to prove that these roots also cancel
 $\J(u)-\J(v)$, \emm i.e.,, that
 $$
 T_m(\cos  (\psi+2j\pi/m))= T_m(\cos\psi).
 $$
 This clearly holds, given that  $T_m(\cos\phi)=\cos(m\phi)$ for any
 $\phi$. 

 This concludes the proof of Proposition~\ref{prop:source}.
 \qed

\section{Invariants for planar maps}
\label{sec:inv-maps}
Consider the functional equation~\eqref{eq:M} we have established
for colored planar maps. By Euler's relation, we do not lose
information by setting  the indeterminate  $z$ to
1: we thus decide to do so.
The functional equation is linear in the main unknown series,
$M(x,y)$. The coefficient of $M(x,y)$ is called the \emm kernel,, and
is denoted by $K(x,y)$:
$$
K(x,y)=1 
-\frac{x^2ywt(\nu-1)}{x-1}-\frac{xy^2t}{y-1}-xyt(x\nu-1)\gM(x,1)
-xywt\left((\nu-1)(y-1)+qy\right)\gM(1,y).
$$
\begin{Lemma}\label{lem:Yi-planar} 
Set  $x=1+ts$.
  The kernel $K(x,y)$, seen as a function of $y$, has two
  roots, denoted $Y_1$ and $Y_2$, in the ring $\qs(q,
  \nu,w,s)[[t]]$. Their constant terms are $1$ and
 $s/(w(\nu-1))$ respectively.
The coefficient of $t$ in $Y_1$ is $s/(w-w\nu+s)$, and in
particular, is non-zero.
\end{Lemma}
\begin{proof}
With $x=1+st$, the equation $K(x,Y)=0$ reads
\begin{multline*}
  (Y-1)\left(1-\frac{Yw(\nu-1)}{s}\right)
=tY \Big({xY}+{(1+x)w(\nu-1)(Y-1)}+x(Y-1)(x\nu-1)\gM(x,1)\\
+xw(Y-1)\left((\nu-1)(Y-1)+qY\right)\gM(1,Y)
\Big).
\end{multline*}

In this form, it is clear that the constant term of a root $Y$ must be
$1$ or $s/w/(\nu-1)$. For each of these choices, the factor $t$
occurring on the right-hand side guarantees the existence of a unique
solution $Y$ (the coefficient of $t^n$ can be determined by induction on $n$).
\end{proof}

\begin{Proposition}\label{prop:inv-planaires}
Set $x=1+ts$ and let $Y_1$, $Y_2$ be the series defined in
Lemma~{\rm\ref{lem:Yi-planar}}. Define
$$
I(y)=
wtyq M ( 1,y ) +{\frac {y-1}{y}}+{\frac {ty}{y-1}}.
$$
Then $I(y)$ is an \emm invariant,. That is, $I(Y_1)=I(Y_2)$.

If, moreover,  $q$ is of the form 
$$
q=2+2\cos \frac{2k\pi}m,
$$
with $0 <2k <m$ and $k$ and $m$ coprime,
then there exists a second invariant,
$$
J(y)= D(y)^{m/2} \ 
T_m\left(
\frac{\be(4- q ) (\by-1)+( q+2\,\be ) I(y) -q}
{2\sqrt {D(y)}}
%
\right),
$$
where  $T_m$ is the $m^{\hbox{th}}$ Chebyshev
polynomial~\eqref{Cheb-def}, $\be=\nu-1$, and 
$$
D(y)= 
(q\nu+\be^2)I(y)^2-q (\nu+1 ) I(y)+ \be t ( q-4 )  ( wq+\be ) +q.
$$
\end{Proposition}
Before proving this proposition, let us recall that
$T_m(x)$ is a polynomial in $x$ of degree $m$, which is even
(resp. odd) if $m$ is even (resp. odd). This implies that $J(y)$ only
involves non-negative integral powers of $D(y)$, and thus is a \emm polynomial,
in $q$, $\nu$, 
$w$, $t$, $\by$  and $I(y)$ with rational coefficients.
Moreover, it follows from the expressions of $I(y)$ and $J(y)$ that,   when
expanded in powers of $t$, $J(y)$ has rational coefficients in 
 $y$ with a pole at $y=1$ of  multiplicity at most $m$. 
\begin{proof}
  Denote $\be =\nu-1$.  The functional equation~\eqref{eq:M} reads
$$
K(x,y) M(x,y)= R(x,y),
$$
where the kernel $K(x,y)$ is
$$
K(x,y)=1-{\frac {{x}^{2}ytw\be}{x-1}}-{\frac {x{y}^{2}t}{y-1}}
-xyt \left( x\nu-1 \right) M ( x,1 ) -xytw \left( y(q+\be) -\be
\right) M ( 1,y ) , 
%
$$
and the right-hand side $R(x,y)$ is:
$$
R(x,y)=1-{\frac {xytM ( x,1) }{y-1}}-{\frac {xytw\be M ( 1,y ) }{x-1}}.
$$
 Both series $Y_i$ cancel the kernel. Replacing $y$ by $Y_i$ in the
 functional equation gives  $R(x,Y_i)=0$. We thus have four equations,
 $K(x,Y_1)=R(x,Y_1)=K(x,Y_2)=R(x,Y_2)=0$, with coefficients in
 $\qs(q,\nu, w,t)$, that relate $Y_1$, $Y_2$,
 $M(1,Y_1)$, $M(1,Y_2)$, $x$ and $M(x,1)$. We will eliminate from this
 system $x$ and $M(x,1)$ to obtain two equations relating $Y_1$, $Y_2$,
 $M(1,Y_1)$ and  $M(1,Y_2)$, and these equations will read
 $I(Y_1)=I(Y_2)$ and  $J(Y_1)=J(Y_2)$.

\smallskip
Let us write $xM(x,1)=S(x)$. We can solve the pair $R(x,Y_1)=0$,
$R(x,Y_2)=0$ for $x$ and $S(x)$. This gives:
$$
\frac 1 x=1 -{\frac {tw\be Y_1\,Y_2\, \left( Y_1-1 \right) M ( 1,Y_1
    ) }{Y_1-Y_2}}
-{\frac {tw\be Y_1\,Y_2\,
 \left( Y_2-1 \right) M ( 1,Y_2 ) }{Y_2-Y_1}},
$$
\beq\label{Ssol1}
S(x)=xM(x,1)=\frac {   \left( Y_1-1 \right) \left( Y_2-1 \right) \left( 
Y_1\,M ( 1,Y_1 ) -Y_2\,M ( 1,Y_2
 )  \right) }
{t{Y_1\,Y_2\, \left(
(Y_1-1)\,M ( 1,Y_1 )   -(Y_2-1)\,M ( 1,Y_2 ) 
\right) }}
\eeq

Let us now work with the equations $K(x,Y_1)=0$ and $K(x,Y_2)=0$. We
eliminate $M(x,1)$ between them. The resulting equation can be 
solved for $x$, yielding a second expression of $1/x$:
\begin{multline}\label{xsol2}
\frac 1 x= 
\frac {tY_1\,Y_2}{   \left( Y_1-1 \right) \left( Y_2-1 \right)}
-{\frac { twY_1\,Y_2 \left((q+\be)Y_1-\be\right)
   \,M ( 1,Y_1    ) }{Y_1-Y_2}}\\
-{\frac {twY_1\,Y_2 \left((q+\be)Y_2-\be\right)
M ( 1,Y_2 ) }{Y_2-Y_1}}.  
\end{multline}
Comparing the two expressions of $1/x$ gives an identity between $Y_1$, $Y_2$,
 $M(1,Y_1)$ and  $M(1,Y_2)$ which can be written as
$$
wt qY_1\, M ( 1,Y_1 ) 
-\frac 1 {Y_1}+{\frac {t}{Y_1-1}}
=
wt q Y_2\, M ( 1,Y_2 )-\frac 1 {Y_2} +{\frac {t}{Y_2-1}}.
$$
This shows that the series $I(y)$ defined in the proposition is indeed
an invariant, as
$$
I(y)=1+t+
 wt qy M ( 1,y ) -{\frac {1}{y}}+{\frac {t}{y-1}}.
$$
 Let us denote $\cI=I(Y_1)=I(Y_2)$.  The above equation gives an
 expression of $M(1,Y_i)$ in terms of $Y_i$ and $\cI$:
\beq\label{MI}
M ( 1,Y_i ) =
-{\frac {(1-{Y_i})^{2}+t{Y_i}^2+\cI \,Y_i(1-Y_i)}
{ tw q {Y_i}^{2} \left( Y_i-1 \right) }}.
\eeq

Now in $K(x,Y_1)=0$, set $M(x,1)=\bx S(x)$, replace $x$ by its expression
derived from~\eqref{xsol2}, and then each $M(1,Y_i)$ by its
expression in terms of $\cI$. Solve the resulting equation for $S(x)$,
and compare  the solution with~\eqref{Ssol1} (where, again,  each term
$M(1,Y_i)$ has been replaced by its expression~\eqref{MI}). This gives
an identity relating $Y_1$, $Y_2$ and $\cI$: 
\begin{multline*}
\be \left( Y_1^2+Y_2^2-(q-2)Y_1Y_2\right)
+ \left((q+2\be)(\cI-2)+q\nu\right)Y_1Y_2(Y_1+Y_2)\\
+ \left( 
\left( q+\be \right) \cI^{2}- \left( 3\,q+4\,\be \right)\cI
+q t\left( wq+\be \right) +2\,q-q\be+4\,\be
\right)
Y_1^2Y_2^2=0.
\end{multline*}
By an appropriate change of variables, we will transform this identity into 
\beq\label{forme-canonique}
U_1^2+U_2^2-(q-2)U_1U_2-1=0
\eeq
and then  apply Proposition~\ref{prop:source}.
First, setting $Y_i=1/Z_i$ gives an equation of total degree 2 in $Z_1$ and
$Z_2$.
Then, a well-chosen translation $Z_i := V_i+ a$ gives
an equation of total degree 2 in $V_1$ and $V_2$ having no linear term:
$$
V_1^2+V_2^2-(q-2)V_1V_2-\frac{q\, \cD}{(4-q)\be^2}=0,
$$
with
$$
\cD=(q\nu+\be^2) \cI^2- q(\nu+1) \cI
+\be t ( q-4 ) ( wq+\be ) +q.
$$
The value of the shift $a$ is
$$
a= 1-\frac{(q+2\be)\cI-q}{\be(4-q)}.
$$
Finally, we have reached an equation of the form~\eqref{forme-canonique}, with
$$
U_i=\frac{\be \sqrt{4-q}}{\sqrt{q\, \cD}}\, V_i=
\frac{\be (4-q)(1/Y_i-1)+(q+2\be)\cI -q}{\sqrt{q(4-q) {\cD}}}.
$$

Now assume $q=2+2\cos \theta$ with $\theta=2k\pi /m$, where $k$ and
$m$ are coprime and $0<2k<m$. Let $u$ and $v$
be two indeterminates. By
Proposition~\ref{prop:source},  the polynomial $u^2+v^2-(q-2)uv-1$
divides the polynomial $T_m(u\sin \theta)-T_m(v\sin
\theta)$. Returning to~\eqref{forme-canonique} shows that
$T_m(U_1\sin \theta)-T_m(U_2\sin \theta)=0$. Equivalently,
$$
T_m\left( \frac{ \be(4-q)(1/Y_1-1)+(q+2\be)\cI -q}{\sqrt{q(4-q)
   {D(Y_1)}}}
\sin \theta\right)=
T_m\left( \frac{ \be(4-q)(1/Y_2-1)+(q+2\be)\cI -q}{\sqrt{q(4-q)
   {D(Y_2)}}}
\sin \theta\right),
$$
where $D(y)$ is defined in the proposition.
In other words, 
$$
T_m\left(  \frac{ (4-q)\be(\by-1)+(q+2\be)\cI -q}{\sqrt{q(4-q)}
  \sqrt {D(y)}}
\sin \theta\right)
$$
is an invariant. Given that $\sin \theta= \sqrt{q(4-q)}/2$, we 
obtain, after multiplying the above invariant by $D(y)^{m/2}$, the
second invariant $J(y)$ given in the proposition.  The
multiplication by  $D(y)^{m/2}$ preserves the invariance, as $D(y)$
only depends on $y$ via the invariant $I(y)$.
 \end{proof}

\section{Invariants for planar triangulations}
\label{sec:inv-triang}
Consider the functional equation~\eqref{eq:Q} we have established
for colored triangulations. We do not lose information by setting
$z=w=1$: by counting edge-face incidences, we obtain, for any map
$Q\in \mQ$, 
$$
 3(\ff(Q)-\dig(Q)-1)+ 2\dig(Q)+\df(Q)=2\ee(Q),
$$
while Euler's relation reads
$$
\vv(Q)+\ff(Q)-2=\ee(Q).
$$
Thus $\vv(Q)$ and $\ff(Q)$ can be recovered from $\dig(Q), \df(Q)$ and
$\ee(Q)$. Let us thus set $w=z=1$.

Equation~\eqref{eq:Q} is linear in the main unknown series,
$\gQ(x,y)$. We call the coefficient of $Q(x,y)$  the \emm kernel,.
\begin{Lemma}\label{lem:kernel-triang}
Set  $x=ts$.
  The kernel of~\eqref{eq:Q}, seen as a function of $y$, has two
  roots, denoted $Y_1$ and $Y_2$, in the ring $\qs(q,
  \nu,s)[[t]]$. Their constant terms are $0$ and $s/(\nu-1)$
  respectively. Both series actually belong to 
$\qs(\nu)[q,s, 1/s][[t]]$.
\end{Lemma}
\begin{proof}
  Denote by $K(x,y)$ the kernel of~\eqref{eq:Q}. After setting $x=st$,
  the equation $K(x,Y) =0$ reads
\beq\label{ker-tr}
Y\left( 1-\frac{(\nu-1)Y}s\right) = \frac{t}{1-\nu st^2} P(Q(x,Y),
Q_1(x), Q_2(x), Q(0,Y), q, \nu, t, s, Y),
\eeq
for some polynomial $P$. The result follows, upon extracting
inductively the coefficient of $t^n$ in the roots $Y_i$.
\end{proof}
The first few terms of $Y_1$ and $Y_2$ read:
\begin{eqnarray*}
Y_1&=&t+{\frac {\nu-1}{s}}{t}^{2}+O \left( {t}^{3} \right) ,\\
Y_2&=&
{\frac {s}{\nu-1}}-{\frac {{s}^{3}q+{s}^{3}\nu-{s}^{3}+{\nu}^{3}-3\,
{\nu}^{2}+3\,\nu-1}{ \left( \nu-1 \right) ^{3}}}t+O \left( {t}^{2}
 \right) .
\end{eqnarray*}

\begin{Proposition}\label{prop:inv-triang}
Let $x=ts$ and let $Y_1$, $Y_2$ be the series defined in
Lemma~{\rm\ref{lem:kernel-triang}}. Define
$$
I(y)=tyq Q ( 0,y ) -\frac 1 y+{\frac {t}{{y}^{2}}}.
$$
Then $I(y)$ is an invariant. That is, $I(Y_1)=I(Y_2)$.

If, moreover, $q$ is of the form 
$$
q=2+2\cos \frac{2k\pi}m,
$$
with $0 <2k <m$ and $k$ and $m$ coprime,
then there exists a second invariant,
$$
J(y)= D(y)^{m/2} \ 
T_m\left(
\frac{
 \be t ( 4-q ) \by+ tq \nu I(y)+\be(q-2)}
{2\sqrt {D(y)}}
\right),
$$
where  $T_m$ is the $m^{\hbox{th}}$ Chebyshev
polynomial~\eqref{Cheb-def}, $\be=\nu-1$, and 
$$
D(y)= 
q\nu^2 {t}^{2} I(y)^{2}
+\be  \left( 4\be +q \right)t I(y)
-q\be\nu {t}^{3}  \left( 4-q \right) +\be^2.
$$
\end{Proposition}
As in the case of planar maps, the fact that
$T_m(x)$ is a polynomial in $x$ of degree $m$, which is even
(resp. odd) if $m$ is even (resp. odd) implies that $J(y)$  is a \emm polynomial, in $q$, $\nu$,
 $t$, $\by$  and $I(y)$ with rational coefficients.
Moreover, the expressions of $I(y)$ and $J(y)$ show that,   when
expanded in powers of $t$, $J(y)$ has rational coefficients in  $y$
with a pole at $y=0$ of  multiplicity at most $2m$.  
\begin{proof}
   The  proof is similar to the proof  of
Proposition~\ref{prop:inv-planaires}, but the
strategy we adopt to eliminate $x$, $Q_1(x)$ and $Q_2(x)$ is different.
First, in Eq.~\eqref{eq:Q}, we replace $Q_2(x)$ by its
expression in terms of $Q_1(x)$,  given in
Proposition~\ref{prop:eq-Q}. This yields
$$
K(x,y) Q(x,y)= R(x,y),
$$
where the kernel $K(x,y)$ is
\beq\label{ker-Q}
K(x,y)=
1-xt-{\frac {t}{y}}-{\frac {yt  \be}{ \left( 1-x\nu t\right) x}}
-{\frac {t{y}^{2} \left( \be +q -x\nu q t \right) 
Q( 0,y ) }{1-x\nu t}}
-{\frac {y \left( \nu+x\nu t-1 \right) Q_1 ( x) }{\nu}}
 ,
\eeq
and the right-hand side $R(x,y)$ is:
\beq\label{RHS-Q}
R(x,y)= 1-xt-{\frac {t}{y}}
-{\frac {\be ytQ ( 0,y ) }{x \left( 1-x\nu t \right) }}
-tQ_1 ( x ).
\eeq
 Both series $Y_i$ cancel the kernel. Replacing $y$ by $Y_i$ in the
 functional equation gives  $R(x,Y_i)=0$. We thus have four equations,
 $K(x,Y_1)=R(x,Y_1)=K(x,Y_2)=R(x,Y_2)=0$, with coefficients in
 $\qs(q,\nu,t)$, that relate $Y_1$, $Y_2$,
$Q(0,Y_1)$, $Q(0,Y_2)$, $x$ and $Q_1(x)$.
We will eliminate from this  system $x$ and $Q_1(x)$
to obtain two equations relating $Y_1$, $Y_2$,  
$Q(0,Y_1)$ and  $Q(0,Y_2)$, 
and these equations will read  $I(Y_1)=I(Y_2)$ and  $J(Y_1)=J(Y_2)$.

\smallskip
Here is the elimination strategy we adopt. We first form two equations
that do not involve $Q_1(x)$: the first one is obtained by eliminating
$Q_1(x)$ between $K(x,Y_1)=0$ and $K(x,Y_2)=0$, the second one is
obtained by eliminating $Q_1(x)$ between $R(x,Y_1)=0$ and
$R(x,Y_2)=0$. Eliminating $x$ between the two resulting equations
gives
$$
{Y_2}^{2}{Y_1}^{3}t q Q ( 0,Y_1 ) 
-{Y_1}^{2}{Y_2}^{3}t q Q( 0,Y_2 ) 
- \left(Y_1 -Y_2 \right) 
 \left( tY_1+tY_2-Y_2\,Y_2 \right),
$$
or equivalently,
$$
I(Y_1)=I(Y_2),
$$
where $I(y)$ is defined as in the proposition. We have thus proved
that $I(y)$ is an invariant. 

Let us denote $\cI=I(Y_1)=I(Y_2)$. From the definition of $I(y)$, we obtain
$$
Q ( 0,Y_i ) 
={\frac {Y_i-t+\cI\,{Y_i}^{2}}{qt{Y_i}^{3} }}.
$$
Let us now eliminate $Q_1(x)$ between $K(x,Y_1)=0$ and $R(x,Y_1)=0$,
on the one hand, and (again) between $R(x,Y_1)=0$ and $R(x,Y_2)=0$, on
the other hand.
Also, we  replace each occurrence of $Q(0,Y_i)$ by its
expression in terms of $Y_i$ and $\cI$. Eliminating $x$ between the
two resulting equations yields:
\begin{multline*}
\be t^2\left(Y_1^2+Y_2^2-(q-2)Y_1Y_2\right)
+t \left(tq\nu \cI+(q-2)\be \right) Y_1Y_2(Y_1+Y_2)\\
+ \left( q(1-2\nu)t\cI +t^3q^2\nu-(q-1)\be \right) Y_1^2Y_2^2
=0.
\end{multline*}
From this point on, the proof mimics the proof of
Proposition~\ref{prop:inv-planaires}. By  the change of variables
$$
U_i=\frac{ \beta t (4-q) \by +tq \nu \cI +\be(q-2)}
{\sqrt{q(4-q)\cD}},
$$ 
we transform the above identity into an identity of the form~\eqref{forme-canonique},
and conclude using Proposition~\ref{prop:source}.
\end{proof}

\section{Theorems of invariants}
\label{sec:inv-thm}

In the previous section, we have exhibited, for each of the two
problems we study, a pair $(I(y),J(y))$ of invariants. We prove here
that in both cases, $J(y)$ is a \emm polynomial, in $I(y)$ with coefficients
in $\qs(q,\nu,w)((t))$.

\subsection{General maps}

\begin{Theorem}\label{thm:inv-planaires} 
Denote $\GK=\qs(q,\nu,w)$.  
Let $Y_1=1+O(t)$ and
  $Y_2=\frac{s}{w(\nu-1)}+O(t)$ be the series of
  $\GK(s)[[t]]$  defined in Lemma~{\rm\ref{lem:Yi-planar}}.
Let $d \in \ns$ 
 and let $\JI(y)\equiv  \JI(q,\nu,t,w;y)$ be a series in $\GK(y)((t))$
having valuation at least $-d$ in $(y-1)$. By this, we mean that for
all  $n$, the coefficient $\ji_n(y):= [t^n]\JI(y)$  
either has no pole at $y=1$, or a pole of multiplicity at most $d$.
Then the composed series $\JI(Y_1)$ and $\JI(Y_2)$ are well-defined 
and belong to  $\GK(s)((t))$. 

If moreover $\JI(y)$ is an invariant (i.e., $\JI(Y_1)=\JI(Y_2)$), then
there  exist Laurent series $A_0,A_1,\ldots,A_d$ in $\GK((t))$ such
that   
$$
\JI(y)=\sum_{i=0}^d A_i\, {I(y)}^i,
$$
where  $I(y)$ is the first invariant defined in 
Proposition~{\rm\ref{prop:inv-planaires}}.  
\end{Theorem}

Before proving this theorem, let us apply it to the case where
$\JI(y)$ is the invariant $J(y)$
of Proposition~\ref{prop:inv-planaires}. As discussed 
just after this proposition,
$J(y)$ has valuation at least $-m$ in $(y-1)$. Hence the above theorem
gives:
\begin{Corollary}\label{coro:eq-inv-M}
Let $q=2+2\cos 2k\pi/m$, with $k$ and $m$ coprime and $0<2k<m$.
 Let $I(y)$ be the first invariant of
Proposition~{\rm\ref{prop:inv-planaires}}. There exist Laurent series
 $C_0, \ldots, C_m$ in $t$, with coefficients in $\qs(q, \nu, w)$, such that
\beq\label{eq-inv}
D(y)^{m/2} \ 
T_m\left(
\frac{\be(4- q ) (\by-1)+( q+2\,\be ) I(y) -q}
{2\sqrt {D(y)}}\right)=
\sum _{r=0}^{m } C_r\, I(y)^r,
\eeq
where $\be = \nu-1$, $T_m$ is the $m^{\hbox{th}}$ Chebyshev polynomial and
$$
D(y)= (q\nu+\be^2)I(y)^2-q (\nu+1 ) I(y)+ \be t ( q-4 )  ( wq+\be ) +q.
$$
\end{Corollary}

\noindent
\emph{Proof of Theorem~\ref{thm:inv-planaires}.}
Let us first prove that the series 
$\JI(Y_1)$ and $\JI(Y_2)$ are well-defined.  Each coefficient $\ji_n(y)$ of
$\JI(y)$ is a rational function in $y$ with coefficients in $\GK$,
with a pole of multiplicity at most $d$ at $y=1=[t^0]Y_1$. 
Given that $[t]Y_1\not = 0$, this implies that
$t^d\ji_n(Y_1)$ is a power series in $\GK(s)[[t]]$, and $\JI(Y_1)$ is
well-defined. 
Moreover, $\ji_n(y)$ has no pole at $y=\frac{s}{w(\nu-1)}=[t^0]Y_2$
since $\ji_n(y)$ does not depend on $s$. Hence $\ji_n(Y_2)$ is a 
series in $\GK(s)[[t]]$, and  $\JI(Y_2)$ is well-defined.

Observe now that $I(y)=\frac{ty}{y-1}+R(y)$, 
where $R(y)$ is a series in 
whose coefficients have no pole at $y=1$. Hence,  
there exist Laurent series $A_0,\ldots,A_d$ in $\GK((t))$ such that the series
$$
G(y):= \JI(y)-\sum_{i=0}^d A_i{I(y)}^i
$$
has coefficients $g_n(y) := [t^n]G(y)$ which are rational in  $y$ and
cancel at   $y=1$.   
(One begins by cancelling the coefficient of $(y-1)^{-d}$ in $\JI(y)-A_d
 I(y)^d$ by an appropriate choice of $A_d$, and then proceeds up to the
 cancellation of the coefficient of $(y-1)^0$ by an appropriate choice
 of $A_0$.)

\medskip
We now suppose that $\JI(y)$ is an invariant and proceed to prove that
$G(y)=0$. 
Note that  $G(y)$ is an invariant (as $I(y)$ and $\JI(y)$
themselves). Thus it suffices to prove the following statement:
\begin{quote}
   \emph{An invariant $G(y)\in \GK(y)((t))$ whose
coefficients $g_n(y)$ vanish at $y=1$ is zero.}
\end{quote}
Let $G(y)=\sum_{n} g_n(y)t^n$ be such an invariant. 
Assume $G(y)\not =0$, and that $G(y)$
has valuation $0$ in $t$ (a harmless assumption, upon multiplying
$G(y)$ by a power of $t$).
We will  prove that the coefficients $[t^0]G(Y_1)$ and
$[t^0]G(Y_2)$ are not equal, which contradicts the fact that $G(y)$ is
an invariant. 

Given that the constant terms of $Y_1$ and $Y_2$ are not poles of any
$g_n(y)$, both  $g_n(Y_1)$ and $g_n(Y_2)$ are formal power series in $t$,
and
$$
[t^0]G(Y_i)= g_0([t^0]Y_i) \quad \hbox{ for } i=1,2.
$$
On the one hand, $[t^0]Y_1=1$ and  $g_n(1)=0$ for all $n$, 
so that $[t^0]G(Y_1)=0$.
On the other hand, $[t^0]Y_2=\frac{s}{w(\nu-1)}$ and  
$g_0(y)$ is different from $0$ by assumption and does not depend on
$s$. Thus $g_0(\frac{s}{w(\nu-1)})\not =0$ 
and we have reached a contradiction. This proves that $G(y)=0$.
\qed

\subsection{Triangulations}
\begin{Theorem} \label{thm:invariant-triang}
Denote  $\GK=\qs(q,\nu)$. 
Let $Y_1=t+O(t^2)$ and $Y_2=\frac{s}{\nu-1}+O(t)$ be the series in
$\GK[s,1/s][[t]]$ defined in Lemma~{\em\ref{lem:kernel-triang}}.
Let $d \in \ns$ and let $\JI(y)\equiv  \JI(q,\nu,t;y)$ be a
series  in $\GK[y,1/y][[t]]$ 
of valuation at least $-2d-1$ in $y$.
Then the composed series $\JI(Y_1)$ and $\JI(Y_2)$ are well-defined 
and belong to $\GK(s)((t))$. 

If moreover $\JI(y)$ is an invariant (i.e. $\JI(Y_1)=\JI(Y_2)$), then there
exist series $A_0,A_1,\ldots,A_d$ in $\GK[[t]]$ such that   
$$
\JI(y)=\sum_{i=0}^d A_i\, {I(y)}^i,
$$
where  $I(y)$ is the first invariant defined in Proposition~{\rm\ref{prop:inv-triang}}.
\end{Theorem}

Before proving this theorem, let us apply it to the case where
$\JI(y)$ is the invariant $J(y)$
of Proposition \ref{prop:inv-triang}. 
As discussed just after this proposition, $J(y)$ has valuation (at
least) $-2m$ in $y$. Hence the above theorem gives, with $A_r=t^rC_r$:
\begin{Corollary}\label{coro:eq-inv-Q}
  Let $q=2+2\cos 2k\pi/m$, with $k$ and $m$ coprime and $0<2k<m$.
 Let $I(y)$ be the first invariant  defined in
  Proposition~{\rm\ref{prop:inv-triang}}. There exist Laurent series
$C_0, \ldots, C_m$ in $t$, with coefficients in  $\qs(q, \nu)$ such that 
\beq\label{eq-inv-triang}
D(y)^{m/2} \ 
T_m\left(
\frac{
 \be ( 4-q ) t\by+ q \nu t I(y)+\be(q-2)}
{2\sqrt {D(y)}}
\right)=
\sum _{r=0}^{m } C_r\,( t I(y))^r,
\eeq
where $\be = \nu-1$, $T_m$ is the $m^{\hbox{th}}$ Chebyshev polynomial and
$$
D(y)= 
q\nu^2 {t}^{2} I(y)^{2}
+\be  \left( 4\be +q \right)t I(y)
-q\be\nu {t}^{3}  \left( 4-q \right) +\be^2.
$$
\end{Corollary}
\noindent 
(The convention  $A_r=t^rC_r$ happens to be convenient in
Section~\ref{sec:alg-triang}.)

\begin{Lemma}\label{lem:degreeYi}
Let $j=1$ or $2$. 
For all $n\in\ns$, the coefficient $[t^n]Y_j$ 
has valuation larger than $-n$ in $s$.
Equivalently,  for all $n,i\geq 0$, $[s^{-n}t^{n-i}]Y_j=0$.
This also means that replacing $t$ by $st$ in $Y_j$ gives a series of
$s\GK[s][[t]]$. The same properties hold for $Y_j^k$, for $k>0$.
\end{Lemma}
\begin{proof}
It is easy to see that, for a series $Y\equiv Y(t)\in \GK[s,1/s][[t]]$, the
following properties are equivalent:
\begin{itemize}
\item[--]  for all $n,i\geq 0$ $[s^{-n}t^{n-i}]Y=0$,
\item[--] $Y(ts)$ belongs to $s\GK[s][[t]]$.
\end{itemize}
The second statement shows that these properties hold for $Y^k$ if
they hold for $Y$. 

We now prove that  each $Y_j$ satisfies the second property.
We start with the series $Y_1=O(t)$.
It satisfies~\eqref{ker-tr}, which implies that the series $Z:=
Y_1(ts)/s$ satisfies
$$
Z= \frac{t}{(1-\nu s^3t^2) ( 1-(\nu-1)Z )} P(Q(x,sZ),
Q_1(x), Q_2(x), Q(0,sZ), q, \nu, ts, s, sZ),
$$
from which it is clear that $Z$ has coefficients in $\GK[s]$.

Similarly, the series $Y_2=\frac{s}{\nu-1}+O(t)$
satisfies~\eqref{ker-tr}, which implies that the series $Z:=Y_2(ts)/s$
satisfies
$$
 Z=\frac 1{\nu-1}- \frac{t}{(\nu-1)(1-\nu s^3t^2)Z} P(Q(x,sZ),
Q_1(x), Q_2(x), Q(0,sZ), q, \nu, ts, s, sZ),
$$
from which it is clear that $Z$ has coefficients in $\GK[s]$.
\end{proof}

\begin{proof}[Proof of Theorem~{\rm\ref{thm:invariant-triang}}]
As in the proof of Theorem~\ref{thm:inv-planaires}, the fact that the
valuation of $\JI(y)$ in $y$ 
is bounded from below, combined with
the fact that $Y_1$ is a power series in $t$, 
implies that $\JI(Y_1)$ is well-defined and is a
Laurent series in $t$. The fact that
$\ji_n(y)$ is independent of $s$, while $Y_2=s/(\nu-1)+O(t)$, implies
that $\JI(Y_2)$ is well-defined and is a formal power series in $t$.

\medskip
Let us construct  series $A_0 \ldots, A_d$ in
$\GK[[t]]$  such that, for $0\le k \le d$,  the coefficient of $y^{-k}$ in
$$
G(y):=\JI(y)- \sum_{i=0}^d A_i I(y)^i
$$
is zero.  
This condition gives a system of linear equations that relates the series $A_i$:
\beq\label{syst}
[y^{-k}] \JI(y)= \sum_{i=0}^d A_i \ [y^{-k}] I(y)^i \quad \hbox{ for } 0\le k \le
d.
\eeq
Recall that
$I(y)=-\frac{1}{y}+tR(y)$ where $R(y)= 1/y^2+qy Q(0,y)$ is a formal
power series in $t$. This implies  that 
$$
[y^{-k}] I(y)^i =
\left\{
\begin{array}{rl}
  (-1)^i + O(t) & \hbox{ if } i=k,\\
 \hskip 3mm O(t) & \hbox{ otherwise.}
\end{array}
\right.
$$
Hence the determinant of the system~\eqref{syst} is $\pm 1 +O(t)$. Hence this
system determines   a unique $(d+1)$-tuple $(A_0,\ldots, A_d)$ of
series of $\GK[[t]]$ satisfying the required conditions.
Note  that the valuation of $G(y)$ in $y$ is at least $-2d-1$.

\medskip
We now suppose that $\JI(y)$ is an invariant, and  proceed to prove that
$G(y)=0$. Note  that  $G(y)$ is an invariant (as 
$\JI(y)$ and $I(y)$ themselves). 
Thus it suffices to prove the following statement:
\begin{quote}
  \emph{An invariant  $G(y) \in \GK[y,1/y]((t))$ 
whose coefficients $g_n(y)$  contain no monomial $y^{-k}$ for $0\le k \le d$ and $k>2d+1$ is zero.}
\end{quote}
Let $G(y)$  be such an invariant. Assume that $G(y) \not = 0$, and 
let $r$ be the valuation of $G(y)$ in $t$.  Write 
$$
G(y)=\sum_{i\ge r} g_i(y)t^i=\sum_{i\ge r,j\ge
  -2d-1} g_{i,j}t^i y^j
$$
 with $g_{i,j}\in \GK$. By assumption, $g_{i,j}=0$ for $-d\le j \le
 0$. Upon multiplying $G(y)$
by a suitable power of $t$, we may assume that
\beq\label{p-def}
 \min\{ i+j : g_{i,j}\not = 0\}=0. 
\eeq
This property is illustrated in Fig.~\ref{fig:diagramme}.
We now want to prove that $r\ge 0$.
This will follow from  studying the valuation of $G(Y_1)=G(Y_2)$ in $t$. 
Recall that $Y_2$ is a formal power series in $t$ with constant term
$\frac{s}{\nu-1}$. This implies that $G(Y_2)$, as $G(y)$ itself, has
valuation $r$ in $t$, the coefficient of $t^r$ in $G(Y_2)$ being
$$
g_r\left(\frac{s}{\nu-1}\right) \not = 0
$$
(since $g_r(y)$ is independent of $s$). Now $Y_1= t+O(t^2)$ and
$$
G(Y_1)= \sum_{i\ge r, j\ge -2d-1} g_{i,j} t^i Y_1^j,
$$
which, according to~\eqref{p-def}, shows that the valuation of
$G(Y_1)$ in $t$ is non-negative. Given that $G(Y_1)=G(Y_2)$, we have
proved that $r\ge 0$.  
By~\eqref{p-def}, there are non-zero coefficients of the form
$g_{i,-i}$.  As $ r\ge 0$ and $g_{0,0}=0$ (by assumption on
$G(y)$), $g_{i,-i}$ can only be non-zero if 
  $i>0$. But then, the assumption on $G(y)$ implies that
the non-zero coefficients  $g_{i,-i}$ are such that $d+1\le i\le 2d+1$.

{\begin{figure}[b]
\begin{center} 
\input{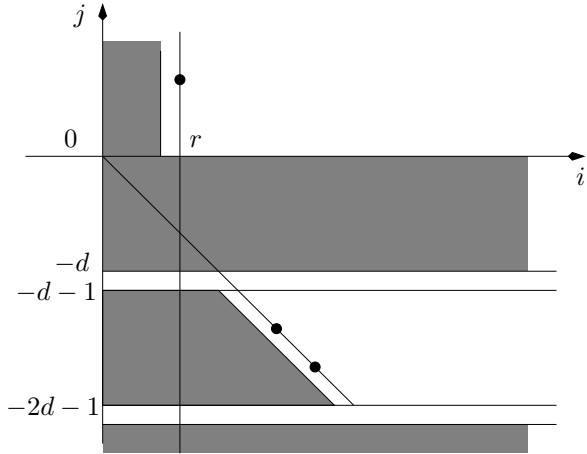}
\caption{The coefficients $g_{i,j}$ of $G(y)$. If $(i,j)$ lies
  in one of the shaded areas,  then $g_{i,j}=0$. The dots indicate
   non-zero coefficients.}
\label{fig:diagramme} 
\end{center}
\end{figure}}

For $i\in [d+1, 2d+1]$, let us denote
  $c_i:=g_{i,-i}$.  One of these coefficients at least is non-zero.
We will now obtain an homogeneous system of $d+1$ linear equations
relating the $c_i$'s 
by writing $[s^{-n}t^{n}]G(Y_1)=[s^{-n}t^{n}]G(Y_2)$ for $n=0,\ldots,d$.

We start with the series $G(Y_2)$. 
 For all $i\ge r$, $g_i(Y_2)$ is a power series of $\GK
 [s,1/s][[t]]$. Hence, for $n=0,\ldots,d$, 
$$
[s^{-n}t^{n}]G(Y_2)~=~[s^{-n}t^{n}]\sum_{i=0}^{n} t^{i}g_{i}(Y_2)
=\sum_{i=0}^{n}[s^{-n}t^{n-i}]g_{i}(Y_2).
$$
Given that $g_{i,j}=0$ for $-d\le j \le 0$ and for $i+j<0$,
the coefficient $g_{i}(y)$ is a \emph{polynomial} in $y$ 
for all $i\leq d$,
with constant term $0$ (see Fig.~\ref{fig:diagramme}).
This, combined with the last statement of
Lemma~\ref{lem:degreeYi}, implies that 
for all $i\leq d$, and all $n\geq 0$,
$[s^{-n}t^{n-i}]g_{i}(Y_2)=0$. Hence
\begin{equation}\label{eq:G(Y2)}
[s^{-n}t^{n}]G(Y_2)~= ~0~~~\textrm{ for all } n=0,\ldots,d.
\end{equation}

Let us  now determine the coefficients $[s^{-n}t^{n}]G(Y_1)$, for
$n=0,\ldots,d$. One has:
$$
[s^{-n}t^{n}]G(Y_1)~=~[s^{-n}t^{n}]\sum_{i\ge
  0}t^ig_i(Y_1)~=~\sum_{i\ge 0}[s^{-n}t^{n-i}]g_i(Y_1).
$$
By Lemma \ref{lem:degreeYi}, the coefficient $[s^{-n}t^{n-i}]Y_1^j$
is $0$ for $j>0$.  Given that, by assumption on $G(y)$,
$$
g_i(y)= \sum_{j=d+1}^{\min(i,2d+1)}g_{i,-j} y^{-j}
+\sum_{j>0} g_{i,j}y^j,
$$
we are left with
\beq \label{eq:coeff-G(Y1)}
  [s^{-n}t^{n}]G(Y_1)=\sum_{i>d}
\sum_{j=d+1}^{\min(i,2d+1)}g_{i,-j} [s^{-n}t^{n-i}]Y_1^{-j}.
\eeq
Let $W=t/Y_1=1+O(t)$. By Lemma~\ref{lem:degreeYi}, 
the coefficient of $t^n$ in $Y_1/t$ has valuation at least $-n$ in $s$.
Hence the same holds for $[t^n]W$, and more generally, for 
$[t^n]W^j$, for all $j>0$.
Thus, for  $n\geq 0$ and  $0<j<i$, 
\beq\label{j<i}
[s^{-n}t^{n-i}]{Y_1}^{-j}~=~[s^{-n}t^{n-i+j}]W^j~=~0.
\eeq
To capture the case  $j=i$, let us denote
$a_n=[s^{-n}t^n]W=[s^{-n}t^{n-1}]1/Y_1$ 
and introduce the series $ A(t)=\sum_{n\ge 0} a_nt^n$, the first terms
of which are found to be
\beq\label{A-def}
A\equiv A(t)~=~\sum_{n\ge 0} a_nt^n~=~1-(\nu-1)t+O(t^2).
\eeq
Then
\beq\label{j=i}
[s^{-n}t^{n-i}]{Y_1}^{-i}~=~[s^{-n}t^{n}]W^i~=~[t^{n}]A^i.
\eeq
The second equality only holds because $[t^n]W$ has valuation at least
$-n$ in $s$.
Recall that $g_{i,-i}=c_i$. Returning to~\eqref{eq:coeff-G(Y1)}
and using~\eqref{j<i} and \eqref{j=i} now gives
\begin{equation}\label{eq:G(Y1)}
[s^{-n}t^{n}]G(Y_1)~=~\sum_{i=d+1}^{2d+1}c_i\,[t^{n}]A^i, ~~~\textrm{ for } n=0,\ldots,d.
\end{equation}
Given that $G(y)$ is an invariant, we can now
equate \eqref{eq:G(Y2)} and \eqref{eq:G(Y1)}. This gives a homogeneous system of
$d+1$ equations  
\begin{equation}\label{eq:system-c}
\sum_{i=d+1}^{2d+1}c_i\,[t^{n}]A^i~=~0, ~~~\textrm{ for } n=0,\ldots,d,
\end{equation}
that relates  $d+1$ unknown coefficients $c_i$,
$i=d+1,\ldots,2d+1$. We will now prove that this system implies that
each $c_i$ is zero, thereby reaching a contradiction.

Define the polynomials $C(x)$ and $D(x)$ by $C(x)=\sum_{i=d+1}^{2d+1}c_i
x^i=x^{d+1}D(x)$.  Note that $D(x)$ has degree at most $d$. The above system means that $C(A)=O(t^{d+1})$, or
equivalently, $D(A)=O(t^{d+1})$ (since $a_0\not = 0$). Write $A= a_0+tB$,
with $B=\sum_{n\ge 1} a_nt^{n-1}$.  By Taylor's formula,
$$
D(A)=\sum_{k=0}^{d} \frac{D^{(k)}(a_0) } {k!} (tB)^k=O(t^{d+1}).
$$
Extracting the coefficient of $t^0$ in this identity gives
$D(a_0)=0$. Then extracting the coefficient of $t^1$ gives
$D'(a_0)a_1=0$. But $a_1\not = 0$ (see~\eqref{A-def}), and thus
$D'(a_0)=0$. Extracting inductively the coefficients of $t^2, \ldots,
t^{d}$ gives finally $D(a_0)=D'(a_0)=\cdots = D^{(d)}(a_0)
=0$. But a polynomial of degree (at most) $d$ with a root of
multiplicity $d+1$ must be zero, hence $D(x)=C(x)=0$ and all
coefficients $c_i$ vanish. We have reached a contradiction, and the
invariant $G(y)$ must be zero.
\end{proof}

\section{Equations with one catalytic variable and algebraicity}
\label{sec:alg}
 One key tool of this paper is an
  algebraicity theorem which applies to series satisfying a
 polynomial equation   with one catalytic variable. It generalizes
 slightly Theorem 3  in~\cite{mbm-jehanne}.

Let $\GK$ be a field of characteristic 0. 
Let $F(u)\equiv F(t,u)$ be a power series in $\GK(u)[[t]]$, that is, a series in
$t$ with rational coefficients in $u$. Assume that these coefficients 
have no pole at $u=0$. 
The following divided difference (or discrete derivative) is then well-defined:
$$
\Delta F(u) = \frac{F(u)-F(0)}u.
$$
Note that 
$$
\lim _{u\rightarrow 0} \Delta F(u) =F'(0),
$$
where the derivative is taken with respect to $u$. The operator
$\Delta ^{(i)}$ is obtained by applying $i$ times $\Delta$, so that:
$$
\Delta ^{(i)} F(u) = \frac{F(u)-F(0)-uF'(0) -\cdots -
  u^{i-1}/(i-1)!\,F^{(i-1)}(0)}{u^i}. 
$$
Now
$$
\lim _{u\rightarrow 0} \Delta ^{(i)} F(u) =\frac{F^{(i)}(0)}{i!}.
$$
Assume $F(t,u)$ satisfies a functional equation of the form
\beq
\label{main-eq}
F(u)\equiv F(t,u) = F_0(u)+ t\  Q\Big( F(u), \Delta F(u), \Delta
^{(2)}F(u),\ldots ,\Delta ^{(k)}F(u), t; u\Big) ,
\eeq
where $F_0(u)\in \GK(u)$ and $Q(y_0, y_1, \ldots , y_k, t; v)$ is  a polynomial in the $k+2$ first
indeterminates $y_0, y_1, \ldots , y_k, t$, and a \emm rational function, in
the last indeterminate $v$, having  coefficients
 in  $\GK$. 
Extract from~\eqref{main-eq} the coefficient of $t^0$: this gives
 $F_0(u)=F(0,u)$. In particular, $F_0(u)$ has no pole at $u=0$.
\begin{Theorem} 
\label{generic-thm}
Under the above assumptions, the series $F(t,u)$  is algebraic over $\GK(t,u)$.
\end{Theorem}
\begin{proof}
Let us first prove that $F(t,u)$ satisfies an equation
of the form~\eqref{main-eq} such that $Q(y_0,  \ldots , y_k,
t; v)$ has no pole at $v=0$
(but possibly with a larger value of $k$). 
 Assume  that $Q(y_0, y_1, \ldots , y_k,
t; v)$ has a pole of order $m>0$ at $v=0$. Write
$$
Q(y_0, y_1, \ldots , y_k, t; v)= \frac 1 {v^m} Q_m(y_0, y_1, \ldots ,
y_k, t)
+ \bar Q(y_0, y_1, \ldots , y_k, t; v)
$$
where $Q_m$ is a polynomial in its $k+2$ variables, and $\bar Q(y_0,
y_1, \ldots , y_k, t; v)$ is a polynomial in its first $k+2$ variables
and a rational function in $v$, having a pole of order at
most $m-1$ at $v=0$.
Multiply~\eqref{main-eq} by $u^m$, and take the limit as $u\rightarrow
0$. This gives 
\beq\label{id}
Q_m(F(0), F'(0), \ldots, F^{(k)}(0)/k!, t)=0.
\eeq
Note that  for all $i\ge 0$,
$$
\Delta^{(i)}F(u)=F^{(i)}(0)/i! + u \Delta^{(i+1)} F(u).
$$ 
In $Q_m \left( F(u),  \ldots ,\Delta ^{(k)}F(u), t\right)$, replace each
$\Delta^{(i)}F(u)$ by  the above expression.
This gives:
\begin{multline*}
 Q_m \Big( F(u),  \ldots ,\Delta ^{(k)}F(u), t\Big)
=\\
 Q_m(F(0), F'(0),\ldots, F^{(k)}(0)/k!, t)
+ u\, \hat Q_m\Big(F(0),\ldots, F^{(k)}(0), \Delta F(u),  \ldots
,\Delta ^{(k+1)}F(u), t\Big),
\end{multline*}
for some polynomial $\hat Q_m$, or, after replacing $F^{(i)}(0)$ by
$i! (\Delta^{(i)}F(u)- u \Delta^{(i+1)} F(u))$ in  $\hat Q_m$,
\begin{eqnarray*}
Q_m \Big( F(u),  \ldots ,\Delta ^{(k)}F(u), t\Big)
&=&Q_m(F(0), \ldots, F^{(k)}(0)/k!, t)
+ u \tilde Q_m\Big( F(u),  \ldots ,\Delta ^{(k+1)}F(u),
t\Big)\\
&=& u \,\tilde Q_m\Big( F(u), \ldots ,\Delta ^{(k+1)}F(u), t\Big).
\end{eqnarray*}
(by~\eqref{id}),
 for some polynomial $\tilde Q_m$.
Hence~\eqref{main-eq} can be rewritten as
\begin{eqnarray*}
F(u)&=&F_0(u)+ \frac t {u^{m-1} } \,\tilde Q_m\Big( F(u), \ldots ,\Delta ^{(k+1)}F(u), t\Big)
+t \bar Q \Big( F(u), \ldots ,\Delta ^{(k)}F(u), t;u\Big)
\\
&=&F_0(u)+ t \tilde Q \Big( F(u), \Delta F(u),\ldots ,\Delta ^{(k+1)}F(u), t;u\Big)
\end{eqnarray*}
where now $\tilde Q(y_0,y_1, \ldots , y_k, t; v)$ is a polynomial 
in its first $k+2$ variables
and a rational function in $v$, having a pole of order at
most $m-1$ at $v=0$. 
In this way, we can decrease step by step the order of the pole at
$v=0$ in $Q$, until we reach a rational function  $Q$ that has no pole at
$v=0$. Observe that  $k$ increases during this procedure.

\medskip
Let us now assume that~\eqref{main-eq} holds and that $Q(y_0, \ldots, t;v)$ has no pole
at $v=0$. We want to prove that $F(t,u)$ is algebraic.
 As in \cite{mbm-jehanne}, we first introduce a
small perturbation of ~\eqref{main-eq}. Let $\eps$ be a
new indeterminate, and  consider the equation
\beq
\label{epsilon-eq}
G(u)\equiv G(z,u, \epsilon) = F_0(u)+ \eps ^k z \Delta ^{(k)}G(u) +
z^2 Q\Big( G(u), \Delta G(u), \ldots 
,\Delta ^{(k)}G(u), z^2; u\Big) 
\eeq
where  $F_0$ and $Q$ are the same  as in the equation satisfied by $F$.
Given that $Q$ has no pole at $v=0$, one can see, by extracting
inductively the coefficient of $z^n$, for $n\ge0$,  that this equation
defines a unique solution $G(u)$ in the  ring of \fps \ in $z$ with  
coefficients in $\GK(u)[\eps]$. These coefficients have no pole at
$u=0$.
Moreover, $G(z,u, 0)= F(z^2,u)$, so
that it suffices to prove that $G(z,u, \epsilon)$ is algebraic over
$\GK(z,u,\eps )$.

\smallskip
Let $u^m D(u)$, with $D(0)\not = 0$, be a polynomial of $\GK[u]$ of minimal degree
such that multiplying~\eqref{epsilon-eq} by  $u^m D(u)$ gives  a \emm
polynomial, equation of the form
\beq
\label{eq-G-pol}
P\Big( G(u), G_1, \ldots , G_k, z;u \Big) =0
\eeq
for some polynomial $P(x_0, x_1, \ldots, x_k, z;u)$, with $G_i=
G^{(i-1)}(0)/(i-1)!$.  
Note that, because of the  
term $\eps ^k z \Delta ^{(k)}G(u)$ occurring in~\eqref{epsilon-eq}, we
have $m\ge k$.
Let us apply to~\eqref{eq-G-pol} the general strategy of~\cite{mbm-jehanne}. We 
need to find sufficiently many \emm fractional power series, $U$ in $z$
(that is, \fps\ in $z^{1/p}$ for some $p>0$),
 with coefficients in some algebraic closure of $\GK(\eps)$, satisfying
$$
\frac{\partial P}{\partial x_0}\Big( G(U),  G_1, \ldots , G_k,
z;U\Big) =0.
$$
This reads
$$
U^m D(U) \left( 
1- \frac{\eps ^k}{U^k} z - z^2 
\sum_{i=0}^k \frac 1{U^{i}} \, \frac{\partial Q}{\partial y_i}
\Big(G(U), \ldots, \Delta^{(k)} G(U), z^2; U\Big)
\right)=0.
$$
Among the solutions of this equation are the 
 solutions of
$$
U^k= \eps ^k z + z^2 
\sum_{i=0}^k U^{k-i} \, \frac{\partial Q}{\partial y_i}
\Big(G(U), \ldots, \Delta^{(k)} G(U), z^2; U\Big).
$$
Observe that the right-hand side has no pole at $U=0$.
Let us focus on solutions $U\equiv U(z)$ having constant term 0.
It is not hard to see, by an harmless extension of Puiseux's theorem~\cite[Chap.~4]{walker}, that this
equation has exactly $k$ such solutions
$U_1, \ldots , U_k$. Their coefficients lie in an algebraic closure of
$\GK(\eps)$. More precisely, the Newton-Puiseux algorithm shows that
these series  can be written as 
\beq \label{Ui-expansion}
U_i= \eps \, \xi ^i  s\, (1+V( \xi ^i  s))
 \eeq
where $s=z^{1/k}$, $\xi$ is a primitive $k$th root of unity and $V(s)$
is a \fps \ 
in $s$ with coefficients in $\GK( \eps)$, having constant term $0$. In
particular, the $k$ series $U_i$ are distinct, non-zero, and
$D(U_i)\not = 0$.

The rest of the proof is a simple adaptation
of~\cite[p.~636--638]{mbm-jehanne}. The only difference is the factor
$D(u)$ now involved in the construction of the polynomial $P$. One has
to use the fact that $D(U_i)=D(0)+O(s)$ where  $D(0)\not = 0$.
\end{proof}

\section{Algebraicity for colored planar maps}
\label{sec:alg-maps}
In this section, we prove our first algebraicity theorem for colored maps.
 We consider the Potts \gf\  $M(x,y)\equiv M(q, \nu, t, w ,z ; x,y)$ of
 planar maps, defined by~\eqref{potts-planar-def}.
This series is characterized by the functional equation~\eqref{eq:M}.

\begin{Theorem}\label{thm:alg-planaires}
 Let $q\not= 0,4$ be of the form
$2 +2 \cos j \pi/m$ for two integers $j$ and $m$.
Then the series $M( q, \nu, t, w,z ; x,y)$ 
is algebraic over $\qs(q, \nu,t,w,z,x,y)$. 
\end{Theorem}

\noindent{\bf Caveat.}
 The series $M( q, \nu, t, w,z ; x,y)$ is \emm not, algebraic for a
generic value of $q$. 
That is, there exists no non-trivial polynomial $P$ such that 
$P( q, \nu, t, w,z , x,y, M( q, \nu, t, w,z ; x,y))=0$ when $q, \nu,
t, w,z , x,y$ are indeterminates.
Otherwise, the series  $\gtM( \mu, \nu, w ,z; x,y)$ 
 counting maps weighted by their Tutte polynomial 
and related to $M$ by~\eqref{Potts-Tutte-gfs}
would be algebraic over
$\qs(\mu, \nu,  w ,z, x,y)$ for generic values of $\mu$ and
$\nu$. However, it is known
that~\cite{mullin-boisees,bernardi-boisees}: 
$$
\gtM(1,1, t, t ;1,1)=\sum_{n\ge 0} \frac 1{(n+1)(n+2)} {{2n}\choose
  n}{{2n+2}\choose {n+1}} t^n,
$$
and the asymptotic behaviour of the $n^{\hbox{th}}$ coefficient, being
$\kappa\, 16^n n^{-3}$, prevents this series from being
algebraic~\cite{flajolet-context-free}. 
By Tutte's original  
description of what was not yet  called the Tutte polynomial,
the above series counts planar maps enriched with a spanning
tree~\cite{tutte-dichromate}. 

\medskip
As the variable $z$ is redundant, it suffices to prove Theorem~\ref{thm:alg-planaires}
for $z=1$. We thus set $z=1$ and denote the series
$M(q,\nu,t,w,1; x,y)$ by $M(q,\nu,t,w;x,y)$.
The conditions on $q$ imply that there exist two coprime integers  $k$
and $m$  such  that $0<2k <m$ and  $q=2+2\cos
2k\pi/m$. Corollary~\ref{coro:eq-inv-M} thus applies, and gives a
polynomial equation  in $I(y)$ involving $m+1$ unknown series
$C_r$. We call this equation the \emm invariant equation,.
From this point, we prove 
Theorem~\ref{thm:alg-planaires} in two steps: we first show that the series $C_r$ can be expressed in
terms of the $y$-derivatives of $M(1,y)$, evaluated at $y=1$; then, we
prove that, once $I(y)$ and each $C_r$ are replaced, in the invariant equation, by
their expressions in terms of $M$, our general algebraicity theorem applies. That is to say, the equation
obtained for $M(1,y)$ has the form~\eqref{main-eq},  with $u$ replaced by
$(y-1)$. This second step is more delicate than the first.

Before we study the general case, let us examine thoroughly  a simple
example: $q=1$. We refer  the reader who would like to see more explicit
cases to Sections~\ref{sec:two-planar} and~\ref{sec:three-planar} (respectively
devoted to $q=2$ and $q=3$).

\subsection{A simple example: one-colored planar maps}
\label{sec:one-color}
Take $k=1$ and $m=3$, so that the number of colors is 
$q=1$. Of course, all edges of a $1$-colored
map are monochromatic, so that the variable $\nu $ becomes redundant,
but we keep it for the sake of generality (a degeneracy  actually
occurs if we set $\nu=1$ at this stage). 

The third Chebyshev polynomial is $T_3(x)= 4x^3-3x$. The invariant
equation~\eqref{eq-inv}  thus reads
\beq\label{eq-invq1}
\frac 1 2 N(y)^3 -\frac 3 2 N(y) D(y)-\sum _{r=0}^{3} C_r\, I(y)^r=0,
\eeq
with 
$$
N(y)=3 (\nu-1) (\by-1)+( 2\nu-1 ) I(y) -1,\quad 
I(y)= wty M ( 1,y ) +{\frac {y-1}{y}}+{\frac {ty}{y-1}}
$$
and
$$
D(y)= (\nu^2-\nu+1)I(y)^2- (\nu+1 ) I(y)-3  t (\nu-1) ( w+\nu-1 )+1.
$$
We write
$$
I(y)= K(y) +{\frac {y-1}{y}}+{\frac {t}{y-1}}
\quad \hbox{where} \quad K(y)=t+ wty M ( 1,y ).
$$
This is not crucial in this simple case, but will be
convenient in the general case. 

Recall that the  series $C_r$ depend on $\nu, t, w$, but not on
$y$. Expand the left-hand side of~\eqref{eq-invq1} around $y=1$: the first
non-trivial term is $O((y-1)^{-3})$, and one obtains:
$$
\frac{{t}^{3}}2 \left( 2-3\,\nu-3\,{\nu}^{2}+2\,{\nu}^{3}-2\,C_{{3}}
 \right)  \left( y-1 \right) ^{-3}+O (  \left( y-1 \right) ^{-2}
 ) =0,
$$
from which we  determine $C_3$ explicitly:
\beq\label{C3-sol-11} 
C_3=
\left( \nu-2 \right)  \left( 2\,\nu-1 \right)  \left( \nu+1
 \right) /2.
\eeq
By pushing the expansion of~\eqref{eq-invq1} around $y=1$ up to the
term $(y-1)^0$, we find explicit expressions of the other three series $C_r$:
\begin{eqnarray} 
C_2&=&
6\,\nu-3\,{\nu}^{2}/2-3/2,\nonumber
 \\ 
C_1&=&-9/2\, \left( \nu-1 \right)  \left( {\nu}^{2}-2\,\nu\,w-1+w \right) t-
3/2-3\,\nu/2, \label{Cr-q=1}
\\
C_0&=&-{\frac {27}{2}}\,\nu \left( \nu-1 \right) ^{2}t\, K(1) 
+{\frac {27}{2}}\,\nu \left( \nu-1 \right) ^{2}{t}^2
+\frac 9 2\, \left( \nu-1 \right)  \left( 2\,\nu-2-w \right) t +1.
\nonumber
\end{eqnarray}
Observe that the expression of $C_0$ involves the (unknown) series $K(1)$.

Now replace in~\eqref{eq-invq1} each series $C_r$ by its expression. This
gives
\beq\label{eq-K-1}
-27/2 (\nu-1)^2 
\Big( \nu (1-\by ) K(y) ^{2}+
 ( 2\,\by\nu\,t-\by+\by^2-\nu\,t ) K (y) 
-\nu  t K(1) +t (1-\by) ( w+\nu\,t+\by ) \Big)=0.
\eeq
This equation involves a single catalytic variable, $y$. However, it
cannot be immediately written in the form~\eqref{main-eq}: when $t=0$,
the expression between parentheses contains a quadratic term $K(y)^2$,
which is absent from~\eqref{main-eq}.

Let us  replace $K(y)$ by $t+twyM(y)$ and $K(1)$ by $t+twM(1)$, where
$M(y)\equiv M(1,y)$.  More factors come out, including a factor
$t$. Precisely, the equation now reads:
\beq\label{eq-M-1}
-27/2(\nu-1)^2 tw \by \Big( 
{y}^{2}t\nu\,w \left( y-1 \right) M( y )^{2}
+ \left( \nu\,t{y}^{2}+1-y \right) M( y )
-ty\nu\,M (1 ) +y-1
\Big) =0,
\eeq
or, after dividing by $27/2 (\nu-1)^2 t w\by (1-y)$ and isolating the
term $M(y)$, 
\beq\label{1cat-planaires1}
M(y)= 1+ y^2 t\nu w  M(y)^2 + \nu t y \, \frac{yM(y)-M(1)}{y-1}.
\eeq
This equation has the form~\eqref{main-eq} (with
$u$ replaced by $y-1$), so that Theorem~\ref{generic-thm} applies: The series
$M(1,y)\equiv M(1, \nu, t,w;1,y)$ is algebraic. 
The algebraicity of $M(1,\nu, t,w;x,y)$ easily follows, as explained at the end of this section.
The experts will have recognized
in~\eqref{1cat-planaires1} the standard functional equation obtained by
deleting recursively  the root-edge in planar maps~\cite{tutte-general}.

\subsection{The general case}

We now want to prove that the treatment we have
applied above to~\eqref{eq-inv} in the case $q=1$ can
  be applied for all values $q=2+2\cos 2 k\pi/m$. More precisely:
\begin{itemize}
\item expanding~\eqref{eq-inv} around $y=1$ and extracting the
  coefficient of $(y-1)^{-r}$
provides an expression of the series $C_r$, for $0\le r \le m$, as a
polynomial in $t$,  $K(1), K'(1), \ldots,  K^{(m-r)}(1)$ (where $K(y)=t+ wtyq M ( 1,y )$), with coefficients in $\GK:=\qs(q,\nu,w)$;
\item after expressing in~\eqref{eq-inv} the invariant $I(y)$ and each
series $C_r$ in terms of $K$,  then in terms of $M$, and finally dividing 
by $t$ and by a non-zero element of 
  $\GK(y)$, the resulting equation can be written in the form
\beq\label{generic-M}
M(y)=1 + t \, P\left(M(y), \Delta M(y), \ldots, \Delta^{m+1}M(y), t;y\right),
\eeq
where $M(y)\equiv M(1,y)$,  $\displaystyle ~\Delta F(y)=
\frac{F(y)-F(1)}{y-1},$ and $P(x_0, x_1, \ldots, x_{m+1}, t;v)$ is a
polynomial in its first $m+3$ variables, and a rational function in
the last one, having coefficients in $\GK$. 

\end{itemize}
One can then apply Theorem~\ref{generic-thm}, and conclude that the
series $M(y)\equiv M(q, \nu, t, w; 1,y)$ is algebraic. A duality argument, combined 
with the original equation~\eqref{eq:M}, finally proves that the
\gf\ $M(q, \nu, t, w; x,y)$ counting 
$q$-colored planar maps is  algebraic as well.

\medskip
\noindent{\bf Remark.} As suggested by the example of
Section~\ref{sec:one-color}, the series $C_r$ can be
expressed in terms of $M(1), \ldots, M^{(m-r-3)}(1)$ only, 
but we do not need so much precision here.

\subsubsection{Determination of the series $C_r$}
\label{sec:Cr}

It will be convenient to write
\beq\label{I-K}
I(y)= K(y) +{\frac {y-1}{y}}+{\frac {t}{y-1}}
\eeq
where 
$$
K(y)=t+ wtyq M ( 1,y ).
$$
Consider  the invariant equation~\eqref{eq-inv}. Recall that
$T_m(x)$ is a polynomial in $x$ of degree $m$, which is even
(resp. odd) if $m$ is even (resp. odd). That is, denoting
$m=2\ell+\eps$ with  $\eps\in\{0,1\}$,
\beq\label{Tma}
T_m(x)=\sum_{a=0} ^\ell T_m^{(a)}\;x^{2a+\eps},
\eeq
where $T_m^{(a)}\in \qs$.
Thus the left-hand side of~\eqref{eq-inv} reads
\begin{multline*}
 \sum_{a=0} ^\ell T_m^{(a)}\; 2^{-(2a+\eps)}
\left( \be(4- q ) (\by-1)+( q+2\,\be ) I(y) -q\right)^{2a+\eps}\\
\left((q\nu+\be^2)I(y)^2-q (\nu+1 ) I(y)+ \be t ( q-4 )  ( wq+\be ) +q\right)^{\ell-a},
\end{multline*}
and thus appears as a polynomial  of degree $m$ in $I(y)$, with coefficients in
$\qs[q,\nu,t,w,\by]$ (recall that $\be =\nu-1$). We denote by $L_r(t;y)$ the coefficient of
$I(y)^r$  in this polynomial, so that the invariant equation now reads
\beq \label{eq-inv-K}
\sum_{r=0}^m L_r(t;y)I(y)^r=\sum _{r=0}^{m } C_r\, I(y)^r,
\eeq
where $L_r(t;y) \in \qs[q,\nu, t,w,\by]$.

\begin{Lemma}\label{lem:cr}
The series $C_m, C_{m-1}, \ldots , C_0$ can be determined
  inductively by
  expanding~\eqref{eq-inv-K} in powers of   $y-1$ and
extracting  the  coefficients of $(y-1)^{-m}, \ldots, (y-1)^0$.
This gives, for   $0\le r\le m$, 
$$
C_r=\Pol_r(K(1), K'(1), \ldots, K^{(m-r)}(1),t)
$$
for some polynomial $\Pol_r(x_1, \ldots, x_{m-r+1},t)$ 
having coefficients in $\GK:=\qs(q,\nu,w)$. 
Moreover, $\Pol_r(x_1, \ldots, x_{m-r+1},t)$ has constant term
$\Pol_r(0, \ldots, 0)=L_r(0;1)$ and  contains no
monomial $x_j$,  for $1\le j \le m-r+1$. 
\end{Lemma}

The first and last statements in this lemma are easily seen to hold
in the case $q=1$, using~\eqref{C3-sol-11} and~\eqref{Cr-q=1}.

\begin{proof}
Recall the expression~\eqref{I-K} of $I(y)$ in terms of
$K(y)$,
and expand the right-hand side of~\eqref{eq-inv-K} as follows:
\begin{eqnarray*} 
\RHS&=& \sum_{a=0}^{m} C_a\,  \left( K(y)+ \frac {y-1}{y} + \frac t {y-1}\right)^{a}\\
&=& \sum_{i=0}^{m} \frac{t^{i}}{(y-1)^i}\sum_{a=i} ^m {a\choose i}
C_a\,  \left( K(y)+ \frac {y-1}{y} \right)^{a-i}
\\
&=&\sum_{i=0}^{m} \frac{t^{i}}{(y-1)^i} \sum_{j\ge 0} (y-1)^j T_{i,j} 
\end{eqnarray*}
where $T_{i,j}$ is independent of $y$.
The sum over $a =i, \ldots, m$ has been transformed into the sum over $j\ge 0$
using
$$
K(y)=\sum_{j\ge 0} K^{(j)}(1) \, \frac {(y-1)^j}{j!}
\quad \hbox{and} \quad \frac{y-1}y=\sum_{n\ge 0} (-1)^n (y-1)^{n+1}.
$$ 
Hence $T_{i,j}$ is a linear combination of $C_i, C_{i+1}, \ldots,
C_m$, with coefficients in $\qs[K(1), K'(1), \ldots, K^{(j)}(1) ]$.
In particular,
\beq\label{Ti0-expr}
T_{i,0}=\sum_{a=i}^m {a\choose i} C_a \, K(1)^{a-i}.
\eeq
Similarly, the left-hand side of~\eqref{eq-inv-K} reads
\begin{eqnarray*} 
\LHS &=& \sum_{i=0}^{m} \frac{t^{i}}{(y-1)^i}\sum_{a=i} ^m {a\choose i}
L_a(t;y) \left( K(y)+ \frac {y-1}{y} \right)^{a-i}\\
&=&\sum_{i=0}^{m} \frac{t^{i}}{(y-1)^i} \sum_{j\ge 0} (y-1)^j S_{i,j} 
\end{eqnarray*}
where $S_{i,j}$ is a polynomial of $ \GK[K(1), K'(1), \ldots,
K^{(j)}(1), t]$. In particular,
\beq\label{Si0}
S_{i,0}=\sum_{a=i}^m {a\choose i} L_a(t;1)\, K(1)^{a-i}.
\eeq
\begin{table}
\begin{tabular}{|l|l|}
$L_r(t;y)$ 	& Polynomial in $t$ and $\by$ with coeffs. in
$\GK\equiv \qs(q,\nu, w)$\\
$T_{i,j}$ & Linear combination of $C_i, C_{i+1},\ldots,  C_m$ with coeffs. in
$\qs[K(1), \ldots, K^{(j)}(1)]$\\
$S_{i,j}$ & Polynomial in $t$,  $K(1), \ldots, K^{(j)}(1)$ with
coeffs. in $\GK$
\\
$C_r$ (Lemma~\ref{lem:cr}) & Polynomial in $t$,  $K(1), \ldots, K^{(m-r)}(1)$ with
coeffs. in $\GK$
\end{tabular}
\vskip 4mm
\caption{The nature of the  series involved in this section.}
\label{table}
\end{table}

Table~\ref{table} summarizes the properties of the various series met in
this section. 
 The invariant equation~\eqref{eq-inv-K} can now be rewritten
$$
\sum_{i=0}^{m} \frac{t^{i}}{(y-1)^i} \sum_{j\ge 0} (y-1)^j
(S_{i,j} -T_{i,j} )
%
=0,
$$
where $S_{i,j}$ and $T_{i,j}$ do not involve $y$. In particular,
extracting the coefficient of $(y-1)^{-r}$, for $0\le r\le m$, gives
\beq\label{id1}
\sum_{i=r}^m t^{i-r} (S_{i,i-r} -T_{i,i-r} )=0.
\eeq
Recall the expression~\eqref{Ti0-expr} of $T_{i,0}$. Equivalently,
$$
T_{r,0}
= C_r +\sum_{a=r+1} ^m {a\choose r} C_a K(1)^{a-r}.
$$
Hence the  identity~\eqref{id1} gives
\beq\label{cr-rec}
 C_r 
=-\sum_{a=r+1} ^m {a\choose r} C_a K(1)^{a-r}
+ S_{r,0} +\sum_{i=r+1}^m t^{i-r} (S_{i,i-r} -T_{i,i-r} ).
\eeq
Recall that $S_{i,j}$  involves none of the series $C_a$. Moreover,
$T_{i,i-r}$ only involves the series $C_a$ if 
$a\ge i$ (see Table~\ref{table}). Hence the right-hand side of the
above identity only involves $C_a$ if $a\ge r+1$. Consequently, this
identity allows one to determine the coefficients 
$C_m, \ldots, C_1, C_0$  inductively in this 
order. Moreover, the properties of the series $S_{i,j}$ and $T_{i,j}$
 imply  that $C_r$ is of the form  $\Pol_r(K(1), K'(1),
\ldots, K^{(m-r)}(1),t)$, for some polynomial $\Pol_r(x_1, \ldots,
x_{m-r+1},  t)$ having coefficients in $\GK$.

\medskip
We  now address the properties of $\Pol_r$ stated in the lemma. 
 Let us first determine the constant term  of
$\Pol_r(x_1, \ldots,x_{m-r+1},t)$.  
In the recursive expression of $C_r$ given by~\eqref{cr-rec}, all terms coming from the second sum are
multiples of $t$, so that they  do not contribute to this constant term.
  Similarly, the
first sum is a multiple of $K(1)$, and  does not contribute either.
 The constant term of $\Pol_r$ thus reduces to the constant term of $S_{r,0}$,
 seen as a polynomial in $K(1), K'(1), \ldots, K^{(m-r)}(1)$ and $t$.  In
 sight of~\eqref{Si0} this gives
\beq\label{CT}
{\CT}\  \Pol_r(x_1, \ldots, x_{m-r+1},t)= L_r(0;1).
\eeq

Let us now determine the coefficient of the monomial $x_j$ in $\Pol_r(x_1, \ldots,
x_{m-r+1},t)$, for $1\le j \le m-r+1$. Again, the  second sum
of~\eqref{cr-rec}, being a multiple of $t$,  does not give any
such monomial.
In view of~\eqref{Si0}, the
series $S_{r,0}$ gives a term $(r+1) L_{r+1} (0;1)\, x_1$ but no linear
term $x_j$ for $j>1$.  The first sum occurring in~\eqref{cr-rec} is a multiple of $K(1)$. Hence
it does not give any linear term $x_j$ for $j>1$, but it does
give a term $-(r+1)\, (\CT\ \Pol_{r+1})\, x_1$, which, in view of~\eqref{CT},
cancels with the linear term in $x_1$ coming from $S_{r,0}$. This
proves that $\Pol_r$ contains no monomial $x_j$, for $j\ge 1$.
\end{proof}

\subsubsection{The final form of the invariant equation}
\label{sec:final-form}
Now return to the invariant equation~\eqref{eq-inv-K}, 
 replace $I(y)$ by its expression~\eqref{I-K} in terms of $K(y)$
and each series $C_r$ by its polynomial expression in terms of $K(1), \ldots,
K^{(m-r)}(1)$ and 
$t$. By forming the difference of the left-hand side and right-hand
side, one obtains   an equation of the form
\beq\label{eq-gen1}
\Pol(K(y), K(1), \ldots, K^{(m)}(1), t;y)=0
\eeq
where $\Pol(x_0, x_1, \ldots, x_{m+1},t;y)$ is a polynomial in $x_0,
x_1, \ldots, x_{m+1}$, $t$, $\by$ and $1/(y-1)$, having coefficients
in $\GK$.
In the case $q=1$, this is Eq.~\eqref{eq-K-1}. That is,
$$
\Pol(x_0, \ldots, x_{4},t;y)
=-27/2 (\nu-1)^2 
\Big( \nu (1-\by ) x_0 ^{2}+
 ( 2\,\by\nu\,t-\by+\by^2-\nu\,t ) x_0
-\nu  t x_1 +t (1-\by) ( w+\nu\,t+\by ) \Big).
$$

\begin{Lemma} \label{lem:x0}
 Consider $\Pol(x_0, x_1, \ldots, x_{m+1},t;y)$ as a polynomial in $t$
 and the $x_i$'s having coefficients in $\GK[ \by, 1/(y-1)]$.

Then the constant term of $\Pol$, that is, $\Pol(0, \ldots, 0,0;y)$, is zero. 
For $1\le j \le m+1$, the coefficient of the monomial $x_j$ in
$\Pol$ is also zero. The coefficient of the monomial $x_0$ is a non-zero
Laurent polynomial in $y$ with coefficients in $\GK$.
\end{Lemma}
  \begin{proof}
Set  $t=0$ in the identity~\eqref{eq-gen1}.  As $K(y)$ is a multiple of $t$, we
have $K(y)=K(1)=\cdots = K^{(m)}(1)=0$ when $t=0$. This gives $\Pol(0,
\ldots, 0,0;y)=0$, which precisely means that $\Pol$ has no constant term.

\smallskip
 For the second point, consider~\eqref{eq-inv-K}. As $I(y)$ (or more precisely,
 $K(y)$) only gives terms $x_0$ in $\Pol$,  the monomials
 $x_j$, for $j\ge 1$, can only come from the terms $C_r$. But  by
 Lemma~\ref{lem:cr}, $C_r$  contains no such monomial, so $\Pol$ does
 not either.

\smallskip
Finally,  the coefficient of $x_0$ in $\Pol$ can be read off
from~\eqref{eq-inv-K}: 
\begin{eqnarray*}
  [x_0]\Pol&=&\sum_{r=0}^m r  \left( \frac{y-1} y\right)^ {r-1} L_r(0;y) - 
\sum_{r=0}^m r \left( \frac{y-1} y\right)^ {r-1} \, \CT\ \Pol_r \,  \\ 
&=&\sum_{r=1}^m r \left( \frac{y-1} y\right)^ {r-1} (L_r(0;y)-L_r(0;1))
\end{eqnarray*}
by Lemma~\ref{lem:cr}.  Clearly, this is a
Laurent polynomial in $y$ with coefficients in $\GK$, which admits $1$
as a  root. 
In order to prove that this polynomial is non-zero, 
we will prove that
its  derivative with respect to $y$, evaluated at $y=1$, which is
$$
\frac{\partial L_1} {\partial y} (0;1),
$$
is non-zero.
Recall that  the functions $L_r(t;y)$ arise from the expansion in
$I(y)$ of the second invariant of Proposition~\ref{prop:inv-planaires}:
$$
J(y)=
D(y)^{m/2} \  T_m\left(
\frac{\be(4- q ) (\by-1)+( q+2\,\be ) I(y) -q}
{2\sqrt {D(y)}}\right)= \sum_{r=0}^m L_r(t;y)\, I(y)^r.
$$
A straightforward calculation (preferably done using Maple) gives
\begin{multline*}
L_1(0;y)=
- \frac {m(\nu+1)} 2\, {q }^{m/2} \, T_m \left( 
\frac{\be(4- q ) (\by-1) -q}
{2\sqrt {q}}
\right) \\
+
\frac { (\nu-1) \left( 4-q\right) ((1+\nu)\by -\nu)} 4\,{q }^{(m-1)/2} \,
T'_m  \left( \frac{\be(4- q ) (\by-1) -q}
{2\sqrt {q}} \right) 
%
 ,
\end{multline*}
so that
$$
\frac{\partial L_1} {\partial y} (0;1)=
\frac {\left( m-1 \right)  \left( 4-q \right) 
(\nu^2-1) } 4\, {q }^{(m-1)/2}
\,T_m'  \left(  -\frac{\sqrt {q}} 2\right)
-
\frac {\left( 4-q \right) ^{2} (\nu-1)^2} 8\,{q }^{m/2-1} \,
T_m''  \left( -\frac{\sqrt {q}} 2 \right) 
 .
$$
Recall that $q=2+2\cos(2k\pi/m)$, so that $-\sqrt q /2= -\cos
(k\pi/m)= \cos ((k+m)\pi/m)$. Moreover, $T_m(\cos \phi)= \cos (m \phi)$, and the
derivatives of $T_m$ at a point of the form $\cos \phi$ are easily
derived from this identity. This gives
$$
 T_m'  \left(  -\frac{\sqrt {q}} 2\right)=0 
\quad \hbox{and}\quad 
 T_m''  \left(  -\frac{\sqrt {q}} 2\right)= (-1)^{k+m+1}\frac{4 m^2}{4-q},
$$
which allows us to conclude that $\frac{\partial L_1} {\partial y}
(0;1)\not = 0$, so that the coefficient of $x_0$ in $\Pol$ is non-zero,
as claimed.
\end{proof}

\medskip

\noindent
{\em{Proof of Theorem~{\rm\ref{thm:alg-planaires}}.}}
The functional equation~\eqref{eq-gen1} involves a single catalytic
variable, $y$.  However, the case $q=1$ shows that its form may not be
suitable for a direct application of our algebraicity theorem
(see~\eqref{eq-K-1}). As it happens, a simple remedy for this is to
  reintroduce the original series $M(y)$. This is
the counterpart of the transformation of~\eqref{eq-K-1}
into~\eqref{eq-M-1} performed in the case $q=1$. 
So, in~\eqref{eq-gen1}, replace $K(y)$ by $t+twyq M(y)$ (where we 
now denote $M(y)=M(1,y)$) and replace
similarly each derivative $K^{(j)}(1)$ by its expression in terms of
$M$:
$$
  K(1)= t+twqM(1), \quad \hbox{and for } 1\le j \le m, \quad
\quad K^{(j)}(1)= twq (j M^{(j-1)}(1)+M^{(j)}(1)).
$$
Observe the factor $t$ in all these expressions.
 According to Lemma~\ref{lem:x0}, $\Pol(x_0, \ldots, x_{m+1}, t;y)$, seen as a
 polynomial in $t$ 
and the $x_i$'s, has no constant term. This implies that, once the $K$'s
have been replaced by $M$'s,  the resulting equation  contains a
factor $t$: divide it by $t$ to obtain 
\beq\label{eq-gen2}
\Pol'(M(y), M(1), \ldots, M^{(m)}(1), t;y)=0,
\eeq
where 
$$
\Pol'(x_0, \ldots,x_{m+1} ,t;y)=
 \frac 1 t\,  \Pol(t+twy q x_0, t+twq x_1, twq(x_1+x_2)\ldots,
 twq (mx_m+x_{m+1}),t;y).
$$
Have we at last reached  an equation of the
form~\eqref{generic-M}, to which we could apply our algebraicity
theorem? If this were the case, $\Pol'(x_0, \ldots,x_{m+1} 
,0;y)$ should reduce to $x_0-1$.
We have 
\begin{multline*}
  \Pol'(x_0, \ldots,x_{m+1} ,0;y)= (1+wyq x_0) [x_0] \Pol(x_0,\ldots,t;y)
+ (1+wq x_1) [x_1] \Pol(x_0,\ldots,t;y)
\\+ \cdots
  +wq(mx_m+x_{m+1})[x_{m+1}] \Pol(x_0,\ldots,t;y)
+ [t] \Pol(x_0, \ldots ,t;y).
\end{multline*}
But $[x_1]\Pol = \cdots = [x_{m+1}] \Pol=0$ by  Lemma~\ref{lem:x0}, 
so that $\Pol'(x_0, \ldots,x_{m+1} ,0;y)$ reads $a_0(y)x_0+b_0(y)$, where
$a_0(y)=wyq  [x_0] \Pol $ is a non-zero Laurent polynomial in $y$ with
coefficients in $\GK$ (by  Lemma~\ref{lem:x0} again)
and $b_0 \in \GK(y)$.
Hence dividing~\eqref{eq-gen2} by $a_0(y)$ finally gives
$$
M(y)= M_0(y)+ tP_1(M(y), M(1), \ldots, M^{(m)}(1),t;y)
$$
where  $M_0(y) \in \GK(y)$ and $P_1(x_0, \ldots, x_{m+1},t;y)$ is a
polynomial in $t$ and the 
$x_j$'s, and a  rational function in $y$, with coefficients in
$\GK$. By setting $t=0$, we obtain $M_0(y)= 1$ (for the one-vertex map).
 Upon writing
$$
\frac{M^{(i)}(1)}{i!}= \Delta^{i}M(y)- (y-1)\Delta^{i+1}M(y),
$$
where 
$$
\Delta F(y)= \frac{F(y)-F(1)}{y-1},
$$
the equation reads
$$
M(y)= 1+ tP(M(y), \Delta M(y), \ldots, \Delta^{m+1} M(y), t;y)
$$
where $P(x_0, \ldots, x_{m+1},t;y)$ is a
polynomial in $t$ and the 
$x_j$'s, and a  rational function in $y$, with coefficients in
$\GK$.
 Our general algebraicity theorem (Theorem~\ref{generic-thm}) 
implies that $M(y)\equiv M(q,\nu,t,w;1,y)$ is algebraic 
over $\qs(q, \nu,t,w,y)$. Using the identity~\eqref{Potts-Tutte-gfs},
we conclude that the Tutte \gf\ $\gtM(\mu, \nu,  w,z; 1,y)$ is algebraic over
$\qs(\mu, \nu, w,z,y)$ when $(\mu-1)(\nu-1)=q$.
Now by the duality property~\eqref{eq:duality-Tutte-poly},  
$\gtM(\mu, \nu, w,z; 1,y)=\gtM(\nu, \mu, z, w; y,1)$. But the
condition  $(\mu-1)(\nu-1)=q$ is symmetric in $\mu$ and $\nu$, and hence
$\gtM(\mu, \nu, z, w; x,1)$ is algebraic as well, under the same assumption.
Returning to the functional equation~\eqref{eq:tM}, this
implies that $\gtM(\mu, \nu,  w,z; x,y)$ is algebraic. A second
application of~\eqref{Potts-Tutte-gfs} yields the algebraicity of
the Potts series $M(q, \nu, t, w; x,y)$.
\qed

\section{Algebraicity for colored  triangulations}
\label{sec:alg-triang}
We now prove a second algebraicity theorem, which  applies to the
\emm quasi-triangulations, of Section~\ref{sec:eq-quasi-triang}.   We
consider the Potts \gf\  $Q(x,y)\equiv Q(q, 
\nu, t ,w,z ; x,y)$ of these maps, defined by~\eqref{Q-ser-def}.  This series
is characterized by the functional equation~\eqref{eq:Q}.

\begin{Theorem}\label{thm:alg-triang}
 Let $q\not = 0,4$ be of the form $2+2\cos j\pi/m$ for two integers
 $j$ and $m$.
Then the series $Q(q, \nu, t ,w,z ; x,y)$
is algebraic over $\qs(q, \nu, t, w,z,x,y)$. 
\end{Theorem}

It follows that the \gf\ of properly $q$-colored triangulations,
studied by Tutte in his long series of papers from 1973 to 1984, is
algebraic at these values of $q$. Indeed, this series is, with our notation,
$q [y^3] Q(q,0,1,1,z;0,y)$.

\smallskip
\noindent{\bf Caveat.}
 The series $Q( q, \nu, t, w,z ; x,y)$ is \emm not, algebraic for a
generic value of $q$. Otherwise, the series $\gtQ( \mu, \nu,  w ,z;
x,y)$ counting quasi-triangulations weighted by their Tutte polynomial
(which is related to $Q$ by the change of
variables~\eqref{Potts-Tutte-gfs}) would be algebraic over $\qs(\mu,
\nu,  w ,z,  x,y)$, for generic values of $\mu$ and $\nu$. However,
it is known that~\cite{mullin-boisees}: 
$$
[y^2]\gtQ(1,1,w,1;0,y)=\sum_{n> 0 } \frac{1} { 2n(n+1)}{{2n}
  \choose n}{{4n-2} \choose {2n-1}}\,  w^{n},
$$
and the asymptotic behaviour of the $n^{\hbox{th}}$ coefficient, being
$\kappa\, 64^n n^{-3}$, prevents this series from being
algebraic~\cite{flajolet-context-free}. 
Again, the above series counts near-triangulations of outer degree 2
 enriched with a spanning tree.

Another way to establish the transcendence of $Q$ is to use
$$
[y^2] \frac{\partial Q}{\partial q} (1,0, 1, w,1 ;0,y)=
\sum_{n\ge 0 } (-1)^{n}
\frac{2\,  (3n)!}{n!(n+1)! (n+2)!}\,  w^{n+1}, 
$$
plus the fact that  the $n^{\hbox{th}}$ coefficient of this series behaves like
$\kappa\, 27^n n^{-4}$. The above identity was proved
by Tutte in 1973~\cite{lambda12}. It is now known that the numbers 
${\partial  \Ppol_M}/{\partial q}\,(1,0)$ which are involved in the above series
count (up to a sign)  
\emm bipolar orientations, of $M$ (see, e.g.,~\cite{greene-zaslavsky,lass-orientations}).

\medskip

As explained at the beginning of Section~\ref{sec:inv-triang}, the variables $w$ and
$z$ are redundant. Hence it suffices to prove Theorem~\ref{thm:alg-triang} for
$w=z=1$. We thus set $w=z=1$ and denote $Q(q,\nu,t,1,1;x,y)$ by
$Q(q,\nu, t; x,y)$. The conditions on $q$ imply that there exist two
coprime integers $k$ and $m$ such that $0<2k<m$ and $q=2+2 \cos
2k\pi/m$. Corollary~\ref{coro:eq-inv-Q} thus applies, and gives a polynomial
equation in $I(y)$ involving $m+1$ unknown series $C_r$. We call this
equation the \emm invariant equation,. 
From this point, we prove 
Theorem~\ref{thm:alg-triang} in two steps: we first show that the
series $C_r$ can be expressed in 
terms of the $y$-derivatives of $Q(0,y)$, evaluated at $y=0$; then, we
prove that, once $I(y)$ and each $C_r$ are replaced, in the invariant equation, by
their expressions in terms of $Q$, our general algebraicity theorem applies. That is to say, the equation
satisfied by $Q(0,y)$ has the form~\eqref{main-eq},  with $u$ replaced by
$y$. This second step is more delicate than the first. The whole proof
is also more complicated than in the case of general planar 
maps, due to the pole of order~2 found in the invariant $I(y)$ at
$y=0$ (see Proposition~\ref{prop:inv-triang}).

Before we study the general case, let us examine thoroughly  a simple
example: $q=1$. We refer  the reader who would like to see more explicit
cases to Sections~\ref{sec:two-triang} and~\ref{sec:three-triang} (respectively
devoted to $q=2$ and $q=3$).

\subsection{A simple example: one-colored triangulations}
\label{sec:one-color-triang}
Take $k=1$ and $m=3$, so that the number of colors is 
$q=1$. Of course, all edges of a $1$-colored
map are monochromatic, so that the variable $\nu $ becomes redundant,
but we keep it for the sake of generality (a degeneracy  actually
occurs if we set $\nu=1$ at this stage). 

The third Chebyshev polynomial is $T_3(x)= 4x^3-3x$. The invariant
equation~\eqref{eq-inv-triang}  thus reads
\beq\label{eq-inv-triangq1}
\frac 1 2 N(y)^3 -\frac 3 2 N(y) D(y)-\sum _{r=0}^{3} C_r\, (tI(y))^r=0,
\eeq
with 
$$
N(y)=
 3 (\nu-1)  t\by+  \nu t I(y)-(\nu-1),
\quad 
I(y)= ty Q ( 0,y ) -\by +t \by^2
$$
and
$$
 D(y)= \nu^2 {t}^{2} I(y)^{2}+(\nu-1)  \left( 4\nu-3 \right)t
 I(y)-3\nu(\nu-1) {t}^{3} +(\nu-1)^2. 
$$
We write
$$
t\, I(y)= K(y) -t\by +t^2\by^2
\quad \hbox{where} \quad K(y)=t^2y Q(0,y).
$$
This is not crucial in this simple case, but will be
convenient in the general case. 

Recall that the  series $C_r$ depend on $\nu$ and  $t$, but not on
$y$. Expand the left-hand side of~\eqref{eq-inv-triangq1} around
$y=0$: the first non-trivial term is $O(y^{-6})$, and one obtains:
$$
-{t}^{6} \left( {\nu}^{3}+C_{{3}} \right) {y}^{-6}+O ( {y}^{-5} ) 
=0,
$$
from which we  determine $C_3$ explicitly:
\beq\label{C3-sol-triang}
C_3=-\nu^3.
\eeq
By pushing the expansion of~\eqref{eq-inv-triangq1} around $y=0$ up to the
term $y^0$, setting $K(0)=0$, and extracting the coefficients of
$y^{-4}$, $y^{-2}$ and finally $y^{-0}$, we find explicit expressions
of the other three series $C_r$: 
\begin{eqnarray} 
\nonumber
C_2&=&
3/2\,\nu\, ( \nu-1 )  ( 5\,\nu-6 ),\\ 
C_1&=&9/2\,{\nu}^{2} ( \nu-1 ) {t}^{3}+3/2\, ( 4\,\nu-3
 )  ( \nu-1 ) ^{2},\label{Cr-q=1-triang}
\\ 
\nonumber C_0&=&27/2\,(\nu-1 ) ^{2}
t\, K'  ( 0 ) 
-27/4\,\nu\, ( \nu-1 ) ^{2}
{t}^{2} K'' ( 0)
 -9/2\,\nu\, ( \nu-1 ) ^{2}{t}^{3}+ ( \nu-1 ) ^{3}
.
\end{eqnarray}
Observe that we have not exploited the fact that  the coefficients of $y^{-5}$, $y^{-3}$,
$y^{-1}$  must be zero as well.

Now replace in~\eqref{eq-inv-triangq1} each series $C_r$ by its
expression. This gives 
\beq\label{eq-K-1-triang}
27/4(\nu-1)^2   \Big( 
-2\nu\,  K ( y ) ^{2}
+2\,t\by ( 1-\nu t\by ) K ( y ) 
-2t\,K'  ( 0 ) + 
\nu{t}^{2} K''  ( 0 ) 
+2\nu \,{t}^{4}\by
\Big) =0.
\eeq
This equation involves a single catalytic variable, $y$. However, it
cannot be immediately written in the form~\eqref{main-eq}: when $t=0$,
the expression between parentheses contains a quadratic term $K(y)^2$,
which is absent from~\eqref{main-eq}.

Let us replace $K(y)$ by $t^2y Q(y)$, $K'(0)$ by $t^2 Q(0)$  and $K''(0)$
by $2t^2Q'(0) $, where
$Q(y)\equiv Q(0,y)$.  More factors come out, including a factor
$t^3$. Precisely, the equation now reads:
\beq\label{eq-M-1-triang}
-27/2(\nu-1)^2t^3
\Big( 
t \nu {y}^{2}  Q ( y ) ^{2}
- ( 1-\nu t\by ) Q ( y ) +Q ( 0 ) -\nu t\,Q'(0) -\nu t\by
\Big) =0,
\eeq
or, after dividing by $27/2(\nu-1)^2 t^3$ and isolating the term $Q(y)$,
\beq\label{1cat-planaires1-triang}
Q(y)= Q(0)+t\nu y^2 Q(y)^2+  t\nu\,  \frac {Q(y)-1-yQ'(0)}y  .
\eeq
Now replace $Q(0)\equiv Q(0,0)$ by its value 1. The resulting equation
has the form~\eqref{main-eq} (with 
$u$ replaced by $y$), so that Theorem~\ref{generic-thm} applies: The series
$Q(y)\equiv Q(1, \nu, t;0,y)$
is algebraic. The algebraicity of $Q(1,\nu, t;x,y)$ follows, as
explained at the end of this  section.
The experts will have recognized
in~\eqref{1cat-planaires1-triang} the standard functional equation
obtained by deleting recursively the root-edge of a
 near-triangulation
 (that is, a map in which all internal faces have
degree 3)~\cite{bender-canfield,mullin-nemeth-schellenberg}.

\subsection{The general case}
%

We now want to  prove that the treatment we have
applied above to~\eqref{eq-inv-triang} in the case $q=1$ can
  be applied for all values $q=2+2\cos 2 k\pi/m$. More precisely:
\begin{itemize}
\item expanding~\eqref{eq-inv-triang} in powers of $y$ and extracting
  the coefficient of $y^{-2r}$ 
provides an expression of the series $C_r$, for $0\le r \le m$, as a
polynomial in
$t$,  $ K'(0), \ldots,  K^{(2m-2r)}(0)$, where $K(y)=t^2yq Q ( 0,y )$,
with coefficients in
$\GK:=\qs(q,\nu)$;
\item after expressing in~\eqref{eq-inv-triang} the invariant $I(y)$
  and each series $C_r$ in terms of $K$, then in terms of $Q$, setting
  $Q(0,0)=1$,  and   finally dividing by $t^3$ and by a non-zero
  element of $\GK$, the resulting equation can be written in the form
\beq\label{generic-Q}
Q(y)=1 + t \, 
P\left(Q(y),  \Delta Q(y), \Delta^{(2)} Q(y), \ldots, 
\Delta^{2m}Q(y), t;y\right),
\eeq
where $Q(y)\equiv Q(0,y)$, 
$\displaystyle ~ \Delta F(y)= \frac{F(y)-F(0)}{y},$
and $P(x_0,  x_1, \ldots, x_{2m}, t;y)$ is a polynomial in its
first $2m+2$ variables  and a Laurent polynomial in $y$, having
coefficients in $\GK$. 
\end{itemize}
One can then apply Theorem~\ref{generic-thm}, and conclude that the
\gf\ $Q(y)\equiv Q(q,\nu,t;0,y)$ is algebraic. We finally return to
the original equation~\eqref{eq:Q} to prove that the
more general series $Q(q,\nu,t;x,y)$ is also
algebraic.

\medskip
\noindent{\bf Remark.} As suggested by the case $q=1$
(Section~\ref{sec:one-color-triang}), the series $C_r$ can be
expressed in terms of $K'(0), \ldots, K^{(2m-2r-4)}(0)$ only, but we
do not need so much precision here.

\subsubsection{Determination of the series $C_r$}
\label{sec:Cr-triang}

It will be convenient to write
\beq\label{I-K-triang}
t I(y)= K(y) -t\by +t^2\by^2,
\eeq
where 
$$
K(y)=t^2yq\;Q(0,y).
$$
Consider  the invariant equation~\eqref{eq-inv-triang}.
With the notation~\eqref{Tma} introduced in Section~\ref{sec:alg-maps}
for Chebyshev polynomials,
the left-hand side of~\eqref{eq-inv-triang} reads 
\begin{multline}
 \LHS=\sum_{a=0} ^\ell T_m^{(a)}\; 2^{-(2a+\eps)}
\left( 
 \be ( 4-q )t \by+ q \nu t I(y)+\be(q-2)
\right)^{2a+\eps} \label{LHS-triang}\\
\left(
q\nu^2 {t}^{2} I(y)^{2}
+\be  \left( 4\be +q \right)t I(y)
-q\be\nu {t}^{3}  \left( 4-q \right) +\be^2
\right)^{\ell-a}.
\end{multline}
Using~\eqref{I-K-triang}, this can be written as a polynomial in
$t\by$, $K(y)$ and $t^3$, of degree $2m$
in $t\by$, having  coefficients in $\GK=\qs(q,\nu)$ (recall that $\be
=\nu-1$). We write this expression as
\beq\label{Li-def-triang}
\LHS= \sum_{i=0}^{2m} (t\by)^i L_i(K(y),t^3),
\eeq
where $L_i(K(y),t^3)$ is a polynomial in $K(y)$ and $t^3$ with
coefficients in $\GK$.
Similarly, the right-hand side of~\eqref{eq-inv-triang} appears as a
polynomial in $C_0,
\ldots, C_m, t\by, K(y)$ with coefficients in $\qs$. It is easily seen
that, when one expands it in $t\by$,  the
coefficient of $(t\by)^i$ only involves $C_{\lceil i/2\rceil}, \ldots
, C_m,K(y)$. More precisely,
\beq\label{RHS-triang}
\RHS= \sum _{a=0}^{m } C_a\,\left(K(y)-t\by+t^2\by^2\right)^a
=  \sum_{i=0}^{2m} (t\by)^i R_i(C_{\lceil i/2\rceil}, \ldots, C_m,K(y))
\eeq
where
\beq\label{Ri-def-triang}
R_i(C_{\lceil i/2\rceil}, \ldots , C_m,K(y))=(-1)^i \sum_{a=\lceil
  i/2\rceil }^m\sum_{b=\lceil i/2\rceil}^{\min (i,a)}  {a\choose
  b}{b\choose {i-b}} C_a\,   K(y)^{a-b}.
\eeq
We  thus write the invariant equation as follows:
\beq \label{eq-inv-triang-K}
\sum_{i=0}^{2m} (t\by)^i L_i(K(y),t^3)
=\sum_{i=0}^{2m} (t\by)^i R_i(C_{\lceil i/2\rceil}, \ldots , C_m,K(y)).
\eeq

\begin{table}
\begin{tabular}{|l|l|}
$L_i(x_0,t)$ 	& Polynomial in $x_0$ and $t$ with coeffs. in
$\GK\equiv \qs(q,\nu)$\\
$T_{i,j}$ & Linear combination of  $C_{\lceil i/2\rceil}, 
\ldots,
 C_m$ with coeffs. in $\qs[K'(0), \ldots, K^{(j)}(0)]$
\\
$S_{i,j}$ & Polynomial in  $  K'(0), \ldots,K^{(j)}(0), t^3$ with
coeffs. in $\GK$
\\
$C_r$ (Lemma~\ref{lem:cr-triang}) & Polynomial in $t$,  $K'(0), \ldots,
K^{(2m-2r)}(0)$ with coeffs. in $\GK$
\end{tabular}
\vskip 4mm
\caption{The nature of the series involved in this section.}
\label{table-triang}
\end{table}

\begin{Lemma}\label{lem:cr-triang}
The series $C_m, C_{m-1}, \ldots , C_0$ can be determined
  inductively by
  expanding~\eqref{eq-inv-triang-K} in powers of   $y$ and 
extracting the  coefficients of $y^{-2m}, y^{-2m+2},\ldots, y^2, y^0$.
%
This gives, for   $0\le r\le m$, 
$$
C_r=\Pol_r( K'(0), \ldots, K^{(2m-2r)}(0),t)
$$
for some polynomial $\Pol_r(x_1, \ldots, x_{2m-2r},t)$ 
 having coefficients in $\GK=\qs(q,\nu)$.
Moreover, $\Pol_r$
contains no monomial $x_j$,  for $1\le j \le 2m-2r$, and no monomial
$tx_j$  for $2\le j \le 2m-2r$. Finally, denoting by $c_r$ the
constant term of $\Pol_r$, we have
\begin{eqnarray}
\nonumber 
\sum_{a=0}^m c_a (z^2-z)^a&=& \sum_{a=0}^{2m} z^a L_a(0,0)\\
&=&
\label{cr-CT}
\tilde D(z)^{m/2} \ 
T_m\left(
\frac{
 \be ( 4-q )z+ q \nu (z^2-z)+\be(q-2)}
{2\sqrt {\tilde D(z)}}
\right)
\end{eqnarray}
where $\beta=\nu-1$ and
$
\tilde D(z)= 
q\nu^2 (z^2-z)^2
+\be  \left( 4\be +q \right)(z^2-z)  +\be^2.
$
\end{Lemma}

All statements of the lemma, apart from the last one, can be checked
at once in the case $q=1$,
using~(\ref{C3-sol-triang}--\ref{Cr-q=1-triang}). 

\begin{proof}
In  the right-hand side of~\eqref{eq-inv-triang-K}, expand
$R_i(C_{\lceil i/2\rceil}, \ldots, C_m,K(y))$ in powers of $y$, using
$$
K(y)=
\sum_{j\ge 1} K^{(j)}(0) \, \frac {y^j}{j!}
$$ 
(since $K(0)=0$). 
This gives
\beq\label{RHS-triang-T}
\RHS\ =\ \sum_{i=0}^{2m} (t\by)^i\, R_i(C_{\lceil i/2\rceil}, \ldots, C_m,K(y))
\ =\ \sum_{i=0}^{2m}  (t\by)^i \sum_{j\ge 0} y^j T_{i,j}, 
\eeq
where $T_{i,j}$ is a linear combination of $C_{\lceil i/2\rceil}, 
\ldots,  C_m$, with coefficients in $\qs[K'(0), \ldots, K^{(j)}(0)]$.
In particular,  one derives from~\eqref{Ri-def-triang} that
\beq\label{Ti0-expr-triang}
T_{2i,0} =R_i(C_{\lceil i/2\rceil}, \ldots, C_m,0)
=\sum_{a=i}^m {{a}\choose {2i-a}} C_a .
\eeq
Similarly, expanding $L_i(K(y),t^3)$ in the left-hand side of~\eqref{eq-inv-triang-K} gives
\beq\label{LHS-triang-S}
\LHS\ = \ 
\sum_{i=0}^{2m} (t\by)^i\, L_i(K(y),t^3)
\ =\ \sum_{i=0}^{2m} (t\by)^i\, \sum_{j\ge 0} y^j S_{i,j}, 
\eeq
where $S_{i,j}$ is a polynomial in $  K'(0), \ldots,K^{(j)}(0), t^3$ with  coefficients in $\GK$.
In particular,
\beq\label{Si0-triang}
S_{i,0}=L_i(0,t^3).
\eeq
Table~\ref{table-triang} summarizes the properties of the various
series met in this section. 
 The invariant equation~\eqref{eq-inv-triang-K} can now be rewritten
$$
\sum_{i=0}^{2m} 
(t\by)^i \sum_{j\ge 0} y^j
(S_{i,j} -T_{i,j} )
%
=0,
$$
where $S_{i,j}$ and $T_{i,j}$ are independent of $y$. In particular,
extracting the coefficient of $\by^{2r}$, for $0\le r\le m$, gives
\beq\label{id1-triang}
\sum_{i=2r}^{2m} t^{i-2r} (S_{i,i-2r} -T_{i,i-2r} )=0.
\eeq
Recall the expression~\eqref{Ti0-expr-triang} of $T_{2i,0}$. Equivalently,
$$
T_{2r,0}=C_r+ \sum_{a=r+1}^m {{a}\choose {2r-a}} C_a .
$$
Hence the  identity~\eqref{id1-triang} gives
\beq\label{cr-rec-triang}
 C_r 
=- \sum_{a=r+1}^m {{a}\choose {2r-a}} C_a
+ S_{2r,0} +
     \sum_{i=2r+1}^{2m} t^{i-2r} (S_{i,i-2r} -T_{i,i-2r} ).
\eeq
Recall that $S_{i,j}$  involves none of the series $C_a$. Moreover,
$T_{i,i-2r}$ only involves the series $C_a$ if
$a\ge i/2$ (see Table~\ref{table-triang}). Hence the right-hand side of
the above identity only involves $C_a$ if $a\ge r+1$. Consequently,
this identity allows one to determine the coefficients
$C_m, \ldots, C_1, C_0$  inductively in this 
order. Moreover, the properties of the series $S_{i,j}$ and $T_{i,j}$
 imply  that $C_r$ is of the form $\Pol_r( K'(0),
\ldots, K^{(2m-2r)}(0),t)$, for some polynomial $\Pol_r(x_1, \ldots,
x_{2m-2r},  t)$ having coefficients in $\GK$.

\medskip
We  now address the properties of $\Pol_r$ stated in the
lemma.  Let us first prove that the coefficient  of
the monomial $x_j$ in  $\Pol_r(x_1, \ldots, x_{2m-2r},t)$, for $j\ge 1$, is zero.
In the recursive expression of $C_r$ given by~\eqref{cr-rec-triang},
all terms coming from the second sum are 
multiples of $t$, so that they  do not contribute to this coefficient. 
  In  sight of~\eqref{Si0-triang}, the coefficient of $x_j$ in
  $S_{2r,0}$ (seen as a polynomial in  $K'(0),
\ldots, K^{(2m)}(0),t$) is 0, and
  thus by a decreasing induction on $r=m, \ldots,0$ we conclude
  from~\eqref{cr-rec-triang} that the coefficient of $x_j$ in $\Pol_r$ is
  zero.

Let us now prove, by a decreasing induction on $r$, that the coefficient of the monomial $tx_j$ in
$\Pol_r(x_1, \ldots, x_{2m-2r},t)$, for $2\le j \le 2m-2r$, is zero. Again, the
term   $S_{2r,0}=L_{2r}(0,t^3)$ does not contain any such monomial. In
the second sum  
of~\eqref{cr-rec-triang}, the monomial $tx_j$ may only come from the term
$t(S_{2r+1,1}- T_{2r+1,1})$
obtained for $i=2r+1$. Recall that both $S_{i,j}$ and $T_{i,j}$ are
  obtained by extracting the coefficient of $y^j$ in certain
  polynomial expressions  in $K(y)$ (see~\eqref{LHS-triang-S} and~\eqref{RHS-triang-T}). Hence $S_{i,1}$ and $T_{i,1}$
  may contain some terms $x_1$, but no term $x_j$ for $j\ge  2$ 
(because $K^{(j)}(0)$ always comes with a power $y^j$). 
Consequently, $t(S_{2r+1,1}- R_{2r+1,1})$ does not contain any
  terms $tx_j$ for $j \ge 2$, and we finally conclude from~\eqref{cr-rec-triang}
  that the coefficient of $tx_j$ in $\Pol_r$ is   zero for $j\ge 2$.

 Let us finally prove the last statement of Lemma~\ref{lem:cr-triang},
which deals with
the constant term $c_r$ of $\Pol_r$. From~\eqref{cr-rec-triang}
and~\eqref{Si0-triang}, one derives that, for $r=0, \ldots, m$,
 $$
c_r  =- \sum_{a=r+1}^m {{a}\choose {2r-a}} c_a
+ L_{2r} (0,0).
$$
This means  
that the following  two polynomials in $z$,
$$
\sum_{a=0}^m c_a (z^2-z)^a
\quad \hbox{ and } \quad 
\sum_{a=0}^{2m} z^a L_a (0,0)
$$ 
 have the same  even part.
It is easy to see that a polynomial in $z^2-z$
is completely determined by its even part. Thus, in order to prove
that the above polynomials  coincide,  it
suffices to prove that the second one is also a
polynomial in $z^2-z$, that is, that 
\beq\label{even-odd}
\sum_{a=0}^{2m} z^a L_a (0,0)= \sum_{a=0}^{2m} (1-z)^a L_a (0,0).
\eeq
Let us use the expression of $\sum_{a=0}^{2m} z^a L_a (0,0)$ given
in the lemma, which follows from the definition~\eqref{Li-def-triang}
of $L_i$. We observe that $\tilde D(z)$ is a polynomial in $(z^2-z)$. Thus, in
order to prove~\eqref{even-odd}, it suffices to prove that
$T_m(x_1)=T_m(x_2)$, where
$$
x_1= \frac{  \be ( 4-q )z+ q \nu (z^2-z)+\be(q-2)}
{2\sqrt {\tilde D(z)}}
\quad \hbox{ and } \quad 
x_2= \frac{  \be ( 4-q )(1-z)+ q \nu (z^2-z)+\be(q-2)}
{2\sqrt {\tilde D(z)}}.
$$
By Proposition~\ref{prop:source}, 
the bivariate polynomial $T_m(z_1)-T_m(z_2)$ is
divisible  by
$z_1^2+z_2^2-(q-2)z_1z_2-\sin^2(2k\pi/m)=z_1^2+z_2^2-(q-2)z_1z_2-q(4-q)/4$. 
But 
$
x_1^2+x_2^2-(q-2)x_1x_2-q(4-q)/4$ is found to be $0$, so that
$T_m(x_1)=T_m(x_2)$. This concludes the proof of the lemma.
\end{proof}

\subsubsection{The final form of the invariant equation}
\label{sec:final-form-triang}
Now return to the invariant equation~\eqref{eq-inv-triang-K}, and replace each
$C_r$ by its polynomial expression in terms of $K'(0), \ldots, K^{(2m-2r)}(0)$ and
$t$. By forming the difference of the left-hand side and right-hand
side, one obtains   an equation of the form
\beq\label{eq-gen1-triang}
\Pol(K(y), K'(0), \ldots, K^{(2m)}(0), t,t\by)=0,
\eeq
where $\Pol(x_0, x_1, \ldots, x_{2m},t,z) \in \GK[x_0,
x_1, \ldots, x_{2m}, t, z]$.
In the case $q=1$, this is Eq.~\eqref{eq-K-1-triang}. That is,
$$
\Pol(x_0, \ldots, x_{6},t,z)=
27/4(\nu-1)^2   \Big( 
-2\nu\,  x_0 ^{2}
+2\,z( 1-\nu z ) x_0
-2t\,x_1 + 
\nu{t}^{2} x_2
+2\nu \,{t}^{3}z
\Big).
$$

\begin{Lemma} \label{lem:x0-triang}
 In the polynomial $\Pol\equiv \Pol(x_0, x_1, \ldots, x_{2m},t,z)$:
 \begin{itemize}
\item[$(i)$]  for $1\le j \le 2m$, 
$$[x_j] \Pol= 0,$$
\item[$(ii)$]  for $2\le j \le 2m$,
$$
[tx_j] \Pol= 0,
$$
\item[$(iii)$]  for $1\le j \le 2m$,
$$
[z x_j] \Pol= 0,
$$
 \item[$(iv)$] the constant term is zero,
\item[$(v)$] the  coefficients of the monomials $t$,  $z$, $x_0$, $t^2$, $z^2$,
  $tz$, $tx_0$,
  $t^2z$, $tz^2$, $z^3$  are zero,
\item[$(vi)$]  finally,
$$
q[z x_0] \Pol= - q[tx_1] \Pol - [t^3]\Pol 
=\frac{m^2} 2q (4-q)(\nu-1)^{m-1}\not =0. 
$$
 \end{itemize}
 \end{Lemma}
  \begin{proof}
 Consider  the functional
 equation~\eqref{eq-inv-triang-K}, with each $C_r$ replaced by its
 expression in terms of $K'(0), \ldots, K^{(2m)}(0)$ and $t$. As $K(y)$ only give terms $x_0$,  the monomials
 $x_j$, for $j\ge 1$,  only occur in the right-hand side, and more
 precisely in the terms $C_r$. But  by
 Lemma~\ref{lem:cr-triang}, $C_r$  contains no such monomial, so $\Pol$ does
 not either. By a  similar argument, no monomial $tx_j$ occurs in
 $\Pol$ for $j\ge 2$. Finally, a monomial $zx_j$ with $j\ge 1$ could
 only arise from the term $C_1(K(y)-t\by +t^2\by^2)$ in~\eqref{RHS-triang}. But this is not the case, as $C_1$
 contains no monomial $x_j$.  We have  proved the first three points
 of the lemma.

Now recall that $K(y)=qt^2yQ(y)$ where $Q(y)\equiv Q(0,y)$, and that
$Q(y)=1+O(ty)$. Moreover, $Q^{(i)}(0)/i!$ counts colored near-triangulations
with outer degree $i$. These maps have at least $\lceil i/2\rceil$
edges. Hence
$$
\begin{array}{lll}
  K(y)&=& qt^2y+O(t^3),\\
K'(0)&=& qt^2Q(0)=t^2q, \\
K''(0)&=&2qt^2Q'(0)=O(t^3), \\
K^{(3)}(0)&=&3qt^2Q''(0)=O(t^3), \\
K^{(4)}(0)&=&4qt^2Q^{(3)}(0)=O(t^4),\\
&\ldots&,\\
K^{(2m)}(0)&=&2mq t^2 Q^{(2m-1)}(0)=O(t^{m+2}).
\end{array}
$$
Using this, let us 
expand~\eqref{eq-gen1-triang} in powers of $t$ to third order: 
\begin{multline}\label{exp1}
\Pol(K(y), K'(0), \ldots, K^{(2m)}(0), t,t\by)=
\CT\, \Pol + t \big( [t] \Pol +\by  [z] \Pol\big)\\
+ t^2 \big( qy [x_0] \Pol + q[x_1] \Pol + [t^2] \Pol +\by [tz] \Pol
+\by^2  [z^2] \Pol\big)
+ O(t^3) =0.
\end{multline}
Recall that the coefficients of $\Pol$ belong to
$\GK=\qs(q,\nu)$, and that we have established in the first part of
the proof that $[x_1] \Pol=0$. 
Hence, the above expansion, followed by an expansion
in powers of $y$, gives at once
$$
\CT\, \Pol = [t] \Pol=  [z] \Pol= [x_0] \Pol = [t^2] \Pol=[tz] \Pol= [z^2] \Pol=0.
$$
 This proves $(iv)$ and part of $(v)$.  Let us push the expansion~\eqref{exp1} one step
further, using  $[x_0]\Pol=[x_2]\Pol=[x_3]\Pol=[zx_1]\Pol=0$
(which we have proved above):
\begin{multline}\label{exp2}
0=\Pol(K(y), K'(0), \ldots, K^{(2m)}(0), t,t\by)=\\
t^3 \big( qy [tx_0] \Pol + q [zx_0] \Pol +q [tx_1] \Pol 
+ [t^3] \Pol + \by [t^2z] \Pol + \by^2 [tz^2] \Pol + \by^3 [z^3] \Pol
\big)
+O(t^4).
\end{multline}
Expanding the coefficient of $t^3$ in powers of $y$ gives
$$
[tx_0] \Pol= [t^2z] \Pol=  [tz^2] \Pol= [z^3] \Pol,
$$
which completes the proof of $(v)$, and
$$
q[zx_0] \Pol + q[tx_1] \Pol + [t^3] \Pol =0,
$$
which proves part of $(vi)$.
Finally, we read off from~\eqref{eq-inv-triang-K}
and~\eqref{RHS-triang} that
$$
[zx_0] \Pol= [x_0] L_1(x_0,0)
+2 \CT\, \Pol_2.
$$
The coefficient of $x_0$ in $L_1(x_0,0)$ can be determined
using~(\ref{LHS-triang}--\ref{Li-def-triang}), preferably using
Maple. It is found to be a linear 
combination of $T_m(q/2-1)$,  $T_m'(q/2-1)$ and  $T_m''(q/2-1)$ with
coefficients in $\GK$. The constant term $c_2$ of $\Pol_2$ can be determined
using~\eqref{cr-CT}: we set $z=(1-\sqrt{1+4u})/2$ in this equation, so
that $z^2-z=u$, 
and extract the coefficient of $u^2$. This gives $c_2= \CT\, \Pol_2$
as a linear combination of $T_m(q/2-1)$,  $T_m'(q/2-1)$ and
$T_m''(q/2-1)$ with coefficients in $\GK$. Given that $q/2-1=\cos
2k\pi/m$, we have  $T_m'(q/2-1)=0$. Putting these results together
gives
$$
[zx_0] \Pol= 
%
-\frac 1 8\,{\beta}^{m-1}q\, ( 4-q ) ^{2}\, T_m'' (q/2-1 ) .
$$
The last statement of Lemma~\ref{lem:x0-triang} then follows from
$$
T_m''(q/2-1)= -\frac{4m^2}{q(4-q)}.
$$
\end{proof}


\noindent
{\em{Proof of Theorem~{\rm\ref{thm:alg-triang}}.}}
The functional equation~\eqref{eq-gen1-triang} involves a single catalytic
variable, $y$.  However, the case $q=1$ shows that its form may not be
suitable for a direct application of our algebraicity theorem
(see~\eqref{eq-K-1-triang}). As it happens, a simple remedy for this is to
  reintroduce the original series $Q(y)\equiv Q(0,y)$. This is the counterpart of the transformation of~\eqref{eq-K-1-triang}
into~\eqref{eq-M-1-triang} performed in the case $q=1$. 
So, in~\eqref{eq-gen1-triang}, replace $K(y)$ by $qt^2y\,Q(y)$ and replace
 each derivative $K^{(j)}(0)$ by $jqt^2 \,Q^{(j-1)}(0)$.
According to Lemma~\ref{lem:x0-triang}, 
$$
\CT\, \Pol = [t] \Pol=  [z] \Pol= [t^2] \Pol=[tz] \Pol= [z^2] \Pol=
      [x_0] \Pol = \cdots = [x_m] \Pol =0.
$$
This implies that, once the series $K$ have been
expressed in terms of $Q$ in~\eqref{eq-gen1-triang}, a factor $t^3$ appears. 
  Divide the equation by $t^3$ to obtain
\beq\label{eq-gen2-triang}
\Pol'(Q(y), Q(0),Q'(0), \ldots, Q^{(2m-1)}(0), t;y )=0,
\eeq
where 
$$
\Pol'(x_0, \ldots,x_{2m} ,t;y)=
 \frac 1 {t^3}\,  \Pol(t^2yqx_0, t^2q x_1,2t^2q x_2, \ldots, 2mt^2q x_{2m},t;t\by)
$$
is a polynomial in $x_0, x_1, \ldots, x_{2m}, t$ and a Laurent
polynomial in $y$. 
Have we at last reached  an equation of the
form~\eqref{generic-Q}, to which we could apply our algebraicity
theorem? If this were the case, $\Pol'(x_0, \ldots,x_{2m} 
,0;y)$ should reduce to $x_0-1$.
 We have 
\begin{multline*}
  \Pol'(x_0, \ldots,x_{2m} ,0;y)= q x_0\big( y  [tx_0] \Pol +[zx_0]
  \Pol\big)\\
+\sum_{i=1}^{2m} i q x_i \big(   [tx_i] \Pol +\by [zx_i]  \Pol\big)
+ [t^3] \Pol + \by [t^2z] \Pol + \by^2 [tz^2] \Pol + \by^3 [z^3] \Pol .
\end{multline*}
By Lemma~\ref{lem:x0-triang}, this reduces to 
$$
  \Pol'(x_0, \ldots,x_{2m} ,0;y)= q x_0[zx_0]  \Pol
+ q x_1    [tx_1] \Pol  +  [t^3] \Pol.
$$
This means that, upon replacing $Q(0)$  by its value $1$,  the
functional equation~\eqref{eq-gen2-triang} can be written in the form
$$
qQ(y) [zx_0]  \Pol + q [tx_1] \Pol  +  [t^3] \Pol 
%
= t P_1( Q(y), Q'(0), \ldots, Q^{(2m-1)}(0), t; y)
$$
for some  $P_1(x_0, x_2, \ldots, x_{2m},t;y ) \in \GK[x_0, \ldots ,
  x_{2m}, t, y , \by]$. Moreover, the last identity of
Lemma~\ref{lem:x0-triang} allows us to rewrite this as
$$
q  \big( Q(y)-1\big) [zx_0] \Pol
= t P_1( Q(y), Q'(0), \ldots, Q^{(2m-1)}(0), t; y).
$$
Upon dividing by $q [zx_0] \Pol$ (which is
non-zero by Lemma~\ref{lem:x0-triang}), this has
the form
$$
Q(y)= 1+ tP_2(Q(y), Q'(0), \ldots, Q^{(2m-1)}(0),t;y)
$$
where   $P_2(x_0, x_2, \ldots, x_{2m+1},t;y)\in \GK[x_0, \ldots ,
  x_{2m}, t, y, \by]$.
 Finally, upon writing
$$
\frac{Q^{(i)}(0)}{i!}= \Delta^{i}Q(y)- y\Delta^{i+1}Q(y),
$$
where 
$$
\Delta F(y)= \frac{F(y)-F(0)}{y},
$$
the equation reads
$$
Q(y)= 1+ tP(Q(y), \Delta Q(y),\Delta^{(2)} Q(y), \ldots, \Delta^{2m} Q(y), t;y)
$$
where $P(x_0, x_1, x_2, \ldots, x_{2m},t;y)$ is a
polynomial in $t$ and the 
$x_j$'s, and a  Laurent polynomial  in $y$, with coefficients in
$\GK$.
Applying the general algebraicity theorem (Theorem~\ref{generic-thm}) 
implies that $Q(y)\equiv Q(0,y)$ is algebraic over $\qs(q, \nu, t, y)$.

Let us complete the proof of Theorem~\ref{thm:alg-triang} by proving
that $Q(x,y)$ is also algebraic.  We return to the functional equation
defining $Q(x,y)$, written in the form $K(x,y)Q(x,y)=R(x,y)$, where
$K(x,y)$ and $R(x,y)$ are given respectively by~\eqref{ker-Q}
and~\eqref{RHS-Q}. Recall that the series $Y_1$ defined in
Lemma~\ref{lem:kernel-triang} satisfies
$K(x,Y_1)=R(x,Y_1)=0$. By eliminating $Q_1(x)$ between these two
equations, one obtains a rational expression of $Q(0,Y_1)$ in terms of
$q, \nu, x, t$ and $Y_1$. But $Q(0,Y_1)$ is algebraic over $\qs(q,
\nu, t, Y_1)$: that is, there exists a non-zero polynomial $\Pol$ such
that $\Pol(q, \nu, t, Y_1, Q(0,Y_1))=0$. Replacing $Q(0,Y_1)$ by its
rational expression in this equation shows that $Y_1$ is algebraic over $\qs(q, \nu, t,
x)$. Then the  expression of $Q(0,Y_1)$ as a rational function of $q,
\nu, x, t$ and $Y_1$ shows that $Q(0,Y_1)$ itself  is algebraic over
$\qs(q, \nu, t, x)$. Finally, writing $R(x,Y_1)=0$ gives a rational
expression of $Q_1(x)$ in terms of $\nu, t, x, Y_1$ and $Q(0,Y_1)$:
hence $Q_1(x)$ is algebraic over $\qs(q, \nu, t, x)$. Returning to the
 functional equation that defines $Q(x,y)$ finally shows that this series is
 algebraic over $\qs(q,\nu, t, x, y)$.
\qed

\section{Two colors: the Ising model}
\label{sec:two}

In this section, we focus on the case $k=1$, $m=4$, that is, on
$q=2$. We give explicit algebraic equations satisfied by \gfs\ of
2-colored planar maps and 2-colored planar triangulations.  In other
words, we solve the Ising model (with no exterior field), averaged
on planar maps or triangulations of a given size. We also briefly report on
the singularity analysis of the solution, which allows us to locate
the critical value $\nu_c$ where a \emm phase transition, occurs.

 \subsection{Two-colored planar maps}
\label{sec:two-planar}

\begin{Theorem}\label{thm:planar-2-colours}
The Potts \gf\ of planar maps $M(2,\nu,t,w,z;x,y)$, defined
by~\eqref{potts-planar-def} and taken at $q=2$, is algebraic.
 The specialization $M(2,\nu,t,w,z;1,1)$ 
has degree  $8$ over  $\qs(\nu, t,w)$. 

When $w=z=1$, the degree decreases to
 $6$, and the equation admits a rational parametrization.
Let $S\equiv S(t)$ be the unique power series in $t$ with constant
term $0$ satisfying
$$
S=t\; \frac{\left( 1+3\,\nu\,S-3\,\nu\,{S}^{2}-{\nu}^{2}{S}^{3} \right)
  ^{2}}
{
  1-2\,S+2\,{\nu}^{2}{S}^{3}-{\nu}^{2}{S}^{4}  }.
$$
Then
\begin{multline*}
  M(2, \nu, t,1,1;1,1)= \frac{
  1+3\,\nu\,S-3\,\nu\,{S}^{2}-{\nu}^{2}{S}^{3}}
{\left(1-2\,S+2\,{\nu}^{2}{S}^{3}-{\nu}^{2}{S}^{4}\right)^2}\times
\\
\left(
{\nu}^{3}{S}^{6}
+2\, {\nu}^{2} (1- \nu ){S}^{5}
+\nu\, ( 1-6\,\nu ) {S}^{4}
-\nu\, ( 1-5\,\nu ) {S}^{3}
+ (1+ 2\,\nu ) {S}^{2}
-(3+ \nu ) S
+1
\right)
.
\end{multline*}
\end{Theorem}
\begin{proof}
The first statement is a specialization of
Theorem~\ref{thm:alg-planaires}. To obtain an explicit equation
satisfied by $M(2,\nu, t,w,z;1,1)$, we
first construct an equation with one catalytic variable satisfied by $M$, as
described in Section~\ref{sec:alg-maps}. Once again, the variable
$z$ is redundant, and $M(2, \nu, t,w,z;x,y)$ has the same degree over
$\qs(\nu, t,w,z,x,y)$ as $M(2, \nu, t,w,1;x,y)$  over
$\qs(\nu, t,w,x,y)$. We thus set $z=1$.

  We write the invariant equation~\eqref{eq-inv} for $q=2$ and
  $m=4$. It involves five unknown series $C_0, \ldots, C_4$, independent
  of $y$. By expanding this equation in the neighborhood of $y=1$, as
  described in Section~\ref{sec:Cr}, we   obtain the following expressions
  for the series $C_r$:
  \begin{eqnarray*}
    C_4 &=&( {\nu}^{2}+2\,\nu-1 )  ( {\nu}^{2}-2\,\nu-1 ) ,\\
    C_3 &=&-4\, ( \nu+1 )  ( {\nu}^{2}-4\,\nu+1 ) 
 ,\\
    C_2 &=&-4\, ( \nu-1 )  ( {\nu}^{3}+3\,{\nu}^{2}-6\,w{\nu}^{2}
-3\,\nu+2\,w-1 ) t-24\,\nu
 ,\\
    C_1 &=& -32\,\nu\,w ( \nu+1 )  ( \nu-1 ) ^{2}{t}^{2}
 M( 1 ) +8\, ( \nu-1 )  ( 3\,{\nu}^{2}-6\,
\nu\,w+2\,w-3 ) t+8+8\,\nu
,\\
 C_0 &=&-32\,\nu\,w ( \nu+1 )  ( \nu-1 ) ^{2}{t}^{3} M'( 1 ) 
-64\,\nu\,{w}^{2} ( \nu+1 )  ( \nu-1 ) ^{2}{t}^{3}  M ( 1 )  ^{2}\\
&&-32\,w ( \nu-1 ) ^{2} ( {\nu}^{2}t-3\,\nu+\nu\,t-1 ) {t}^{2} M( 1) \\
&&-4\, ( \nu-1 ) ^{2} ( {\nu}^{2}-2\,\nu+12\,\nu\,w+1+
4\,w-4\,{w}^{2} ) {t}^{2}-8\, ( \nu-1 )  ( 3\,
\nu-3-2\,w ) t -4.
  \end{eqnarray*}
In the invariant equation~\eqref{eq-inv}, let us now replace each
$C_r$ by its expression in terms of $M$: as was proved for general
values of $k$ and $m$ in 
Section~\ref{sec:final-form}, this gives (after dividing by
a factor $32tw(\nu-1)^2(1+\nu+y-y\nu)(1-\by)^2$)
an equation with one catalytic variable of the form~\eqref{generic-M},
involving the series $M(1)$ and $M'(1)$, or equivalently, the first two
discrete derivatives of $M(y)$. 

To solve this equation and 
obtain an algebraic equation satisfied by $M(1)$, we can  use
the general strategy of~\cite{mbm-jehanne}. 
An alternative, which requires less heavy calculations, relies on an
observation used by Tutte in the enumeration of properly colored
planar triangulations~\cite{tutte-chromatic-solsII,tutte-chromatic-revisited}. 
Consider the following two polynomials in $X$
$$
P_\pm(X):=\sum_{r=0}^4 C_r X^r\pm
\left((2\nu+\be^2)X^2-2 (\nu+1 ) X-2 \be t  ( 2w+\be ) +2
\right)^2 ,
$$
where $\be=\nu-1$. The second 
term is simply the square of the series $D(y)$ defined in Proposition~\ref{prop:inv-planaires}, seen as a
polynomial in  $X\equiv I(y)$.
The invariant equation~\eqref{eq-inv} can be written
$$
P_\pm(I(y))= D(y)^2\left( T_4\left(\frac {N(y)}{2\sqrt {D(y)}}\right)\pm 1\right).
$$
From the fact the polynomials $T_4(x)\pm 1$ both have a double root,
one can derive that $P_+(X)$ and $P_-(X)$ both have a double root in
$X$.
 Hence the discriminant of each of these polynomials
vanishes. This gives two polynomial equations relating $M(1)$ and
$M'(1)$, from which we obtain an equation for $M(1)$ by elimination.

One thus obtains an equation of degree 8 for the series $M(1)\equiv
M(2,\nu,t,w,1;1,1)$. 
It is too big to be written here.
However, when we do not keep track of the number of vertices (that is, when
$w=1$), this equation contains a factor $(1-tM(1)-t\nu M(1))^2$, which
clearly is not~$0$. The remaining factor is thus an algebraic equation
of degree $6$ satisfied by $M(2,\nu,t,1,1;1,1)$. The genus of the
corresponding curve (in $t$ and $M(1)$) is found to be $0$, so that
the curve has a rational parametrization. The one that we give in the
theorem was constructed  with the
help of the {\tt algcurves} package of {\sc Maple}.
\end{proof}

\medskip

\noindent{\bf Singularity analysis.}
We finally give, without a proof that would make this paper even
longer, the results of our analysis of the singularities of $M(2,
\nu,t,1,1;1,1)$. The singularity analysis of algebraic series in
$\ns[[t]]$
has become  quasi-automatic~\cite[Chap.~VII.7]{flajolet-sedgewick}, but of course
things are a bit more delicate here because of the parameter $\nu$.

\begin{Claim}
\label{claim-asympt}
Let $P_1$ and $P_2$ be the following two polynomials:
\begin{eqnarray*}
  P_1(\nu,\rho)&=&  432\,{\nu}^{3} \left( \nu+1 \right) {\rho}^{3}+108\,{\nu}^{2} \left( 
\nu-1 \right) {\rho}^{2}+1-\nu,
\\
P_2(\nu,\rho)&=&432\,{\nu}^{2} \left( \nu+1 \right) ^{4}{\rho}^{4}+72\,\nu\,
 \left( \nu+1 \right) ^{2}{\rho}^{2}-8\, \left( \nu-1 \right)  \left( 
\nu+1 \right) \rho -1.
\end{eqnarray*}
Consider $M(2,\nu,t,1,1;1,1)\equiv M(2,\nu,t)$ as a series in $t$ depending on the
parameter $\nu$. Let $\rho_\nu$ denote its radius of
convergence. Then $\rho_\nu$ is a continuous decreasing function of
$\nu$ for $\nu\ge 0$, which satisfies
$$
\begin{array}{lllllll}
  P_2(\nu, \rho_\nu)&=&0& \hbox{for}& 0 \le \nu \le
  \nu_c:=\frac{3+\sqrt 5}2,\\
P_1(\nu, \rho_\nu)&=&0& \hbox{for}& \nu_c \le \nu.
\end{array}
$$
Moreover, 
$$
\rho_0= \frac 1 8 \quad \hbox{ and } \quad 
\rho_{\nu_c}= \frac {3\sqrt 5 -5}{60}.
$$
The critical behaviour of $M(2,\nu,t)$ is usually the standard behaviour of
planar maps series, with an exponent $3/2$:
$$
M(2,\nu,t)= \alpha_\nu +\beta_\nu(1-t/\rho_\nu)+\gamma_\nu (1-t/\rho_\nu)^{3/2}\,(1+o(1)),
$$
except at $\nu=\nu_c$, where the nature of the singularity changes:
$$
M(2,\nu_c,t)= \alpha_{\nu_c} +\beta_{\nu_c}(1-t/\rho_{\nu_c})+\gamma_{\nu_c} (1-t/\rho_{\nu_c})^{4/3}\,(1+o(1)).
$$
In particular,
 $$
[t^n]M(2,\nu,t)\sim
\left\{
\begin{array}{ll}
  \kappa\, \rho_\nu^n  n^{-5/2}&\hbox{ for  } \nu\not=\nu_c,
\\
 \kappa\, \rho_{\nu_c}^n  n^{-7/3}&\hbox{ for  } \nu=\nu_c.
\end{array}
\right.
$$
\end{Claim}
\noindent
Figure~\ref{fig:radius} shows a plot of
the curves $P_2(\nu,\rho)=0$ and $P_1(\nu,\rho)=1$.
The first step in the proof is to study the singularities of the
 series $S$ defined in Theorem~\ref{thm:planar-2-colours}. 
This series has constant term 0, and non-negative
 coefficients. The discriminant of the vanishing polynomial of $S$ is,
 up to factor independent of $t$, the product $P_1(\nu,t)P_2(\nu,t)$. 
The series $S$ is found to have a square root singularity at
 $\rho_\nu$, except at $\nu=\nu_c$ where  the singularity is in
 $(1-t/\rho_{\nu_c} )^{1/3}$.   Figure~\ref{fig:radius} shows plots of
 $S(t)$ for several values of $\nu$.  The singular behaviour of
$M(2,\nu,t)$ is then derived from the expression of this series in
terms of $S$.

\begin{figure}[htb]
\includegraphics[height=4cm]{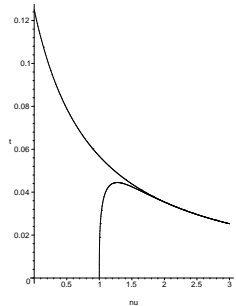}
\caption{The curves $P_2(\nu,\rho)=0$ (top) and $P_1(\nu,\rho)=1$ (bottom), for
  $\nu\in[0,3]$. The two curves meet at  $\nu_c\simeq 2.618$.
  Beyond this value, the curve $P_1$ is
   above $P_2$ (although very close at this scale). For every $\nu$, the radius is
   given by the \emm highest, of the curves.} 
\label{fig:radius}
\end{figure}

\begin{figure}[htb]
\includegraphics[height=3cm]{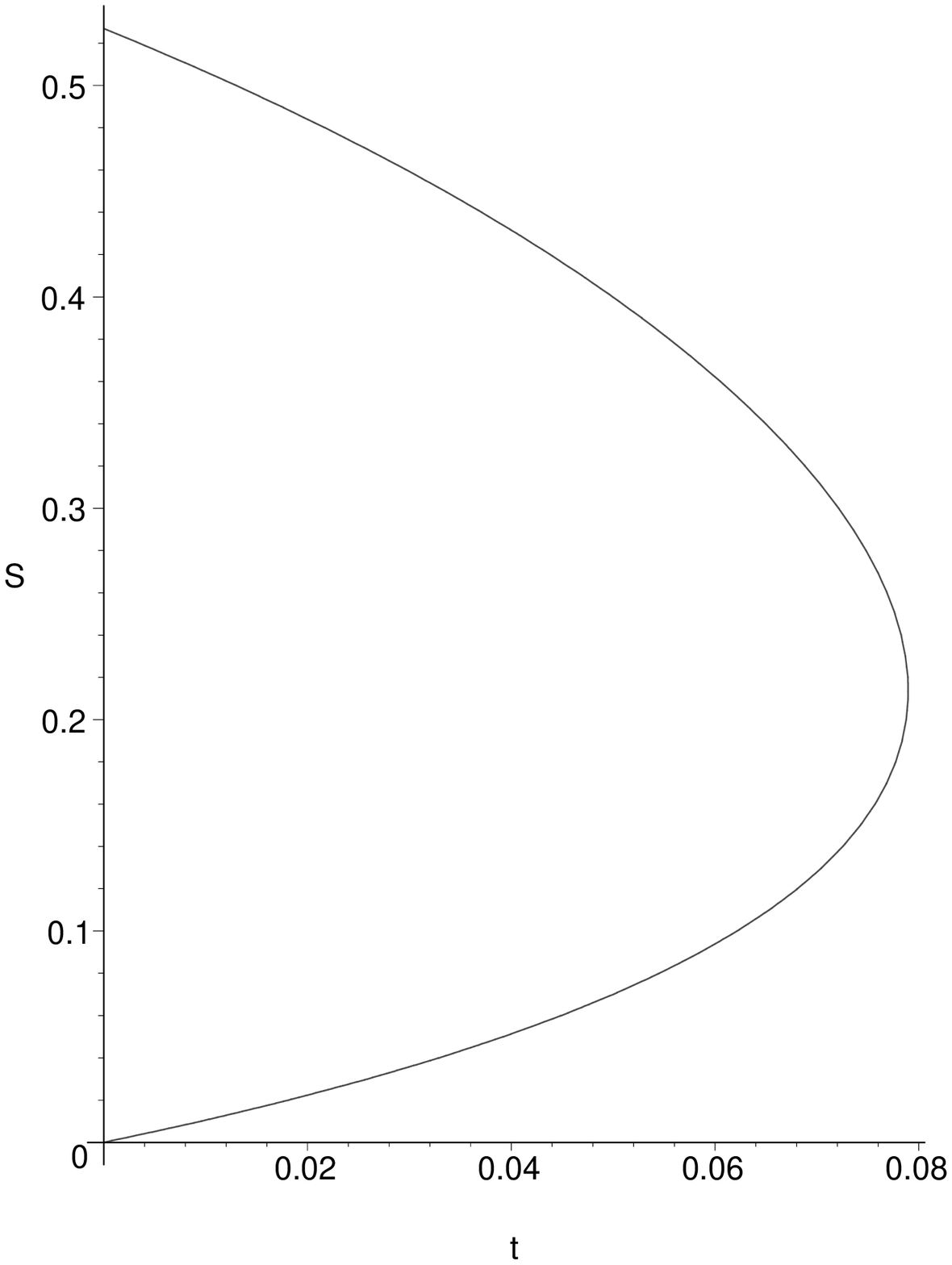}
\hskip 10mm
\includegraphics[height=3cm]{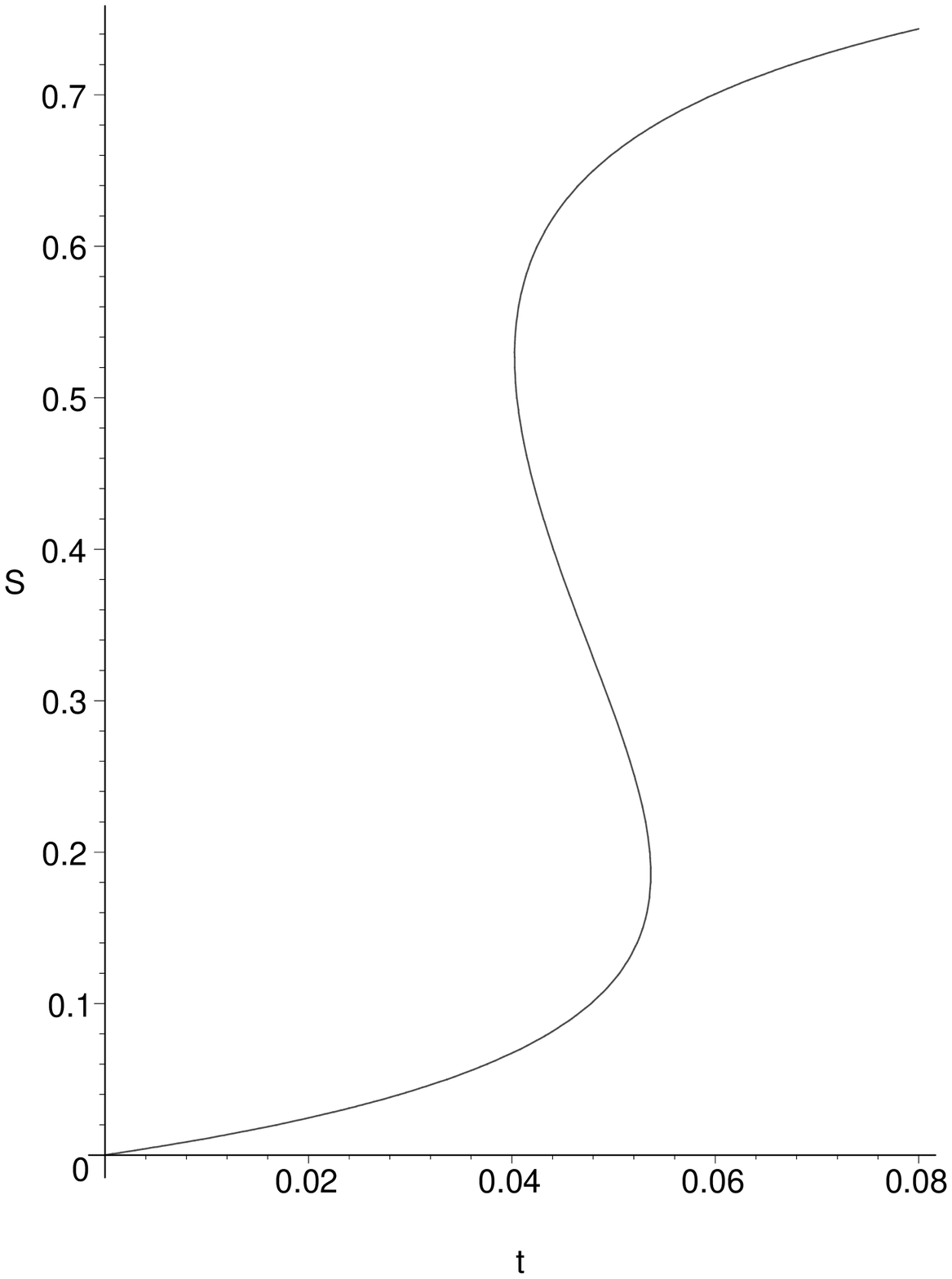}
\hskip 10mm
\includegraphics[height=3cm]{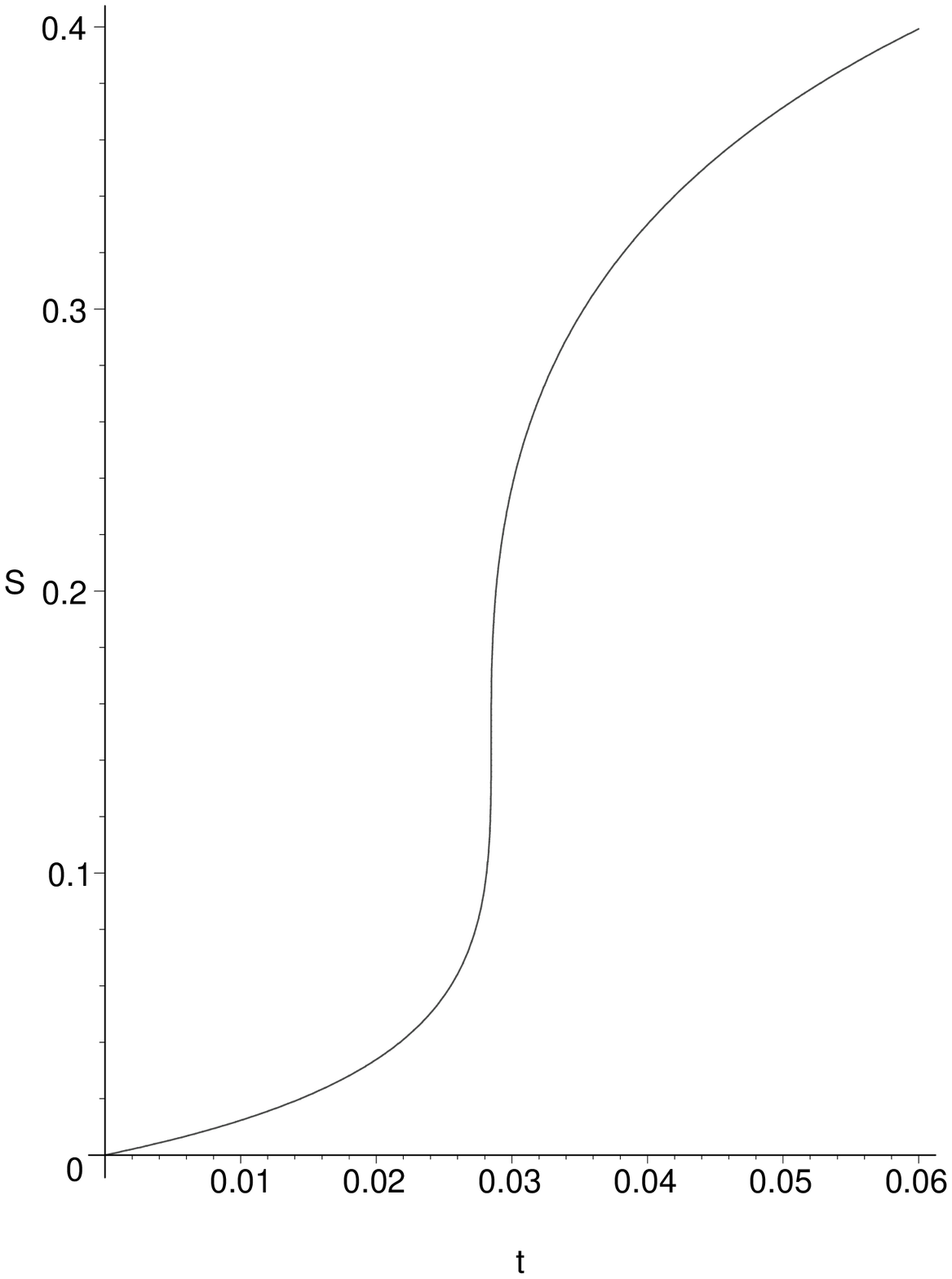}
\hskip 10mm
\includegraphics[height=3cm]{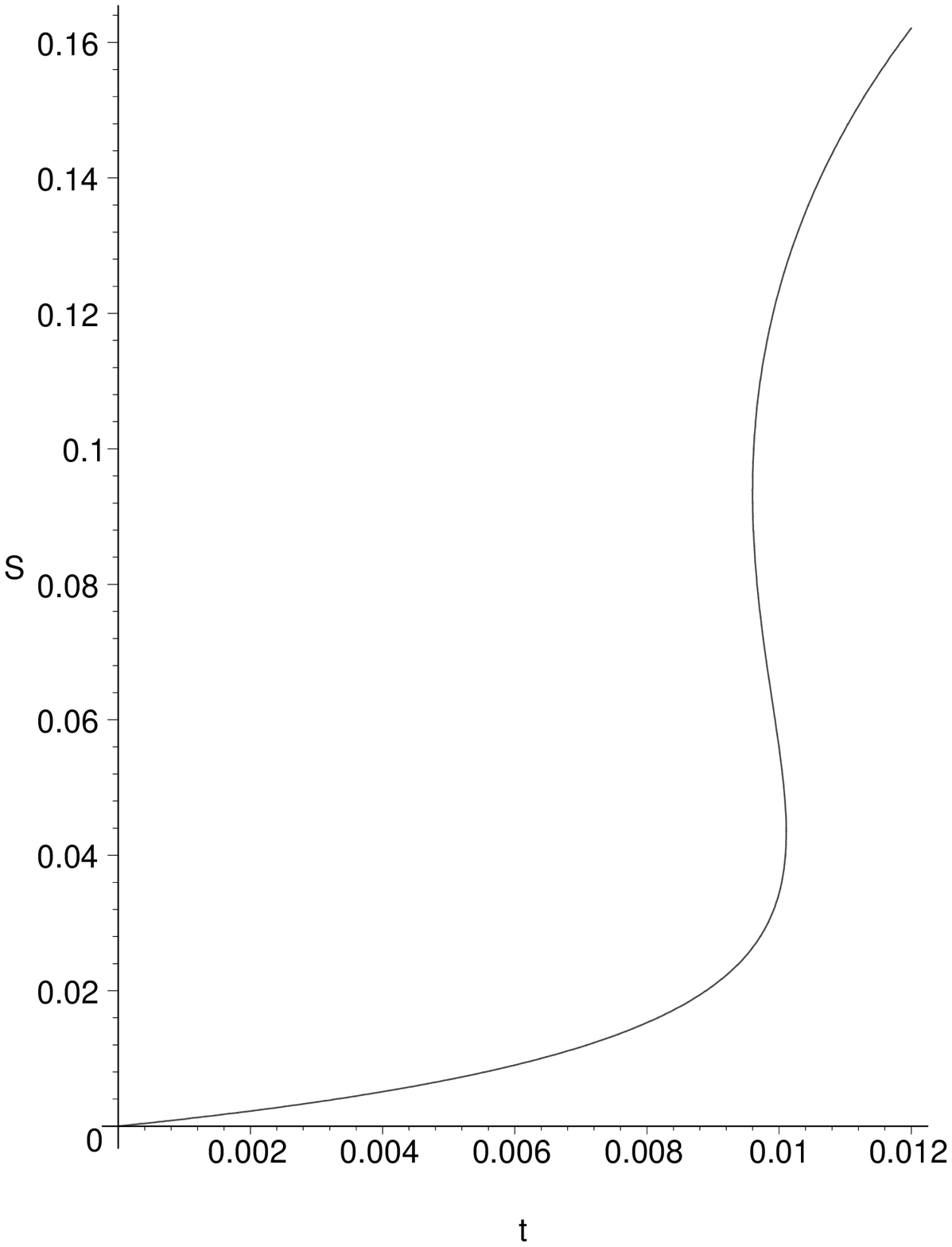}
\caption{The algebraic function $S(t)$ and some of its conjugates, for $\nu=0.5$, $\nu=1.1$, $\nu=\nu_c$
  and $\nu=8$. The function $S$ is the branch that vanishes at the
  origin. The values of $t$ where the tangent is vertical are such
  that $P_1(\nu,t)=0$ or $P_2(\nu,t)=0$. The form of the function is in
  agreement with the fact that the radius of $S$ is given in
  Figure~\ref{fig:radius} by the \emm highest, of the curves.} 
\label{fig:radius2}
\end{figure}

 \subsection{Two-colored triangulations}
\label{sec:two-triang}

\begin{Theorem}\label{thm:triang-2}
 The Potts \gf\ of quasi-triangulations $Q(2,\nu,t,w,z;x,y)$, defined
by~\eqref{Q-ser-def} and taken at $q=2$, is algebraic.

In particular, the series $Q_i(2,\nu,t):=[y^i]Q(2,\nu,t,1,1;0,y)$ that counts
two-colored  near-triangulations of outer degree~$i$ by edges ($t$) and
monochromatic edges ($\nu$) is algebraic of degree (at most) 
$5$ over  $\qs(\nu,t )$ and
admits a rational parametrisation.
Set $v= (\nu+1)/(\nu-1)$. Let $S\equiv S(\nu, t)$ be the unique power
series in $t$ having constant term $v$ and
satisfying
\beq\label{S-def}
t^3={\frac 
{\left( S-v  \right)  \left( S-2+v \right)
 \left( 2\,v-{v}^{2}+2\,S+{S}^{2}-4\,{S}^{3} \right) 
  }
{ 64 \left( 1+v \right) ^{3}{S}^{2}}}.
\eeq
Then $t^i Q_i(2,\nu,t)$ has a rational expression
in terms of $S$ and $v$.
In particular,
$$
t^4Q_1(2,\nu,t)= t^5\nu Q_2(2,\nu,t)={\frac { \left( S-v \right) ^{2}\left( S-2+v \right) 
 \left( -2\,v+{v}^{2}-Sv-{S}^{2}v +3\,{S}^{3}\right)  
}{128  \left( 1+v \right) ^{4}{S}^{2}}},
$$
while
$$
t^6Q_3(2,\nu,t)= \frac{(S-v)^3 (S-2+v)
P(v,S)}{8192 (1+v)^6S^4}
$$
with 
\begin{multline*}
 P(v,S)=
-64\,{S}^{6}+ \left( 232 -128\,v\right) {S}^{5}
- \left(67+ 48\,v-64\,{v}^{2} \right) {S}^{4}
+ \left( 106-102\,v+40\,{v}^{2} \right) {S}^{3}
\\
-2\, \left( v-2 \right)  \left( 32\,{v}^{2}-48\,v-1 \right) {S}^{2}+2
\, \left( 3\,v-1 \right)  \left( v-2 \right) ^{2}S+3\,v \left( v-2
 \right) ^{3}
.
\end{multline*}
\end{Theorem}

\begin{proof}
The first statement is a specialization of
Theorem~\ref{thm:alg-triang}. Recall that the variables $w$ and $z$
are redundant: we thus focus on the case $w=z=1$. To obtain  explicit
algebraic equations for near-triangulations, we
first construct an equation with one catalytic variable satisfied by
$Q(y)\equiv Q(0,y)$, as
described in Section~\ref{sec:alg-triang}. 

  We write the invariant equation~\eqref{eq-inv-triang} for $q=2$ and
  $m=4$. It involves five unknown series $C_0, \ldots, C_4$, independent
  of $y$. By expanding this equation in the neighborhood of $y=0$ and
  writing that the 
  coefficient of $y^{-2r}$ is zero, we   obtain  explicit
  expressions   for the series $C_r$, as   described in
  Section~\ref{sec:Cr-triang}. Moreover, by  writing that the
  coefficient of $y^{-2r+1}$ is zero, for $1\le r \le 4$, we obtain
  additional identities 
  relating the series $Q^{(i)}(0)$. More precisely:
  \begin{itemize}
  \item [---] by extracting the coefficients of $y^{-8}$, $y^{-6}$,
    $y^{-4}$, we obtain:
    \begin{eqnarray*}
    C_4 &=&-4\nu^4  ,\\
    C_3 &=&8\,{\nu}^{2} \left( \nu-1 \right)  \left( 2\,\nu-3 \right) 
 ,\\
    C_2 &=&16\,{\nu}^{3} \left( \nu-1 \right) {t}^{3}+4\, \left( 5\,\nu+1
 \right)  \left( \nu-1 \right) ^{3};
  \end{eqnarray*}
\item [---]  extracting the coefficients of $y^{-7}$ and $y^{-5}$ does not yield
  new identities;
\item [---]  extracting the coefficient of $y^{-3}$ gives $Q(0)=1$,
  which is not a surprise. From now on we systematically replace every
  occurrence of $Q(0)$ by $1$;
\item [---]  extracting the coefficient of $y^{-2}$ gives
 $$
    C_1 =
-64\,{\nu}^{2}  \left( \nu-1 \right) ^{2}{t}^{4}Q'(0)
+48\,\nu\, \left( \nu-1 \right) ^{2}{t}^{3}+4\,
 \left( 2\,\nu-1 \right)  \left( \nu-1 \right) ^{3}
;
$$
\item [---]  extracting the coefficient of $y^{-1}$ gives:
\beq\label{Q1-Q2}
Q''(0)= \frac{2Q'(0)}{t\nu}.
\eeq
This identity is a special case of the last statement of
Proposition~\ref{prop:eq-Q}, and has a
  simple combinatorial explanation.
From now on we systematically replace every occurrence of $Q''(0)$ by
this expression;
\item [---]  finally, extracting the coefficient of $y^{0}$ gives
  \begin{multline*}
    C_0 =
-{\frac {32}{3}}\,{\nu}^{2} \left( \nu-1 \right) ^{2}{t}^{6}
 Q^{(3)}(0) -32\,
 \left( \nu+1 \right)  \left( \nu-2 \right)  \left( \nu-1 \right) ^{2}
{t}^{4}
 Q'(0)\\
-112\,{\nu}^{2}
 \left( \nu-1 \right) ^{2}{t}^{6}-8\, \left( \nu-4 \right)  \left( \nu
-1 \right) ^{3}{t}^{3}+ \left( \nu-1 \right) ^{4}
.
 \end{multline*}
\end{itemize}

Let us now replace each $C_r$ by its expression in the invariant
equation~\eqref{eq-inv-triang}:  this gives, after dividing by
$32(\nu-1)^3t^3$,
an equation with one catalytic variable for $Q(y)$, of the form~\eqref{generic-Q},
involving  the (only) two unknown series $Q'(0)$ and $Q^{(3)}(0)$. 

To solve this equation and 
obtain an algebraic equation satisfied by $Q'(0)$, we can  use
the general strategy of~\cite{mbm-jehanne}. 
But we  can also apply the alternative method  already used for
2-colored planar maps in the previous subsection. 
Consider the following two polynomials in $X$:
$$
P_\pm(X):=\sum_{r=0}^4 C_r (tX)^r \pm \left(
2\nu^2 {t}^{2} X^{2}
+\be  \left( 4\be +2 \right)t X
-4\be\nu {t}^{3}   +\be^2
\right)^2 
$$
where $\be=\nu-1$.
The second term 
 is simply the square of the series $D(y)$ defined in
Proposition~\ref{prop:inv-triang}, 
seen as a polynomial in $X\equiv I(y)$. Then the polynomials
 $P_+(X)$ and $P_-(X)$ have a double root in $X$. Hence the
discriminant of each of them  
vanishes. This gives two polynomial equations relating  $Q'(0)$ and
$Q^{(3)}(0)$, from which we obtain an equation of degree 5 for
$Q_1(2,\nu,t)\equiv Q'(0)$ by
elimination.  The genus of the
corresponding curve (in $t$ and $Q'(0)$) is found to be $0$, so that
the curve has a rational parametrization, which we have constructed  with the
help of the {\tt algcurves} package of {\sc Maple}.

The expression of $Q_2(2,\nu,t)= Q''(0)/2$ follows
from~\eqref{Q1-Q2}. The expression of $Q_3(2,\nu,t)= Q^{(3)}(0)/6$ can
be obtained using any of the polynomial equations relating  $Q'(0)$ and
$Q^{(3)}(0)$ that we have obtained on the way to our derivation of $Q'(0)$.

Let us finally explain why each series $t^i Q_i(2,\nu,t)$ can be
written in terms of $S$ and $v$. 
The equation with one catalytic
variable satisfied by $Q(y)$ reads
$$
6t^3\nu^2(Q(y)-1)= y\, \Pol(\nu, Q(y), Q_1, Q_3, t,y)
$$
for some polynomial $\Pol$ with integer coefficients.
Differentiating $i$ times with respect to $y$, and then setting $y=0$
thus gives $Q^{(i)}(0)\equiv i! \, Q_i(2,\nu,t)$ as a polynomial in $Q_1,
\ldots, Q_{i-1}$ with coefficients in $\qs(\nu, t)$. By combining
Euler's relation and the edge/face incidence relation, one easily
shows that $t^i Q_i(2,\nu,t)$ is a series in $t^3$. Since $\nu$ can
be expressed in terms of $v$, and $t^3, Q_1, Q_2, Q_3$ can be
expressed rationally in terms of $v$ and $S$, the same holds for any
$t^i Q_i(2,\nu,t)$ by induction on $i$.
\end{proof}

\medskip
\noindent
{\bf Connections with previous work.} This result is very close to the solution of
the Ising model  on \emm near-cubic, maps, derived by Boulatov \& 
Kazakov~\cite{BK87} using matrix integrals and then by Bousquet-M\'elou \&
Schaeffer~\cite{mbm-schaeffer-ising} using bijections with trees.
We say that a planar map is 
\emm near-cubic, if its dual is a near-triangulation; that is, every non-root
vertex  has degree 3.
Then for a generic value of $q$, the series $Q_i(q, \nu, t):= [y^i]
Q(q, \nu,t,1,1;0,y)$, which counts $q$-colored near-triangulations of
outer degree $i$ by edges and monochromatic edges, can be interpreted
in terms of near-cubic maps using the duality
relation~\eqref{eq:duality-Potts-poly}:
\begin{eqnarray}
 Q_i(q,\nu,t) 
&=&
\sum _{M \rm{ near-triang.}\atop \df(M)=i} t^{\ee(M)} \Ppol_M(q,\nu)
\nonumber \\
&=&
\sum _{G \, \rm{ near-cubic} \atop  \dv(G)=i} t^{\ee(G)}
\Ppol_{G^*}(q, \nu)
\nonumber \\&=&
\sum _{G \, \rm{ near-cubic} \atop  \dv(G)=i} t^{\ee(G)}\,
\frac{(\nu-1)^{\ee(G)}} {q^{\vv(G)-1}} \Ppol_G\left(q,
1+\frac{q}{\nu-1}\right)
\hskip 30mm \hbox{by~\eqref{eq:duality-Potts-poly}}
\nonumber \\&=&
\left( \frac{t(\nu-1)} q\right)^{-i} \sum _{G \, \rm{ near-cubic} \atop
  \dv(G)=i} 
 \left( \frac{(\nu-1)^3 t^3}{q^2}\right)^{\ff(G)-1}
\Ppol_G\left(q, 1+\frac{q}{\nu-1}\right).
\label{near-t--near-c}
\end{eqnarray}
We have used in the last line Euler's relation and the edge/vertex
incidence relation, according to which $2\ee(G)= 3(\vv(G)-1)+\dv(G)$.

Let us return to the case $q=2$. The series  $I_i(X,u)$ studied
in~\cite{mbm-schaeffer-ising} counts by non-root vertices (variable $X$) and by
\emm bichromatic, edges  (variable $u$) 2-colored near-cubic maps
$G$ 
such that $\dv(G)=i$ (as in the present paper, the color of the root vertex is
fixed).  The  connection between our series $Q_i(2,\nu,t)$ follows
from~\eqref{near-t--near-c}:
$$
 Q_i(2,\nu,t) =
\left( \frac{t(\nu-1)} 2\right)^{-i} (uX)^i I_i(X,u)
$$
with
$$
u= \frac{\nu-1}{\nu+1} , \quad  \quad X^2 =\frac{(\nu+1)^3 t^3}4.
$$
Via this correspondence, the value of $Q_2(2,\nu,t)$ given in
Theorem~\ref{thm:triang-2}  is
equivalent to the case $X=Y$ of~\cite[Proposition~20]{mbm-schaeffer-ising}. The series
$\bar Q$ defined in the latter reference coincides with the series $S$
defined by~\eqref{S-def}.

\medskip

\noindent{\bf Singularity analysis.}
The singular behaviour of $Q_1(2,\nu,t)$ is similar to that of the series $M(2,
\nu,t,1,1;1,1)$ studied in the previous subsection. Again, we state
our results without proof (see also~\cite{BK87}).

\begin{Claim}
Let $P_1$ and $P_2$ be the following two polynomials:
\begin{multline*}
  P_1(\nu,\rho)\ = \ 
      131072\,{\rho}^{3}{\nu}^{9}-192\,{\nu}^{6} \left( 3\,\nu+5 \right) 
 \left( \nu-1 \right)  \left( 3\,\nu-11 \right) {\rho}^{2}
\\
-48\,{\nu}^{3} \left( \nu-1 \right) ^{2}\rho+ \left( \nu-1 \right)  \left( 4\,{\nu
}^{2}-8\,\nu-23 \right), 
\end{multline*}
$$
P_2(\nu,\rho)\ =\ 
27648\,{\rho}^{2}{\nu}^{4}+864\,\nu\, \left( \nu-1 \right)  \left( {
\nu}^{2}-2\,\nu-1 \right) \rho+ \left( 7\,{\nu}^{2}-14\,\nu-9 \right) 
 \left( \nu-2 \right) ^{2}. \hskip 25mm
$$
Consider $tQ_1(2,\nu,t)$ as a series in $t^3$ depending on the
parameter $\nu$. Let $\rho_\nu$ denote its radius of
convergence. Then $\rho_\nu$ is a continuous decreasing function of
$\nu$ for $\nu> 0$, which satisfies
$$
\begin{array}{lllllll}
  P_2(\nu, \rho_\nu)&=&0& \hbox{for}& 0< \nu \le
  \nu_c:=1+ 1/\sqrt 7,\\
P_1(\nu, \rho_\nu)&=&0& \hbox{for}& \nu_c \le \nu.
\end{array}
$$
Moreover, 
$$
\rho_\nu \rightarrow +\infty \ \hbox{ as } \nu\rightarrow 0
 \quad \hbox{ and } \quad \rho_{\nu_c}
= \frac {25\,\sqrt {7}-55}{864}.
$$
The critical behaviour of $tQ_1(2,\nu,t)$ is usually the standard behaviour of
planar maps series, with an exponent $3/2$:
$$
tQ_1(2,\nu,t)= \alpha_\nu +\beta_\nu(1-t^3/\rho_\nu)+\gamma_\nu (1-t^3/\rho_\nu)^{3/2}\,(1+o(1)),
$$
except at $\nu=\nu_c$, where the nature of the singularity changes:
$$
tQ_1(2,\nu_c,t)= \alpha_{\nu_c} +\beta_{\nu_c}(1-t^3/\rho_{\nu_c})+\gamma_{\nu_c} (1-t^3/\rho_\nu)^{4/3}\,(1+o(1)).
$$
\end{Claim}

 \noindent{\bf Note.} The analysis is similar to the case of general
 planar maps, but the role that was played by the series $S$ in the
 proof of Claim~\ref{claim-asympt} is now played by the series $U$
 such that $S=v(1-2U)$. In particular, $U$ has constant term 0 and
 non-negative coefficients (which is not the case of $S$).  

 \section{Three colors}
\label{sec:three}
In this section, we focus on the case $k=1$, $m=6$, that is, on
$q=3$. We give explicit algebraic equations satisfied by \gfs\ of
\emm properly, 3-colored planar maps and  triangulations. This
corresponds to $\nu=0$. The case when $\nu$ is generic leads to
equations with one catalytic variable involving four unknown series
(of the form $M^{(i)}(1)$ or $Q^{(i)}(0)$, depending on whether we
deal with general maps or triangulations), and their solution has
defeated us so far. However, we conjecture
an algebraic equation for the series counting properly 3-colored \emm cubic,
maps (Conjecture~\ref{conj}).

 \subsection{Three-colored planar maps}
\label{sec:three-planar}
%
\begin{Theorem}
The Potts \gf\ of planar maps $M(3,\nu,t,w,z;x,y)$, defined
by~\eqref{potts-planar-def} and taken at $q=3$, is algebraic.

The specialization  $M(3,0,t,1,1;1,1)$ 
that counts \emm properly, three-colored planar maps by edges
has degree $4$ over  $\qs(t)$, and  admits a rational parametrization.
Let $S\equiv S(t)$ be the unique power series in $t$ with constant
term $0$ satisfying
\beq\label{S-def-planar3}
t= \frac{ S(1-2\,S^3)  }{\left( 1+2S \right)  ^{3}}.
\eeq
Then
\beq\label{M3-sol}
  M(3,0, t,1,1;1,1)= 
{\frac { \left(1+ 2\,S \right)  
\left(1 -2\,{S}^{2}-4\,{S}^{3}-4\,{S}^{4} \right) }
{ \left(1- 2\,{S}^{3} \right) ^{2}}}.
\eeq
 The coefficient of $t^n$ in this series, which is the number of
properly $3$-colored maps with $n$ edges, is asymptotic to
$
\kappa\, \mu^n n^{-5/2},
$
{where}
$$
\kappa >0 \quad \hbox{and} \quad 
\mu=\frac{22+8\,\sqrt {6}}3.
$$
\end{Theorem}
\begin{proof}
The first statement is a specialization of
Theorem~\ref{thm:alg-planaires}. We would like to obtain an explicit equation
satisfied by $M(3,\nu, t,w,z;1,1)$.  As described in
Section~\ref{sec:alg-maps},  we
first construct an equation with one catalytic variable satisfied by
$M$.
 Once again, the variable $z$ is redundant, and we set $z=1$.

\medskip
\noindent {\bf{An equation with one catalytic variable.}}
  We start from the invariant equation~\eqref{eq-inv}, written for $q=3$ and
  $m=6$. It involves seven unknown series $C_0, \ldots, C_6$, independent
  of $y$. By expanding this equation in the neighborhood of $y=1$, as
  described in Section~\ref{sec:Cr}, we   obtain explicit expressions
  of the series $C_r$. More precisely, 
  \begin{itemize}
  \item[---]   $C_6, C_5$ and $C_4$ are  polynomials in $\nu$, $t$ and
    $w$,
 \item[---]    $C_3$ is a polynomial in  $\nu$, $t$, $w$ and $M(1)$,
  \item[---]    $C_2 $ is a polynomial in  $\nu$, $t$, $w$,  $M(1)$ and $M'(1)$,
 \item[---]    $C_1$ is a polynomial in  $\nu$, $t$, $w$,  $M(1)$,  $M'(1)$ and
   $M''(1)$,
\item[---]    $C_0$ is a polynomial in  $\nu$, $t$, $w$, $M(1)$,  $M'(1)$,
   $M''(1)$ and $M^{(3)}(1)$.
 \end{itemize}
In the invariant equation~\eqref{eq-inv}, let us now replace each $C_r$ by its expression : as was proved for general values of $k$ and $m$ in
Section~\ref{sec:final-form}, this gives, after dividing by
$$
27tw(\nu-1)^2\, \left(1-\by \right) \left( 1-\nu+\by +\nu \by  \right)  
\left(1-\nu + 2\by+\nu \by\right)  \left(2-2\nu+\by+2\,\nu \by \right)  ,
$$
an equation with one catalytic variable of the form~\eqref{generic-M},
involving the series $M(1)$, $M'(1)$, $M''(1)$   and
$M^{(3)}(1)$, that is, the first four discrete derivatives of $M(y)$. 
Even though the general strategy of~\cite{mbm-jehanne}
 allows one to solve, in theory, this equation, the size of
the calculations has prevented us to do so in the general case.

So let us focus on the simpler case of \emm properly three-colored,
planar maps. That is, we set $\nu=0$ so as to forbid
monochromatic edges. This simplifies the series $C_r$. Indeed, $C_3$
(resp. $C_2$, $C_1$, $C_0$) does not involve $M(1)$ (resp. $M'(1)$,  $M''(1)$,
 $M^{(3)}(1))$ any more. We further ignore the number of vertices by
setting $w=1$. The resulting equation in one catalytic
variable reads
\beq\label{cat-planaires-q3}
\Pol( M(y),M(1), M'(1),M''(1),t;y)=0,
\eeq
and has degree 4 in $M(y)$.  More precisely, the equation can be written
\beq\label{cat-planaires-q3-bis}
 M(y)= 1+ \frac{ty^2}{ 2 \left( 2\,y+1 \right)  \left( y+2 \right)
   \left( y+1 \right) } P(M(y), \Delta M(y), \Delta^{2} M(y), \Delta^{3}M(y),  t,y)
\eeq
where
\begin{multline*}
P(x_0,x_1,x_2,x_3,t,y)=
1+11\,y +4\,{t}^{2}{y}^{2}x_{{3}}
+2 \left( -13\,{y}^{2}+2\,{t}^{2}{y}^{2}+5\,t{y}^{2}-29\,y-2\,ty-9 \right) x_{{0}}
\\
+ \left( 8\,{y}^{3}+2\,t{y}^{3}+32\,{t}^{2}{y}^{3}-62\,t{y}^{2}+54\,{y}^{2}
         +32\,{t}^{2}{y}^{2}-12\,ty+75\,y+25 \right) {x_{{0}}}^{2}
\\
+2\,t{y}^{2} \left( -26\,{y}^{2}+42\,t{y}^{2}+36\,ty-65\,y-26+18\,t \right) {x_{{0}}}^{3}
+36\,{t}^{2}{y}^{4} \left( 2\,y+1 \right) {x_{{0}}}^{4}
\\
+2 \left( 6\,{t}^{2}{y}^{2}-4\,t{y}^{2}
               +16\,{y}^{2}+9\,y-4\,ty+2 \right) x_{{1}}
-2\,ty \left( 22\,t{y}^{2}-33\,{y}^{2}-34\,ty+27\,y+6 \right) {x_{{1}}}^{2}
\\
+36\,{t}^{2}{y}^{2} \left( y-1 \right) ^{2}{x_{{1}}}^{3}
+2\,ty \left( 18\,t{y}^{2}-27\,{y}^{2}+50\,ty-78\,y-12 \right) x_{{0}}x_{{1}}
+36\,{y}^{2} \left( {y}^{2}+2\,y+3 \right) {t}^{2}x_{{1}}{x_{{0}}}^{2}
\\
-36\,{y}^{2} \left( y+3 \right)  \left( y-1 \right) {t}^{2}x_{{0}}{x_{{1}}}^{2}
+2\,ty \left( 6\,ty-11\,y-2 \right) x_{{2}}
+12\,{y}^{2} \left( y+3 \right) {t}^{2}x_{{0}}x_{{2}}
-36\, \left( y-1 \right) {y}^{2}{t}^{2}x_{{1}}x_{{2}}.
\end{multline*}

Still, both the general approach of~\cite{mbm-jehanne} and
Tutte's alternative (used above for two-colored planar maps) require
heavy calculations. Hence we have resorted to Tutte's good old method:
guess and check!

\medskip
\noindent{\bf{An interlude: solving planar maps by guessing and checking.}}
Because things are so heavy with our equation, let us discuss the
principles of this method  on the much simpler example of planar maps
counted by edges (variable $t$) and outer degree 
(variable $y$). The standard equation  with one catalytic variable that
defines the associated \gf\ $G(t;y)\equiv G(y)$ reads
\beq\label{eqG-map}
G(y)=1+ ty^2 G(y)^2 + ty\, \frac{yG(y)-G(1)}{y-1}.
\eeq
It is clear that this equation has a unique solution that is a power
series in $t$. The coefficients of this series are polynomials in
$y$. But how can we determine $G(y)$? Assume that we
find two series, $F(t;y) \equiv F(y) \in \qs[y][[t]]$ and $F_0(t)\equiv F_0 \in
\qs[[t]]$, such that
\beq\label{eqF-map}
F(y)=1+ ty^2 F(y)^2 + ty\, \frac{yF(y)-F_0}{y-1}.
\eeq
Then, by multiplying by $(y-1)$ and setting $y=1$, we discover that
$F_0$ equals necessarily $F(1)$, so that $F(y)$ is the map
\gf\ $G(y)$. Note that it is important that the series $F(y)$ has \emm
polynomial, coefficients in $y$ (or at least, coefficients in $\qs(y)$
having no pole at $y=1$). Otherwise $F(1)$ may not be well-defined.

 From the functional equation~\eqref{eqG-map}, one can compute the first
coefficients of the series $G(y)$. In particular, one can easily \emm
conjecture,, using tools like the {\sc Gfun} package of {\sc
  Maple}~\cite{gfun}, that $G(1)$ is quadratic: 
$$
27\,{t}^{2}G(1)^{2}+G(1)(1-18\,t)-1+16\,t=0.
$$
The corresponding curve has genus 0, and thus admits a rational
parametrization. Let $S\equiv S(t)$ be the unique power series in $t$
satisfying $S=t(1+3S)^2$. Then our conjectured value of $G(1)$ is
$
G(1)=(1-S)(1+3S).
$
Let us define $F_0:= (1-S)(1+3S)$. There exists a unique power series
$F(y)\equiv F(t;y)$ satisfying~\eqref{eqF-map} (indeed, this equation
is quadratic in $F(y)$, and the other root contains negative powers of
$t$). However, this series 
has \emm a priori, coefficients in~$\qs(y)$. We wish  to prove these
coefficients  actually lie in $\qs[y]$. One way goes as follows.
In~\eqref{eqF-map}, replace $F_0$ by its value $(1-S)(1+3S)$, and $t$
by its expression $t=S/(1+3S)^2$. As a curve in $F(y)$ and $y$ over $\qs(S)$, the
resulting equation has genus 0 again, and thus admits a rational
parametrization. Indeed, let $W$ be the unique series in $\qs[y][[t]]$
satisfying
$$
W= y\, \frac{1+SW+S(S+1)W^2}{1+3S}.
$$
Then
$$
F(y)={\frac { \left(1- S(1+S)W \right)  \left( 1+SW+S(S+1){W}^{2} \right) }{1-SW}}.
$$
 This expression shows that $F(y)$ has polynomial
coefficients in $y$, and we have proved that $F(y)$ is the \gf\ $G(y)$
of planar maps.

\medskip
\noindent
{\bf{Back to 3-colored planar maps.}}
Let us return to the functional
equation~(\ref{cat-planaires-q3}--\ref{cat-planaires-q3-bis}). 
It defines a unique 
series $M(y)\equiv M(t;y) \in \qs(y)[[t]]$. Moreover, the form of the
equation implies that the coefficient
of $t^n$ in this series has no pole at $y=1$. (For combinatorial
reasons, we know that the coefficients are
\emm polynomials, in $y$, but the form of the equation does not imply
such a strong statement).
Assume we find 4
series, $F(t;y)\equiv F(y) \in \qs[y][[t]]$, and $F_0, F_1, F_2\in
\qs[[t]]$, such that
$$
\Pol( F(y), F_0, F_1, F_2, t ;y)=0.
$$
By expanding this identity in the
neighborhood of $y=1$, it  follows that $F_0=F(1)$, $F_1=F'(1)$
and $F_2=F''(1)$ (the derivatives being taken with respect to
$y$). Consequently, the series $F(y)$ satisfies the same equation as
$M(y)$, and $F(y)=M(y)$. 

So our first task is to guess the values of $F^{(i)}(1)$, for $0\le i \le
2$. From the functional
equation~\eqref{eq:M} one can compute the first
coefficients of the series $M(y)=M(3,0;t,1;1,y)$. The first 40
coefficients of $M_0:=M(1)$ suffice to conjecture that this series
satisfies
\begin{multline*}  
-12500\,{t}^{6}M_0^{4}
-24\,{t}^{4} \left(71- 1000\,t \right) M_0^{3}
-2\,{t}^{2} \left( 39-1020\,t+7216\,{t}^{2}+3600\,{t}^{3}
\right)M_0^{2}
\\
-M_0\, \left(
1-42\,t+536\,{t}^{2}-1712\,{t}^{3}-9040\,{t}^{4}+864\,{t}^{5} \right) 
+1-40\,t+540\,{t}^{2}-2720\,{t}^{3}+432 \,{t}^{4}
=0.
\end{multline*}
The corresponding curve has genus $0$, and the parametrization by the
series $S$ given
in the theorem is constructed using {\sc Maple}. 

We now compute more coefficients of $M(y)$, in order to conjecture the values of $M'(1)$
and $M''(1)$. In the expansions of these two series, we
systematically replace the variable $t$ by its expression~\eqref{S-def-planar3} in
terms of $S$, as we suspect that  $M'(1)$
and $M''(1)$ will have a high degree over $\qs(t)$, but hopefully a
smaller degree over $\qs(S)$. And indeed, from the first 80 terms of
$M(y)$ (and of course with the help of {\sc Maple}), we conjecture that
$M'(1)$ is quadratic over $\qs(S)$:
$$
M'(1)=
{\frac { \left(1+ 2\,S \right)  
\left(P(S)+ Q(S)\sqrt { (1+2S)(1+2S+4S^2)
} \right)
}{ {S}^{2}\left( 1-2\,{S}^{3} \right) ^{4}}}
$$
with
$$
P(S)= 32\,{S}^{12}+32\,{S}^{11}+12\,{
S}^{10}-32\,{S}^{9}-16\,{S}^{8}+18\,{S}^{7}+14\,{S}^{6}-28\,{S}^{5}-58
\,{S}^{4}-49\,{S}^{3}-25\,{S}^{2}-7\,S-1
$$
and
$$
Q(S)=\left( 1-2\,{S}^{3} \right)  \left(1+2S+ 4\,{S}^{2} \right)
\left( 1+S \right) ^{3}.
$$
Hoping that $M''(1)$ lies in the same quadratic extension of $\qs(S)$
as $M'(1)$, we then look  for
a linear relation between $1$, $M'(1)$ and $M''(1)$ with polynomial
coefficients in $S$ (using the command {\tt hermite\_pade})
and obtain  the conjectured expression:
$$
M''(1)={\frac {2 \left(1+ 2\,S \right)  
\left(\bar P(S)+ \bar Q(S)\sqrt { (1+2S)(1+2S+4S^2)
} \right)
}{ {S}^{3}\left( 1-2\,{S}^{3} \right) ^{6}}}
$$
with
\begin{multline*}
\bar P(S)=
-5-54\,S-300\,{S}^{2}-1082\,{S}^{3}-2721\,{S}^{4}
-4768\,{S}^{5}-5310\,{S}^{6}-1944\,{S}^{7}+4970\,{S}^{8}+10468\,{S}^{9}\\
+8724\,{S}^{10}
+12\,{S}^{11}-8336\,{S}^{12}-10080\,{S}^{13}-6016\,{S}^{14}-1728\,{S}^{15}
+96\,{S}^{16}-64\,{S}^{17}-192\,{S}^{18}-192\,{S}^{19}
\end{multline*}
and
\begin{multline*}
\bar Q(S)=
\left( 1+S \right) ^{3} \left(1- 2\,{S}^{3} \right) 
\left(1+2\,S+ 4\,{S}^{2} \right) \times
\\
 \left(
 8{S}^{7}+8\,{S}^{6}+12\,{S}^{5}-20\,{S}^{4}-48\,{S}^{3}-42\,{S}^{2}-19\,S-5
 \right) .
\end{multline*}

Now return to~\eqref{cat-planaires-q3}. Consider the quartic
equation (in $F(y)$): 
\beq\label{eq-symb}
\Pol\left(F(y), F_0, F_1, F_2, \frac{S(1-2S^3)}{(1+2S)^3};y\right)=0,
\eeq
where $F_0$ (resp. $F_1$, $F_2$) is the \emm conjectured, value of
$M(1)$ (resp. $M'(1)$, $M''(1)$).
The rational function of $S$ that occurs is just the
expression~\eqref{S-def-planar3} of $t$ in  terms of $S$.
When $S=0$, this equation has degree 1 in $F(y)$. 
Hence~\eqref{eq-symb} admits a unique solution in $\qs(y)[[t]]$, denoted
$F(y)\equiv F(t;y)$. As argued above, if we can prove that this series has
coefficients in $\qs[y]$ (or that its coefficients have no pole at
$y=1$), we can conclude that $F(y)=M(y)$, and that the
conjectured values of $M(1)$, $M'(1)$ and $M''(1)$ are correct. 

With the help of the {\tt parametrization} function of {\sc Maple},
and of an extension of it provided by Mark van Hoeij,
we have discovered that the quartic equation~\eqref{eq-symb},
seen as a curve in $y$ and $F(y)$ over $\qs(S, \sqrt { (1+2S)(1+2S+4S^2)})$, admits
a rational parametrization which we now describe. Set $T=2S$, $\Delta=(1+T)(1+T+T^2)$, and
consider the following quartic equation in $W$: 
\beq\label{W-def-3col}
W
 =y\,\frac{P_1+P_2\,\sqrt{\Delta} }
{2(1-TW)(2-TW)(1+T)^2\left((1+T+T^2) (1-TW)+\sqrt \Delta\right)}
\eeq
where
$$
P_1= (1+T)(1+T+T^2) 
\left(4 -12\, T W
- \left( 9\,T+4 \right) T  ^{2}W^{3}
+3 \left( T+1 \right) T^{3}W^{4}
+2  \left( 7\,T+1 \right) TW^{2}\right)
$$
and
\begin{multline*}
P_2=
- \left( T+1 \right) ^{3}T^{3}W^{4}
+ \left( T+1 \right)  \left( T-4 \right) T^{2}W^{3}
+ \left( 5\,{T}^{3}+8\,T^{2}+14\,T+6 \right) TW^{2}\\
-2\,   \left( 3\,T^{2}+5\,T+6 \right) TW
+2\,T^{2}+4\,T+4.
\end{multline*}
Recall that at $t=0$, the series $S$ and $T$ vanish. This implies that~\eqref{W-def-3col} defines a unique series $W\in \qs[y][[t]]$. 
Now in~\eqref{eq-symb}, let us replace 
$S$ by $2T$  and  $y$ by 
its rational expression in terms of $T$, $\sqrt \Delta$ and $W$
derived from~\eqref{W-def-3col}. We
leave it to the reader's computer algebra system to check
that~\eqref{eq-symb} then \emm factors, into a  factor of degree 3 in
$F(y)$, and 
a linear one. Moreover,  setting $t=0$ (that is, $T=0$) shows that
the linear factor is the only one that has a solution $F(y)\in
\qs(y)[[t]]$. Solving it for $F(y)$ gives
$$
F(y)=\frac{P_3(T, \sqrt\Delta, W)}
{P_4(T, \sqrt\Delta, W)}
$$
for two polynomials $P_3$ and $P_4$ with coefficients in $\qs$, such
that $P_4(0,1,W) $ (which is the value taken by $P_4$ when $T=0$) lies
in $\qs$ and is non-zero. This shows  that $F(y)$
belongs to $\qs[y][[t]]$. This
completes our very long proof of the short equation~\eqref{M3-sol}.

A simple singularity analysis~\cite[Chap.~VII.7]{flajolet-sedgewick} of $M(3,0,t,1,1;1,1)$
yields the asymptotic behaviour of the number of $3$-colored planar
maps with $n$ edges. 
\end{proof}

 \subsection{Three-colored triangulations}
\label{sec:three-triang}

By Theorem~\ref{thm:alg-triang},  the series
$Q(3,\nu,t,w,z;x,y)$ is algebraic. Without loss of generality, we can
set $w=z=1$. We can also focus on near-triangulations (no digon
allowed) by considering $Q(y)\equiv Q(3, \nu, t,1,1;0,y)$. This series
counts three-colored near-triangulations by  edges (variable  
  $t$), monochromatic edges (variable $\nu$)  and outer degree
  (variable $y)$. It is algebraic over   $\qs(\nu,t,y )$. 

In what follows, we first  describe the construction of the equation
with one catalytic variable satisfied by $Q(y)$. Alas, it
involves four unknown series $Q^{(i)}(0)$, and we have not succeeded
in solving it for a generic value of $\nu$.  We solve it, however, for
$\nu=0$, thus counting \emm proper, colourings of near-triangulations. But
this result can probably be obtained by simpler means. Due to the duality
relation~\eqref{near-t--near-c}, another interesting case is $\nu=-2$,
for which our series actually counts proper 3-colorings of near-cubic maps. We have not
solved the equation in this case, but we state a conjecture for its
solution, due to Bruno Salvy (and obtained by computing many
coefficients of the solution).

\medskip
  We start from the invariant equation~\eqref{eq-inv-triang}, written
  for $q=3$ and 
  $m=6$. It involves 7 unknown series $C_0, \ldots, C_6$, which are independent
  of $y$. By expanding this equation in the neighborhood of $y=0$ and writing that the
  coefficient of $y^{-2r}$ is zero, we   obtain  explicit
  expressions   for the series $C_r$, as   described in
  Section~\ref{sec:Cr-triang}. Moreover, by  writing that the
  coefficient of $y^{-2r+1}$ is zero, for $1\le r \le 6$, we obtain
  additional identities 
  relating the series $Q^{(i)}(0)$. More precisely:
  \begin{itemize}
  \item [---] by extracting the coefficients of $y^{-12}$, $y^{-10}$,
    $y^{-8}$, we obtain:
    \begin{eqnarray*}
    C_6 &=&-27\nu^6  ,\\
    C_5 &=&
27\,{\nu}^{4} \left( \nu-1 \right)  \left( 2\,\nu-5 \right) 
 ,\\
    C_4 &=&
9/2\,{\nu}^{2} \left( \nu-1 \right)  \left( 18\,{t}^{3}{\nu}^{3}+35\,{
\nu}^{3}-75\,{\nu}^{2}+30\,\nu+10 \right);
  \end{eqnarray*}
\item [---]  extracting the coefficients of $y^{-11}$ and $y^{-9}$ does not yield
  new identities;
\item [---]  extracting the coefficient of $y^{-7}$ gives $Q(0)=1$,
  which is not a surprise. From now on we systematically replace every
  occurrence of $Q(0)$ by $1$;
\item [---]  extracting the coefficient of $y^{-6}$ gives an
  expression of $C_3$ as a polynomial in $\nu$, $t$ and $Q'(0)$;
\item [---]  extracting the coefficient of $y^{-5}$ gives the standard identity
  between $Q''(0)$ and $Q'(0)$:
\beq\label{Q1-Q2-3}
Q''(0)= \frac{2Q'(0)}{t\nu}.
\eeq
From now on we systematically replace every occurrence of $Q''(0)$ by
this expression;
\item [---]   extracting the coefficient of $y^{-4}$ gives an
  expression of $C_2$ as a polynomial in $\nu$, $t$, $Q'(0)$ and $Q^{(3)}(0)$;
 \item [---]  extracting the coefficient of $y^{-3}$ gives an
   expression of  $Q^{(4)}(0)$ in terms of  $\nu$, $t$, $Q^{(3)}(0)$ and $Q'(0)$:
\beq\label{Q4-Q13-3}
Q^{(4)}(0)=-24\left( 6+{\frac {1}{{t}^{3}{\nu}^{2}}} \right) 
Q'(0) 
+4\,{\frac { \left( 1+\nu \right) 
Q^{(3)}(0)
}{\nu\,t}}+24\,{\frac {\nu+2}{\nu\,t}}.
\eeq
\item [---]   extracting the coefficient of $y^{-2}$ gives an
  expression of $C_1$ as a polynomial in $\nu$, $t$, $Q'(0)$,
  $Q^{(3)}(0)$ and $Q^{(5)}(0)$; 
\item [---]  extracting the coefficient of $y^{-1}$ gives an
   expression of  $Q^{(6)}(0)$ in terms of  $Q^{(5)}(0)$, $Q^{(3)}(0)$ and $Q'(0)$;
 \item [---] finally,  extracting the coefficient of $y^{0}$ gives an
  expression of $C_0$ as a polynomial in $\nu$, $t$, $Q'(0)$,
  $Q^{(3)}(0)$, $Q^{(5)}(0)$  and $Q^{(7)}(0)$.
\end{itemize}

Let us now replace each $C_r$ by its expression in the invariant
equation~\eqref{eq-inv-triang}:  this gives, after dividing by
$54(\nu-1)^5t^3$,
an equation with one catalytic variable for $Q(y)$, 
of the form~\eqref{generic-Q},
involving  the series $Q'(0)$, $Q^{(3)}(0)$, $Q^{(5)}(0)$ and
$Q^{(7)}(0)$. 

This is enough to conclude that $Q(3, \nu, t,1,1;0,y)$ is algebraic, but
too big to be solved with the methods that are available at the
moment. However, the case $\nu=0$ comes out very easily. As explained
further down, this is, unfortunately, not very surprising.

\begin{Theorem}\label{thm:triang-3-nu0}
 The series $Q(3,0,t,1,1;0,y)$, which counts properly three-colored
 near-triangulations  by  edges (variable   $t$),  and outer degree
  (variable $y)$ is algebraic of degree $6$ over    $\qs(t,y )$. 

Let $Q_i(t):=[y^i]Q(3,0,t,1,1;0,y)$ be the series  that counts (by edges)
properly three-colored  near-triangulations of outer degree~$i$. Then
$Q_1(t)=0$ and for $i\ge 2$, each $Q_i(t)$ is
(at most) quadratic  over  $\qs(t )$ and admits a rational
parametrization.
Let   $S\equiv S(t)$ be  the unique series in $t$ satisfying
$$
S=t^3(1+2S)^2.
$$
Then $t^iQ_i(t)$ admits a rational expression in terms of $S$. In
particular,
\beq\label{Q-nu0-sol}
Q_2(t)=2t(1+S-S^2) \quad \hbox{and} \quad Q_3(t)= 2S(1-S).
\eeq
\end{Theorem}
\begin{proof}
Let us set $\nu=0$ in the equation with one catalytic variable
obtained by the above construction: this simply gives $Q'(0)=0$, which
is obvious because 
maps counted by this series have a loop. This also follows
from~\eqref{Q1-Q2-3}. 
 In order to obtain a non-trivial equation, we proceed as follows: 
 we first replace $Q^{(3)}(0)$ by its expression in terms of $Q'(0)$ and $Q^{(4)}(0)$ derived from~\eqref{Q4-Q13-3}, and then $Q'(0)$ by its expression in terms of $Q''(0)$ derived from~\eqref{Q1-Q2-3}. We finally divide the resulting equation by $\nu$, and set $\nu=0$.   
 This gives for properly 3-colored
near-triangulations a much simpler equation, with only one unknown
series $Q''(0)$:
\begin{multline}\label{1-cat-Q-nu0}
  4\,{y}^{5}{t}^{3}  Q ( y ) ^{3}
+{y}^{2}t \left({t}^{2} -10\,ty-8\,{y}^{2} \right)   Q ( y  ) ^{2}
+ \left( 6\,{t}^{2}{y}^{3}+2\,ty-2\,{t}^{2}+4\,{y}^{2} \right) Q ( y )
\\ -{t}^{3}{y}^{2}-2\,ty+2\,{t}^{2}-4\,{y}^{2
}+{y}^{2}{t}^{2}  Q''(0)=0 .
\end{multline}
This equation is easily solved, using either the quadratic method, or
the more general method of~\cite{mbm-jehanne}, and one obtains, with $Q_2=Q''(0)/2$:
$$
8{t}^{5}\,Q_2^{2}+ \left( 1-12\,{t}^{3}-8\,{t}^{6} \right) Q_2
-2\,t \left( 1-11\,{t}^{3}-{t}^{6} \right) =0,
$$
from which the first part of~\eqref{Q-nu0-sol} easily follows.

For $i\ge 3$, \eqref{1-cat-Q-nu0} allows to compute $t^2 Q_i(t)$
inductively as a  
polynomial in $t$ and the series $Q_j(t)$, for $2\le j<i$. 
\end{proof}

{\begin{figure}[t!]
\begin{center} 
\input{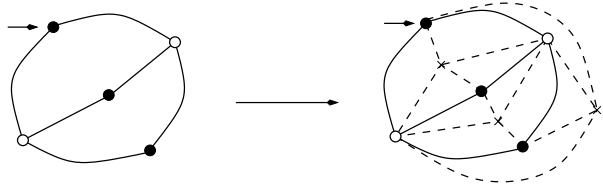}
\caption{A bipartite map and the corresponding Eulerian  triangulation.}
\label{fig:bip-triang} 
\end{center}
\end{figure}}

\noindent{\bf Remark.} A triangulation admits a proper 3-coloring if
and only if it is Eulerian, that is, if its faces are
2-colorable. 
The condition is necessary because
each face of a properly 3-colored triangulation contains, in clockwise
order, either 
the colors 1, 2, 3, or the colors 1, 3, 2, and two adjacent faces are
of different types. 
 That the condition is sufficient can be proved  by induction on the
 face number. 
Moreover, Eulerian triangulation admits exactly 6
proper colorings. 
But there is a standard bijection between bipartite
maps with $n$ edges and Eulerian triangulations with $3n$ edges, illustrated in
Figure~\ref{fig:bip-triang}. Hence counting 3-colorable triangulations
should not be harder than counting bipartite maps, which, as recalled
in Section~\ref{sec:example}, can be done with a single catalytic
variable by simply removing an edge  (see~\eqref{M-eq-y-20}). Even
though the 3-colorable \emm near,-triangulations 
considered here are a bit more general, it is not very surprising to
find, for their enumeration, an equation with one catalytic variable and
only one unknown series (see~\eqref{1-cat-Q-nu0}).


A more exciting perspective is to obtain the \gf\ of properly
3-colored \emm cubic, maps. Indeed, according
to~\eqref{near-t--near-c}, 
$$
 Q_i(3,-2,t) :=[y ^i] Q(3, -2, t, 1,1;0,1)=
(-t)^{-i} \sum _{G \, \rm{ near-cubic} \atop \dv(G)=i} 
 (-3t^3)^{\ff(G)-1}
\Ppol_G\left(3, 0\right).
$$
Thus we have access in particular to the series
$$
C(z):= \sum _{G \, \rm{ near-cubic} \atop \dv(G)=1} 
 z^{\ff(G)}\Ppol_G\left(3, 0\right) =
 4\,{z}^{3}+84\,{z}^{4}+1872\,{z}^{5}
+46464\,{z}^{6}
+O \left( {z}^{7} \right) 
.
$$
Fig.~\ref{fig:cubiques} justifies the value of the first two
coefficients of this series.
By computing recursively many coefficients of $C(z)$, and
feeding {\sc Gfun} with them, Bruno Salvy has reached the following rather
formidable conjecture. 
How likely is it to hold? The  equation below involves ``only'' 87
non-zero coefficients, while it holds at least up to order
$O(z^{169})$. It holds significantly further modulo $p$ for numerous
values of $p$, and so we believe it to be true.

Note that, in contrast with all solutions obtained so far, the genus of the
corresponding curve
is not 0, but 1. So we cannot hope for a rational parametrization.
\begin{figure}[ht!]\begin{center} \input{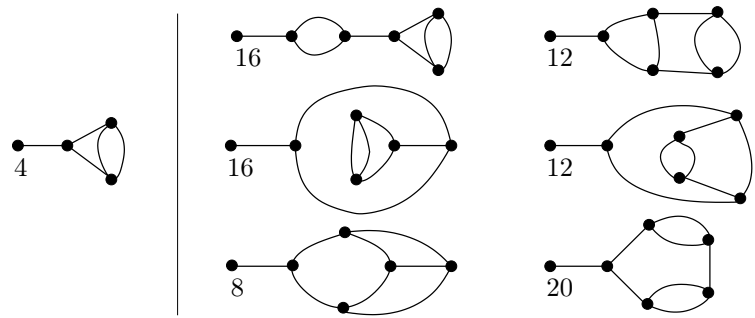}\caption{The loopless near-cubic maps with root-degree 1 and
  3 or 4 faces, and their number of proper 3-colorings (the color of the
  root is fixed).
In each column, the first two maps correspond to the same graph.}\label{fig:cubiques} \end{center}\end{figure}

\begin{Conjecture}
\label{conj}  
The \gf\ $C\equiv C(z)$ of properly $3$-colored near-cubic maps in which the degree of
the root-vertex is $1$, counted by faces, is algebraic of degree $11$
and satisfies:
\begin{multline*}
  922337203685477580800000\,{C}^{11}
+9007199254740992\,
 \left(194560000\,z -5971077 \right) {C}^{10}
\\
+4294967296\,
 \left( 
280335535308800\,{z}^{2}
-25398219177984\,z
+446991689475
 \right) {C}^{9}
\\
-1024\,
\left(
379991218559385600000\,{z}^{4}
-188284129271105978368\,{z}^{3}
+ 74426563120993402880\,{z}^{2}
\right.
\\
\left.
-3460024309515976704\,z
+60644726921050599 
\right) {C}^{8}
\\
-1024
\, \left( 
855256650185747464192\,{z}^{5}
+198557240861845880832\,{z}^{4}
+7030700057733103616\,{z}^{3}
\right.
\\
\left.
-2005025500677518336\,{z}^{2}
+65719379546147724\,z 
-1261082394855783
\right) {C}^{7}
\\
-64\,
 \left(
13794761675403801133056\,{z}^{6}
+1749420037224685109248\,{z}^{5}
-278771160986127695872\,{z}^{4}\right.
\\
\left.
+3443220359730862080\,{z}^{3}
+294527021649617744\,{z}^{2}
- 12400864344288084\,z
+586081179814293
\right) {C}^{6}
\\
-16\, \left(
32829338688610212249600\,{z}^{7}
-541704013946292273152\,{z}^{6}
-549137038895633924096\,{z}^{5} 
\right.
\\
\left.
+41876669882140680192\,{z}^{4}
-936289577498747840\,{z}^{3}
\right.
\\
\left.
+12987916499676352\,{z}^{2}
+208517314053540\,z
-54447680943015
 \right) {C}^{5}
\\
-32\, \left( 
124515522497539473408\,{z}^{9}
+6242274275823592669184\,{z}^{8}
-898808183791057633280\,{z}^{7}
\right.
\\
\left.
-5275329284641325056\,{z}^{6}
+6539785066149118976\,{z}^{5}
-361493662811609868\,{z}^{4}
\right.
\\
\left.
+9979948894517522\,{z}^{3} 
-432679480767965\,{z}^{2}
+6248694091833\,z
+378858660750
\right) {C}^{4}
\\
-8\, \left( 
747093134985236840448\,{z}^{10}
+5932367633073989222400\,{z}^{9}
-1529736206124490686464\,{z}^{8}
\right.
\\
\left.
+132585839072566050816\,{z}^{7}
-3048630269218258944\,{z}^{6}
-135087570198766176\,{z}^{5}
\right.
\\
\left.
+5706147748413032\,{z}^{4}
-229584590608200\,{z}^{3}
+23755821897083\,{z}^{2}
-152875558308\,z
-27738626328
 \right) {C}^{3}
\\
+ \left( 
-3361919107433565782016\,{z}^{11}
-6012198464670331305984\,{z}^{10}
+2332964327872863928320\,{z}^{9}
\right.
\\
\left.
-341248528343609901056\,{z}^{8}
+24933054438553903104\,{z}^{7}
-994662704339242816\,{z}^{6}
\right.
\\
\left.
+33270083406272816\,{z}^{5}
-1608971168541300\,{z}^{4}
+7467003627448\,{z}^{3}
\right.
\\
\left.
+5037279798640\,{z}^{2}
-194388001728\,z
+808501760
 \right) {C}^{2}
\\
+ z\,\left( 
-840479776858391445504\,{z}^{11}
-157618519659107057664\,{z}^{10}
+157170928122096254976\,{z}^{9}
\right.
\\
\left.
-34691457904249143296\,{z}^{8}
+3785139252232855552\,{z}^{7}
-224694559056638912\,{z}^{6}
\right.
\\
\left.
+6999136302319904\,{z}^{5}
-197576502742812\,{z}^{4}
+19551640345287\,{z}^{3}
\right.
\\
\left.
-1347626230088\,{z}^{2} 
+40099744688\,z
-404250880
\right) C
\\
-4\,{z}^{4} \left( 
19698744770118549504\,{z}^{9}
-8025289374453202944\,{z}^{8}
+1366977099830657024\,{z}^{7}
\right.
\\
\left.
-120213529404735488\,{z}^{6}
+5234026490678784\,{z}^{5}
-86995002866345\,{z}^{4} 
\right.
\\
\left.
+4680668094111\,{z}^{3}
-691486996440\,{z}^{2}
+31610476208\,z
-404250880
\right) 
=0.
\end{multline*}
\end{Conjecture}

\section{Non-separable maps}
\label{sec:non-sep}

  A map is 
 \emph{separable} if it is the atomic map $m_0$ (one vertex, no edge)
 or can be obtained by gluing two  non-atomic maps  at a vertex (more
 precisely, a corner of the first map is glued to a corner of the
 second map). 
Observe that both
 maps with one edge are non-separable. 

Several authors have addressed the enumeration of families of colored 
non-separable planar maps. For instance, the series $T(x,y)$
 defined by~\eqref{eq-Tutte} and studied by Tutte in his long series of
 papers counts  non-separable \emm near-triangulations,  (all
 internal faces have degree 3). Also, Liu wrote a
 functional equation for the \gf\ of non-separable planar maps weighted
by their Tutte polynomial~\cite{liu-non-sep}, which was further
studied by Baxter~\cite{baxter-dichromatic}.

In this section, we first prove that the latter problem is equivalent to
the enumeration of general planar maps (weighted, of course, by their Tutte
polynomial).
In particular, the algebraicity result of Theorem~\ref{thm:alg-planaires} 
translates into an algebraicity result for colored non-separable  maps.
Then, we show how Tutte's equation~\eqref{eq-Tutte} can be recovered from
our equation~\eqref{eq:Q} obtained for quasi-triangulations.

\subsection{From general  to non-separable planar maps}
Let $\mN$ be the set of non-separable planar maps and 
let $\gN(q,\nu,t,w,z;x,y)\equiv \gN(x,y)$ be the associated Potts
generating function: 
$$
 \gN(x,y)
=\frac{1}{q}\sum_{N\in\mN}t^{\ee(N)}w^{\vv(N)-1}z^{\ff(N)-1}
x^{\dv(N)}y^{\df(N)}\Ppol_N(q,\nu).
$$
The following proposition relates 
$\gN(x,y)$ to the Potts \gf\  of general planar maps,
denoted by   $\gM(x,y)$ and defined by~\eqref{potts-planar-def}.

\begin{Proposition}\label{prop:MN}
 The series  $\gM$  and  $\gN$ are related by:
$$
\gM(x,y)
=1+\frac{\gM(1,1)\gM(x,y)}{\gM(x,1)\gM(1,y)}\, 
\gN\left(q,\nu,t\gM(1,1)^2,w,z;x\frac{\gM(x,1)}{\gM(1,1)},y\frac{\gM(1,y)}{\gM(1,1)}\right)
$$
where $ \gM(x,y)\equiv\gM(q, \nu, t,w,z;x,y)$.
\end{Proposition}
\noindent

\begin{proof}
A non-atomic map decomposes into a non-separable map  (the \emph{core})
containing the root-edge, in the corners of which are attached
other rooted maps.
This decomposition is illustrated in Figure
\ref{fig:non-separable-core}. It induces a bijection between non-atomic maps
and pairs consisting  of a non-separable map $N$ (the core) and a ordered
sequence of  $2\ee(N)$ maps $M_1,\ldots,M_{2\ee(N)}$ (since $2\ee(N)$
is the number of corners of $N$). 

\begin{figure}[ht!]\begin{center} \input{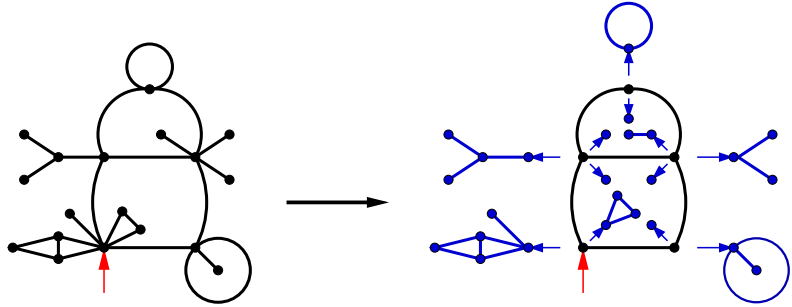}\caption{Decomposition of a map into a non-separable
  core in the corners of which  are attached other maps.}\label{fig:non-separable-core} \end{center}\end{figure}

Let $M$ be a non-atomic map and let $(N;M_1,M_2,\ldots,M_{2\ee(N)})$ be
its image by the decomposition.  One clearly has 
$$
\ee(M)= \ee(N) + \sum_{i=1}^{2\ee(N)}\ee(M_i),
$$
$$
\vv(M)=\vv(N)+\sum_{i=1}^{2\ee(N)}(\vv(M_i)-1),\quad \quad 
\ff(M)=\ff(N)+\sum_{i=1}^{2\ee(N)}(\ff(M_i)-1),
$$
and by \eqref{eq:Potts-1components} the Potts polynomial of  $M$ is 
$$
\Ppol_M(q,\nu)=\Ppol_N(q,\nu)\prod_{i=1}^{2\ee(N)}
\frac{\Ppol_{M_i}(q,\nu)}{q}.
$$
Moreover, exactly $\dv(N)$ of  the maps $ M_i$
contribute to the degree $\dv(M)$ of the root-vertex of $M$.
Similarly, $\df(N)$ of these maps contribute to the degree $\df(M)$ of
the root-face of $M$.  
Finally,  exactly one of these maps  contributes to both $\dv(M)$ and $\df(M)$. 
These observations imply that  $\gM(x,y)$ satisfies 
\begin{multline*}
\gM(x,y)= 1 + \frac 1 q \sum_{N \in \mN} t^{\ee(N)}w^{\vv(N)-1}z^{\ff(N)-1}x^{\dv(N)}y^{\df(N)}\Ppol_N(q,\nu)
\\
\times
\gM(x,y) \gM(x,1)^{\dv(N)-1}\gM(1,y)^{\df(N)-1}\gM(1,1)^{2\ee(N)-\dv(N)-\df(N)+1}
 \end{multline*}
which yields the equation of the proposition.
\end{proof}

\begin{Corollary}\label{coro:non-sep}
Let $q\not=0,4$ be of the form $2+2\cos j \pi/m$, with $j, m \in \zs$. Then the Potts \gf\ of non-separable planar maps,
 $\gN(q,\nu, t,w,z;x,y)$, is algebraic over $\qs(q,\nu,t,w,z;x,y)$.
 \end{Corollary}
 \begin{proof}
Let  $s, u, v$ denote three indeterminates. Consider the following
 system:
$$
 T=s\, {\gM(T;1,1)}^{-2}, 
$$
$$
 X =\displaystyle 
u\,\frac{\gM(T;1,1)}{\gM(T;X,1)}, \quad \quad
Y=\displaystyle 
v\, \frac{\gM(T;1,1)}{\gM(T;1,Y)}, 
$$
where $\gM(t;x,y)$ stands for $\gM(q,\nu, t, w,z;x,y)$.
Recall that  $\gM(t;x,y)$ is a series in $t$ with coefficients in
$\qs[q,\nu,w,z,x,y]$, satisfying $\gM(t;x,y)=1+O(t)$. This implies
that the first equation defines $T$ uniquely as a series in
$s$ with coefficients in $\qs[q,\nu, w,z]$.  Moreover,
$T=s+O(s^2)$. Finally, the 
algebraicity of $\gM$ (Theorem~\ref{thm:alg-planaires}) implies that
of $T$. Indeed, if 
$P(t,\gM(t;1,1)^2)=0$ for some non-trivial polynomial $P$ (with
coefficients in $\qs(q,\nu,w,z)$), then $P(T, s/T)=0$ and
$P(t,s/t)$ is not trivially $0$. 
Similarly, the second and third equations above  respectively define $X$ and $Y$ as algebraic power series in $s$
with coefficients in
$\qs(q,\nu,w,z, u)$ (resp. $\qs(q,\nu,w,z, v)$). By  Proposition~\ref{prop:MN},
\beq\label{eq:MN}
\gN(q,\nu,s,w,z;u,v)= 
\frac {\gM(T;X,1) \gM(T;1,Y)}{\gM(T;1,1)\gM (T;X,Y)}
 \left(  \gM(T;X,Y)-1\right).
\eeq
Given that each of the series $\gM$, $T$, $X$ and $Y$, is algebraic,
 $\gN(q,\nu,s,w,z;u,v)$ is algebraic over $\qs(q,\nu,s,w,z;u,v)$.
\end{proof}

The connection between the series $\gM$ and $\gN$ can be used to
convert the functional equation \eqref{eq:M} into a functional
equation for $\gN$. 
\begin{Corollary}
 The Potts \gf\ $N(q,\nu,s,w,z;u,v)\equiv N(u,v)$ of non-separable
 planar maps satisfies:    
\begin{multline*}
  \gN(u,v)= (q\!+\!\nu\!-\!1)\,swu{v}^{2}+\nu\,sz{u}^{2}v\\
\!\!\!\!\!\!+uvzs\frac{\gN(u,v)-v\gN(u,1)}
{v-1-\gN(1,v)+v\gN(1,1)}
+(\nu\!-\!1)\, uvws\frac{\gN(u,v)-u\gN(1,v)}
{u-1-\gN(u,1)+u\gN(1,1)}. 
\end{multline*}
\end{Corollary}

\begin{proof}
By specializing the equation of Proposition~\ref{prop:MN} to $x=1$ and/or
$y=1$, one obtains: 
\beq\label{Mx11y}
\gM(x,1)=1+\gN(S;U,1),\quad \quad \gM(1,y)=1+\gN(S;1,V),\quad \quad 
\gM(1,1)=1+\gN(S;1,1)
\eeq
with $N(s;u,v)\equiv N(q,\nu, s, w,z;u,v)$ and
\beq\label{eq:SUV}
S=t\, M(1,1)^2, \quad U= x\, \frac{M(x,1)}{M(1,1)} \quad \hbox{ and }\quad  V=
y\, \frac{M(1,y)}{M(1,1)}.
\eeq
This allows us to express $\gM(x,y)$  in terms of specializations of $\gN$:
$$
\gM(x,y)=\left(1-\frac{1+\gN(S;1,1)}{(1+\gN(S;U,1))(1+\gN(S;1,V))}
\gN(S;U,V)\right)^{-1},
$$
Now, in the functional equation \eqref{eq:M} defining $M(x,y)$,   we replace  the
indeterminates $t$, $x$, $y$ by rational expressions  of $S$, $U$, $V$
and specializations of $M$,  using~\eqref{eq:SUV}. Then, we use \eqref{Mx11y} and the above
equation to express all occurrences  of $M$ in terms of $N$.
This gives the equation of the corollary, at $(s,u,v)=(S,U,V)$. Given
that $t, x$ and $y$ can be recovered from $S$, $U$ and $V$
(using~\eqref{eq:SUV} and~\eqref{Mx11y})), this equation must hold at
a  \emm generic, point $(s,u,v)$.  
\end{proof}

Given the relation \eqref{eq:Tutte=Potts} between the  Potts and Tutte 
polynomials, one can translate the equation for $N$ 
into an equation for the series 
$$
   \gN\left((\mu-1)(\nu-1),\nu,t,\frac{w}{\nu-1},z;x,y\right)
=\sum_{N\in\mN} t^{\ee(N)}w^{\vv(N)-1}z^{\ff(N)-1}x^{\dv(N)}y^{\df(N)}
\Tpol_N(\mu,\nu).
$$
One thus recovers the equation of~\cite[Thm.~4.2]{Liu:dichromatic-sums},  
obtained by a recursive approach.

\subsection{Properly colored  non-separable near-triangulations}
\label{sec:non-sep-triang}

Let us return to the equation~\eqref{eq-Tutte} on which Tutte
worked for more than 10 years. The series $T(x,y)$ defined by
this equation  is
\beq\label{T-def}
T(x,y)= \sum_{T\in \mT} z^{\ff(T)-1}x^{\dv(T)} y^{\df(T)} \Ppol_T(q,0),
\eeq
where the sum runs over all non-separable \emm near-triangulations, (maps
in which all internal faces have degree 3).
We now explain how~\eqref{eq-Tutte} can be recovered from the
functional  equation~\eqref{eq:Q} defining $Q(x,y)$.

Recall that $\mQ$ is the set of \emm quasi-triangulations,, that is,
rooted maps such that every
internal face is either a digon incident to the root-vertex or a
triangle. Let us say that a face of a map is  \emm simple, if it is
not incident twice to the same vertex. Then a map is non-separable if
and only if each of its faces is simple.
Define the following subsets of $\mQ$:
\begin{enumerate}
\item[--]  $\mR$ consists of non-atomic maps of $\mQ$ whose  root-face
  is simple, 
\item[--] $\mS$ is the set of non-separable maps in $\mQ$. 
\end{enumerate}
We denote by $\gR(x,y)\equiv\gR(q,\nu,t,w,z;x,y)$ and  $\gS(x,y)\equiv\gS(q,\nu,t,w,z;x,y)$ the
 corresponding Potts \gfs, defined by analogy
 with~\eqref{Q-ser-def} (in particular, with the factor $2^{\ddig(\cdot)}$). Clearly,
$$
\mT \subsetneq \mS  \subsetneq \mR  \subsetneq \mQ.
$$ 
In what follows, we first establish a connection between the series
 $\gR(x,y)$ and $\gQ(x,y)$ (Lemma~\ref{lem:QR}), from which we derive a
 functional equation satisfied by $\gR$. Then, we explain that, even
 though the sets $\mS$ and $\mR$ are distinct, \emm the series $\gS$ and
 $\gR$ coincide when $\nu=0$,, that is, when one counts proper
 colorings (Lemma~\ref{lem:RS}).  Finally, we find a simple connection
 between the series $\gS$ and $\gT$  (Lemma~\ref{lem:ST}), from which
 Tutte's equation~\eqref{eq-Tutte} can be derived.

 \begin{Lemma}\label{lem:QR}
The series $\gQ$ and $\gR$ are related as follows:
$$
\gQ(x,y)=1+\frac{\gQ(x,y)}{\gQ(0,y)}\, \gR\left(x,y\gQ(0,y)\right).
$$
 \end{Lemma}
 \begin{proof}
  Any  non-atomic map in $\mQ$ decomposes into a map $R$
in $\mR$ containing the root-edge, together with some rooted maps in
$\mQ$ attached to 
the  corners of the root-face of $R$. This
decomposition is illustrated in Figure~\ref{fig:R-core}.  It induces a
bijection between non-atomic maps in $\mQ$ and pairs made of a
non-separable map $R$ in $\mR$ and a ordered sequence
$Q_1,\ldots,Q_{\df(R)}$ of  
maps in $\mQ$ such that $Q_2,\ldots,Q_{\df(R)}$
have no internal digons 
(by convention, $Q_1$ is the map attached to the root corner of $R$). 
This bijection translates into the equation of the lemma.
\end{proof}

\begin{figure}[ht!]\begin{center} \input{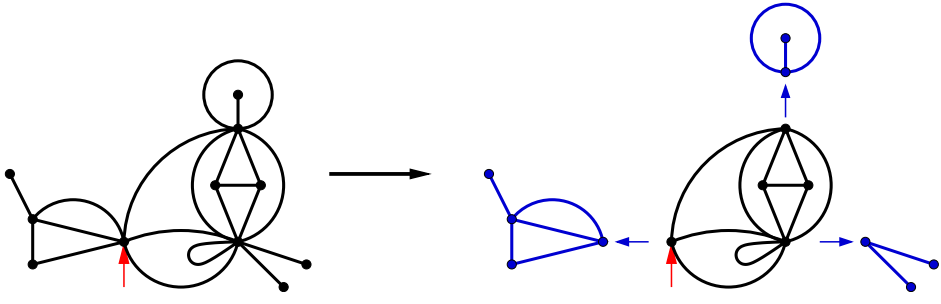}\caption{Decomposition of a map in $\mQ$ 
as a map in $\mR$ (containing the
root-edge) to which are attached some maps in~$\mQ$.}\label{fig:R-core} \end{center}\end{figure} 

One can combine  this result with the functional
equation~\eqref{eq:Q} defining   $\gQ(x,y)$ to obtain  a
functional equation for the series $\gR(x,y)$. It suffices to express
the ingredients of~\eqref{eq:Q}, namely, $\gQ(x,y)$, $\gQ(0,y)$,
$Q_1(x)$, $Q_2(y)$ and $y$, in terms of $Y:=y\gQ(0,y)$ and of 
specializations of $\gR$. 
Setting $x=0$ in Lemma~\ref{lem:QR}  
gives $\gQ(0,y)=1+\gR(0,Y)$. Thus we can now express $\gQ(x,y)$,
$\gQ(0,y)$ and $y$ in terms of $Y$ and specializations of $\gR$. 
Moreover, expanding the equation of Lemma~\ref{lem:QR}   around
$y=0$, and using the obvious relations $\gQ(0,0)=1$,  $\gQ(x,0)=1$, 
$\gR(x,0)=0$, gives
$$
\gQ_1(x)= \gR_1(x) \quad \hbox{and
} \quad \gQ_2(x)=\gR_2(x)+\gR_1(x)^2.
$$
We now replace in~\eqref{eq:Q} the terms $\gQ(x,y), \gQ(0,y), Q_1(x),
Q_2(y)$ and $y$ by their expressions in terms of $Y$ and $\gR$. 
This gives the following equation for $\gR$:
\begin{eqnarray}
  R(x,Y)&=&
{Y}^{2}wt ( 1-x\nu zt ) q+ ( \nu-1 ) wt{Y}^{2}
+xzt ( 1+\nu-x\nu zt ) R ( x,Y ) 
+xztY ( 1-x\nu zt ) R_1 ( x ) 
\nonumber \\
&+&zt ( 1-x\nu zt ) R_1 ( x ) R ( x,Y ) 
+ ( \nu-1 ) Yzt ( 1-x\nu zt )  
\left( 2\,xR_1 ( x ) +R_2 ( x ) +  R_1 ( x ) ^{2} \right) 
%
\nonumber\\
&+&zt ( 1-x\nu zt )  ( 1+R ( 0,Y )  ) {\frac {  R ( x,Y ) -YR_1 ( x )   }{Y}}
+ ( \nu-1 ) wtY {\frac {  R ( x,Y ) -R ( 0,Y )  }{x}}.\nonumber
\end{eqnarray}
Given that $Y=y\gQ(x,y)$, this equation actually holds for
a \emm {generic}, value of $Y$.

\begin{Lemma}\label{lem:RS}
The following identity holds:
$$
\gS(q,0,t,w,z;x,y)=\gR(q,0,t,w,z;x,y).
$$
Moreover, this series reads
$$
\gS(q,0,t,w,z;x,y)=\sum_{S\in  \mS}t^{\ee(S)}w^{\vv(S)-1}z^{\ff(S)-1}x^{\dig(S)}y^{\df(S)}\Ppol_S(q,0). 
$$
\end{Lemma}
\begin{proof}
 Any map $R$ in $\mR\setminus \mS$ has a  non-simple \emm internal,
 face. This face has degree 2 or 3. Thus $R$ has a loop and
 $\Ppol_R(q,0)=0$. This proves the first identity.

Now a map of $\mS$, being non-separable, cannot contain digons that
are doubly-incident to the root. 
Thus the term $2^{\ddig(S)}$ can be removed from the description of
$\gS$ (for any value of $\nu$).
\end{proof}

\begin{Lemma}
\label{lem:ST}
The series $\gS(x,y)$ and $\gT(x,y)$ are related by
$$
\gS(q, 0, 1,1,z;x,y)= \frac 1 q\, \gT\left(\frac{1}{1-xz},y\right). 
$$
\end{Lemma}
\begin{proof}
The set $\mT$ of non-separable near-triangulations coincides with the
set of maps in $\mS$ having no internal digon. Recall that maps of
$\mS$ have no digon doubly incident to the root. Thus, one obtains a map
in $\mT$ by taking a map in  
$\mS$ and \emph{closing} all internal digons
(that is, by identifying the two edges incident to the
digon). Conversely, any map in $\mS$ is obtained from a map in $\mT$
by \emph{opening} each of the edges incident to the root vertex into a
sequence of parallel edges $e_1,\ldots,e_k$ for $k\geq 1$, such that
$e_i$ and $e_{i+1}$ are edges incident to a common digon for all
$i=1,\ldots,k-1$.  The chromatic polynomial is unchanged by
the opening and closing of digons. In view of~\eqref{T-def} and
Lemma~\ref{lem:RS}, this gives the equation of the
lemma (the factor $q$ comes from the fact that Tutte's series $\gT$
weights maps by $\Ppol(q, 0)$ rather than $\Ppol(q, 0)/q$).
\end{proof}

We can now recover Tutte's equation for $T(x,y)$. Set $\nu=0$ and $t=w=1$
in the equation found for $R$ above. 
As maps of outer degree 1 have a loop,
$\R_1(x)=0$ when $\nu=0$. By Lemma~\ref{lem:RS}, we can safely replace $R$ by
$S$.  Finally, let us replace  $x$ by
$(1-1/x)/z$. By Lemma~\ref{lem:ST}, $S(x,y)$ becomes $T(x,y)/q$, $S(0,y)$ becomes $T(1,y)/q$,  and
$S_2(x)$ becomes $\gT_2(x)/q$.  This gives Tutte's
equation~\eqref{eq-Tutte}. As the saying goes, \emm la boucle est boucl\'ee.,


\bigskip
\noindent
{\bf Acknowledgements.} We are grateful to 
Mark van Hoeij for his help with the
{\tt parametrization} function of {\sc Maple}, and to Bruno Salvy,
whose efforts led to Conjecture~\ref{conj}.  We also thank Bertrand
Eynard for his patience in explaining us the methods of~\cite{eynard-bonnet-potts,Eynard-Kristjansen:On-model1}.


\bibliographystyle{plain}
\bibliography{coloured.bib}

\end{document}